\documentclass[10pt]{article}
 \usepackage{style}

\usepackage{mhequ}

\usepackage[utf8]{inputenc}
\usepackage[T1]{fontenc}
\usepackage{amsmath, amssymb, amsthm}
\usepackage{mathtools} % For \coloneqq, etc.
\usepackage{geometry}
\usepackage{bm} % While \boldsymbol is often preferred, \bm is versatile
\usepackage{enumerate}
\usepackage{graphicx}
\usepackage{booktabs} % For professional tables
\usepackage{verbatim} % For comment environment
\usepackage[hidelinks]{hyperref} % For clickable links and references
\usepackage{authblk}

\usepackage{bbm}

\numberwithin{equation}{section} % Number equations by section

\newtheorem{theorem}{Theorem}[section]
\newtheorem{corollary}[theorem]{Corollary}
\newtheorem{proposition}[theorem]{Proposition}
\newtheorem{lemma}[theorem]{Lemma}

\theoremstyle{definition}

\newtheorem{assumption}{Assumption}
\theoremstyle{remark}
\newtheorem{remark}{Remark}[section]

\title{\bf Stochastic Scaling Limits and Synchronization \\by Noise in Deep Transformer Models}
\author[$1$]{Andrea Agazzi}
\author[$1$]{Giuseppe Bruno}
\author[$2,1$]{Eloy Mosig Garc\'ia}
\author[$2,1$]{Samuele Saviozzi}
\author[$2$]{Marco Romito}
\affil[$1$]{Department of Mathematics and Statistics, University of Bern (CH)}
\affil[$2$]{Department of Mathematics, University of Pisa (IT)}
\date{\today}

\begin{document}

\maketitle

\begin{abstract}
We prove pathwise convergence of the layerwise evolution of tokens in a finite-depth, finite-width transformer model with MultiLayer Perceptron (MLP) blocks to a continuous-time stochastic interacting particle system. We also identify the stochastic partial differential equation describing the evolution of the tokens' distribution in this limit and prove propagation of chaos when the number of such tokens is large.
    The bounds we establish are quantitative and  the limits we consider commute. We further prove that the limiting stochastic model displays synchronization by noise, and establish exponential dissipation of the interaction energy on average, provided that the common noise is sufficiently coercive relative to the deterministic self-attention drift. We finally characterize the activation functions satisfying the former condition.
\end{abstract}

\section{Introduction}
\label{sec:intro}

The transformer architecture \cite{vaswani2017attention}, which underlies present-day Large Language Models, has been one of the main drivers of recent advances in machine learning and artificial intelligence. 
At each layer, the hidden state of the network is updated by sequentially applying two distinct operations: \emph{attention} modules \cite{bahdanau2014neural}, which capture long-range interactions in the input sequence, and classical MultiLayer Perceptrons (MLPs), acting separately on each element of that sequence.
Despite their empirical success, the mechanisms governing information propagation through depth, and the way attention and MLP blocks jointly shape internal representations, remain only partially understood from a theoretical viewpoint.

Recent progress has come from viewing transformers in suitable scaling limits as deterministic mean-field interacting particle systems modeling the evolution of $N$ tokens\footnote{Units of information fed to the network, e.g., words in a sentence, henceforth interpreted as particles.} through the layers of the neural network architecture (the so-called \emph{residual stream dynamics}), see, among others, \cite{sander2022sinkformers,geshkovski2023emergence,geshkovski2023mathematical,rigollet2025mean}. In these descriptions, depth plays the role of a continuous time variable, and, in the large-context regime ($N\to\infty$), the evolution of token representations is encoded by a PDE for their empirical distribution. This viewpoint is closely connected to the literature on \emph{scaling laws}, where the effect of various scaling exponents controlling the relative size of the network's hyperparameters (e.g., depth, width, context length) on the effective dynamics of the model is studied asymptotically, a regime naturally motivated by the size of modern architectures.
 This approach has also identified clustering of the particles as a mechanism through which deep transformer layers may organize synthetic internal representations of the input data \cite{geshkovski2023mathematical}.

However, this line of work mainly focuses on the attention mechanism and leaves aside the second fundamental layer-wise component of the architecture, namely the MLP block. As a consequence, a quantitative mathematical description of the interplay between self-attention and MLP layers in the residual stream under generic assumptions on the weights of these components is still lacking.

In this paper, we close this gap by proving quantitative convergence estimates of the residual stream dynamics for models including both attention and MLP layers, at initialization, to a stochastic limiting mean-field interacting particle system, in the joint deep and wide limit. 
We then identify the corresponding Eulerian dynamics as a stochastic partial differential equation for the token distribution, and again provide quantitative convergence estimates in the large-$N$ limit. This yields one of the first stochastic scaling limits for transformer residual dynamics (along with the concurrent papers \cite{koubbi2026homogenized,fedorov2026clustering}) and, to our knowledge, the first such limit in which the MLP component is included in the analysis. We finally use the limiting equation to analyze clustering, identifying explicit conditions under which the common noise induces synchronization. More specifically, we provide exponential dissipation estimates of a certain interaction energy whose global minimizers are clustered states. As the dissipation rate in these estimates naturally decomposes in two terms, respectively resulting from the MLP and the attention module, our results clarify quantitatively the interplay between these components of the architecture.

\subsection{Setup and notation} 

\subsubsection{The architecture} We consider a system of $N$ tokens, denoted by $(X_0^i)_{i=1}^N$, each embedded on the $d$-dimensional sphere $\mathbb S^{d-1}$ for $d \in \mathbb N$ , $d\ge 2$ fixed\footnote{The analysis of this paper is restricted to the architecture \emph{after} the embedding step, which we consider as a map from a dictionary (including positional encoding) to $\mathbb S^{d-1}$.}.
The evolution of these tokens through a transformer of depth $L \in \mathbb N$ is described by the following update rule for layer $\ell=0, \dots, L-1$:
\begin{equ}
\label{eq:transformer-update}
\begin{cases}
    Y^{i}_{\ell} &= \LN\left(X^{i}_\ell + 
    \alpha_1(L)
    \Attn_{\vX_\ell}(X^{i}_{\ell}) \right)\\
    X^{i}_{\ell+1} & = \LN\left(Y^{i}_{\ell} + 
    \alpha_2(L)
    G^m_{\ell+1}(Y^{i}_{\ell})\right) 
\end{cases}.
\end{equ}
Here, $\vX_\ell \coloneqq (X_\ell^1, \dots, X_\ell^N)$ denotes the set of tokens at layer $\ell$, $\LN(v) := v/|v|$ represents the normalization/radial projection on the unit sphere\footnote{We assume $v\neq0$. The probability of encountering the origin is zero under our initialization and noise assumptions.} (ensuring that $X_\ell^i, Y_\ell^i \in \mathbb S^{d-1}$ for all $\ell \in \{0,\dots L-1\}$, $i \in [N] := \{1,\dots N\}$, as in \cite{zhang2019root}), and 
\begin{equ}\label{e:alphas}
    \alpha_1(L)= \frac T {L}, \qquad  \alpha_2(L) = \sqrt{\frac T {L}}
\end{equ}
are scaling factors controlling the relative magnitude of the following two core operations:\begin{enumerate}[(i)]
   \item \textit{Self-Attention Mechanism:} The term $\Attn_{\vX_\ell}(X^i_\ell)$ computes the interaction of token $i$ with all other tokens in $\mathbf X_\ell$. A standard definition is \emph{softmax attention}:
    \begin{equ}
        \Attn_{\vX_\ell}(X^i_\ell) \coloneqq \sum_{j=1}^N \frac{\exp( \langle QX_\ell^i , KX_\ell^j\rangle)}{\sum_{k=1}^N \exp( \langle QX_\ell^i , KX_\ell^k\rangle)} (VX_\ell^j)
    \end{equ}
    where $Q,K,V \in \mathbb R^{d \times d}$ are parameter matrices, 
    and $\langle \,\cdot\,,\,\cdot\,\rangle $ denotes the Euclidean inner product in $\mathbb R^d$.
    \item \textit{Feed-Forward MultiLayer Perceptron (MLP):} The function $G^m_\ell: \R^d \to \R^d$ is a single-hidden layer perceptron of width $m \in \mathbb N$, defined as 
    \begin{equ}\label{e:mlp}
        G^m_\ell
        (Y_\ell^i)
        :=\frac{1}{\sqrt{m}}W^\ell\act\left(\frac{1}{\sqrt{d}}U^\ell 
        Y_\ell^i
        + \biasU{\ell}\right) + \biasW{\ell}.
    \end{equ}
    where $\theta_\ell = (U^\ell, W^\ell, \biasU{\ell},\biasW{\ell}) \in \R^{m\times d}\times \R^{d\times m} \times \R^m \times \R^d$ denotes the set of parameters at layer $\ell$ and  $\act~:~\mathbb R \to \mathbb R$ is a Lipschitz-continuous nonlinearity (or \emph{activation function}) applied element-wise to its input.
\end{enumerate}
Jointly, for a choice of parameters $ Q, K, V, \{\theta_\ell\}_{\ell = 1}^L$ the repeated application of the two steps above results in a discrete-time dynamical system modeling the residual stream dynamics (the layerwise evolution of tokens $(X^i)_{i = 1}^N$ through the architecture) in a transformer with softmax attention.

\subsubsection{Initialization}
In the model we consider, as the notation introduced above suggests, while the parameters of the attention layers are kept fixed throughout the architecture, the parameters of the MLP layer may change  across layers. 
In particular, at initialization the weight matrices $U^\ell, W^\ell$ and biases $\biasU{\ell},\biasW{\ell}$ are typically sampled \iid across layers from a given distribution.
    Throughout this paper, we assume that, for each $\ell$ and each such matrix, the entries are \iid normally distributed, i.e., 
    \begin{equ}\label{e:init}
    W^\ell_{i,j}\iidsim \mathcal{N}(0,1),\quad U^\ell_{j,i}\iidsim \mathcal{N}(0,1), \quad \biasW{\ell}_{i}\iidsim \mathcal{N}(0,\varbW^2)\quad \text{and} \quad \biasU{\ell}_{j} \iidsim \mathcal N(0,\varbU^2),
    \end{equ}
    for each $i\le d$, $j \le m$ and $\ell \le L$, and fixed constants $\varbW,\varbU \ge 0$. 
   
Under the above assumptions, as discussed in detail in Section~\ref{s:noise}, we can interpret the (random) functions $G_\ell^m~:~\mathbb S^{d-1} \to \mathbb R^d$ as \iid random fields with ($m$-independent) covariance structure given for $x,y\in \mathbb S^{d-1}$ by\footnote{while these fields are defined on $\mathbb R^d$, we restrict our discussion to the space that is relevant for our analysis.}
\begin{equ}
    \K(x,y) \Id \qquad \text{for } \quad \K(x,y) := 
    \E\left[\act\left(\frac{U_1 x}{\sqrt{d}} + b\right)\act\left(\frac{U_1 y}{\sqrt{d}} + b\right)\right] 
    + \varbW^2\,,
    \label{eq:zonal-kernel}
\end{equ}
 where $\Id$ denotes throughout the identity matrix in $d$ dimensions, $U_1 $ is the first row of a $m\times d$ matrix $U$ with iid $\mathcal{N}(0,1)$ entries, and $b\sim \mathcal N(0,\varbU^2)$.

\begin{remark} \label{r:isotropic} By the distributional invariance of Gaussian vectors $U_1$ under orthogonal transformations, we note that $\K$ is invariant under the joint action of the orthogonal group $O(d)$ on its arguments. As a consequence, $\K$ is a \emph{zonal} (or \emph{isotropic}) kernel, i.e., it can be expressed as $\K(x,y) = \kiso(\langle x,y\rangle)$ for a function $\kiso \colon [-1,1]\to \R$. With a slight abuse of notation, in the following we will write $\kiso := \kiso(1) > 0$.\end{remark}

It is known that, under \eqref{e:init},  when the width $m$ diverges the network $G_\ell^m(\cdot)$ converges in distribution to a Gaussian process $G_\ell(\cdot)$ \cite{neal96,jacot,Matthews2018GaussianPB}. Here, $G_\ell$ is a centered Gaussian process over $\S^{d-1}$ with values in $\R^d$ characterized by the covariance operator $\K \Id$ from \eqref{eq:zonal-kernel}. Furthermore, since $\K$ is continuous on $\S^{d-1}$, the Karhunen-Loève expansion \cite{Rasmussen2006Gaussian} of $G_\ell$ can be written as
    \begin{equ}
    \label{eq:KL_expansion}
G_\ell(x) = \sum_{k=1}^\infty \sigma_k(x) Z_k,
\end{equ}
for functions $\sigma_k~:~\mathbb S^{d-1}\to \mathbb R$ depending on the choice of $\act$, where $\{Z_k\}_k$ are \iid $d$-dimensional standard  Gaussian random vectors: $Z_k \iidsim \mathcal N_d(0, \Id)$.

\subsection{Stochastic scaling limit}
\label{sec:intro-scaling-limit}

Let us now denote by $\mu^{N,L,m}$ the empirical distributions on the $N$ tokens in a transformer of depth $L$ formed by MLPs of width $m$: for $t \in [0,T)$ we define
\begin{equ}
    \mu_t^{N,L,m}:=\frac{1}{N}\sum_{i=1}^N\delta_{X_\ell^{i}}\qquad \text{whenever } t \in [\ell T/L, (\ell +1) T/L).
\end{equ}
This allows us to rewrite the attention vector-field as $\Attn_{\mathbf X_\ell}(x) = \Attn(\mu_t^{N,L,m}, x) $ where
   \begin{equ}\label{e:attention}
        \Attn(\mu, x) \coloneqq \frac{1}{Z[\mu](x)} \int {\exp( \langle Qx , K y\rangle)} Vy\, \mathrm{d}\mu(y)\,,
    \end{equ}
for $Z[\mu](x) \coloneqq \int {\exp( \langle Qx , K y\rangle)}\, \mathrm{d}\mu(y)$.
    
\subsubsection{Finite $N$ setting} 
In this section, we fix the number $N \in \mathbb N$ of tokens and, for a given initial condition $(X_0^i)_{i = 1}^N$ with $X_0^i \in \mathbb S^{d-1}$ for all $i \in [N]$, we compare the discrete-time dynamics \eqref{eq:transformer-update} to the continuous-time process $ \mathbf{\bar X}~:~\mathbb R_{\geq 0} \to (\mathbb S^{d-1})^N$, which is the solution to the system of SDEs:
\begin{equ}
\label{eq:dynamics_tokens}
d{\bar X}^i_{t} =\left(P_{\bar X^i_t}\attN{t}{\bar X}-\kiso\frac{d-1}{2}\bar X^i_t\right)dt + \sum_{k=1}^\infty \sigma_k({\bar X}^i_t) P_{\bar X^i_t} dB^k_t,
\end{equ}
with the initial condition $\bar X_0^i = X_0^i$ for all $i \in [N]$. Here,  $(B_t^k)_{k = 1}^\infty$ are independent, real, $d$-dimensional Brownian motions adapted to a filtration $\{\mathcal F_t^B\}_{t \geq 0}$, $\kiso > 0$ and the diffusion coefficients $\sigma_k$ were defined in the previous section, $P_x y = y - \langle y, x\rangle x$ denotes the projection on the tangent space $T_x \mathbb S^{d-1}$ at $x$, and
\begin{equ}
    \mu_{\mathbf {\bar X}_t}^N := \frac 1 N \sum_i \delta_{\bar X_t^i}\,.
\end{equ} 

We establish quantitative convergence of the trajectories of the residual stream $\mathbf X_\ell$ to the continuous-time process $\bar {\mathbf X}_t$ in the deep and wide limit ($L, m \to \infty$). Here and throughout, $|\cdot|$ denotes the Euclidean norm.
\begin{theorem}
\label{p:main}
     Fix $N \in \mathbb N$, let $T>0$ and let the above assumptions hold. Then for every initial condition $\mathbf X_0 = \bar {\mathbf X}_0 \in (\mathbb S^{d-1})^N$ and every $m, L \in \mathbb N$ there exists a coupling\footnote{That is, a joint realization of $(G_\ell^m)_{\ell=1}^L$ and $(B^k)_{k=1}^\infty$ with the prescribed marginals, see Remark~\ref{rem:global_coupling}.} $\Gamma_{m,L}$ of $(G_\ell^m)_{\ell=1}^L$ and $(B^k)_{k\ge1}$ such that
    \begin{equ}\label{e:main}
\E_{\Gamma_{m,L}}\left[\sup_{0\leq t\leq T}\frac{1}{N}\sum_{i=1}^N |X^{i}_{\lfloor{Lt/T}\rfloor} - \bar X^i_t|^2\right] \leq  
    \mathsf c_{N,T}\left(\frac{\log L}{L} + \frac 1{ m}\right)
    \end{equ}
    where $\mathsf c_{N,T}$ is a constant depending on $N$ and $T$.
\end{theorem}

  In the limiting, stochastic interacting particle system identified in Theorem~\ref{p:main}, the two components of the architecture play complementary roles. On one hand, the attention component captures the mean-field interaction among the particles and enters the limiting equation \eqref{eq:dynamics_tokens} as a deterministic, interacting drift term, analogously to previous works, e.g.,  \cite{geshkovski2023mathematical}. On the other hand, under the scaling \eqref{e:alphas}, the MLP component yields a non-interacting, transport-type common or environmental noise term, in the form of an asymptotically centered Gaussian random field characterized by the Karhunen-Loève expansion of the limiting covariance kernel \eqref{eq:KL_expansion}. As \eqref{eq:dynamics_tokens} is formulated in It\^o form, the additional radial drift term enforces the invariance of $\mathbb S^{d-1}$ under \eqref{eq:dynamics_tokens} by canceling the opposite It\^o correction term (e.g., when applying It\^o's formula to $|X_t^i|^2$).

\begin{remark}[Assumptions]
While Theorem~\ref{p:main} is stated for a single-head softmax attention map and \iid Gaussian initialization of the MLP weights, neither restriction is essential. For the attention mechanism, our arguments rely exclusively on the induced vector field
\begin{equ}
x \longmapsto \Attn(\mu, x)
\end{equ}
being uniformly (in $\mu$) Lipschitz on $\mathbb S^{d-1}$, which ensures the regularity needed to control the associated Vlasov flow. Consequently, our proof automatically holds for multi-head constructions and other Lipschitz-continuous attention variants, including Unnormalized Self Attention \cite{geshkovski2023mathematical}. Similarly, the Gaussianity of the MLP weights is imposed purely for technical convenience: The central limit theorem result we leverage holds universally, in the wide-network limit, for centered, sufficiently light-tailed \iid initializations, yielding the same qualitative limiting behavior \cite{lee2018deep}. 
\end{remark}

\begin{remark}[Uniformity of the limits] \label{r:uniform} The uniform  estimate in Theorem~\ref{p:main} implies commutativity of the $L, m$ limits. Furthermore, the dependence of $\mathsf c_{N,T}$ on $N$ in the above result can be removed by leveraging the distributional convergence $G^m \to G$ at the \emph{functional} level, as opposed to the pointwise level. Such bounds were established in a recent work \cite[Theorem 3.16]{favaro2025quantitative}, under stronger conditions on the activation, i.e., assuming $\act \in C^\infty(\R)$, but result in a slower rate of $m^{-1/4}$ for the wide limit. We also note that propagating this uniform in $N$ layerwise estimate to the deep limit requires establishing a corresponding uniform in $N$ strong Euler-Maruyama convergence estimate in the common-noise, manifold-valued setting considered here. We expand on these points in Section~\ref{sec:deep_lim}. 
\end{remark}

\begin{remark}[Construction of the global adapted coupling]\label{rem:global_coupling}
In Theorem \ref{p:main}, the coupling $\Gamma_{m,L}$ links the finite-width MLP to the limiting continuous noise term. As explained in Section \ref{s:noise}, the starting point of this construction is the layer-by-layer estimate from Proposition \ref{prop:basteri_trevisan}, which for any deterministic $N$-tuple $(x_1,\dots,x_N)\in(\S^{d-1})^N$ couples the MLP layer $G_\ell^m$ with the appropriately rescaled increments of $(B^k)_k$ so as to approximately minimize their $L^2$ distance.
Since the hidden states $\mathbf X_\ell$ are random, we define the coupling recursively: At layer $\ell$, conditional on the outputs of the previous layers, we use the deterministic-input coupling from Proposition \ref{prop:basteri_trevisan} with the realized current inputs. Iterating this construction yields a global coupling $\Gamma_{m,L}$ adapted to the natural filtration. Whenever the coupling is later applied to random tuples, it is understood in this conditional, layerwise sense.
\end{remark}

\paragraph{Related works} The study of residual stream dynamics in deep learning architectures as a continuous time limit originated with the NeuralODE literature \cite{weinan2017proposal,chen2018neural,lu2020mean}. Recent works have established analogous quantitative ($m, L \to \infty$) convergence results for ResNets in the \emph{deterministic} limit, i.e., under the scaling $1/L^{\alpha}$ for $\alpha \in (1/2,1]$ \cite{chizat25}, and have more recently extended these results to the joint large $d$ limit \cite{chizat26}.  Stochastic scaling limits for ResNets were studied in \cite{peluchetti2020infinitely} and in \cite{marion25}, where the authors identify Brownian motion as the deep limit of the residual stream at initialization, while other works prove that the deep and wide limits commute \cite{hayou23,hayou24}. However, all these results hold in the noninteracting setting, at the level of marginals, and in the distributional sense, as opposed to our interacting, pathwise convergence result in $L^2$. In this sense, our results can be seen as an extension and a sharpening of this line of work\footnote{The logarithmic term in our estimate can be dropped if the supremum in \eqref{e:main} is taken \emph{at} the discrete time updates, i.e., at multiples of $T/L$, see Remark~\ref{r:droplog}}, preparing the ground for their extension to training, where the pathwise viewpoint is required. Furthermore, a distinctive mathematical feature of the setting considered above is the Layer Normalization, implemented as the radial projection $\LN(v) = v/{|v|}$. While this constrains the dynamics to the compact manifold $\mathbb{S}^{d-1}$ (yielding useful uniform moment controls), it introduces nonlinear couplings that invalidate the direct application of standard Wasserstein stability estimates for unprojected Euler-Maruyama schemes, such as those predominantly used in the aforementioned continuous-time limits. Bypassing this obstruction requires a careful expansion of the normalization operation and constitutes a central technical step of our estimates.

In the interacting setting, the interpretation of deep transformers as deterministic mean-field interacting particle systems goes back to the seminal works \cite{sander2022sinkformers, geshkovski2023emergence, geshkovski2023mathematical}, initiating a fruitful line of research (see, e.g., \cite{karagodin2024clustering, geshkovski2024measure, chen2025quantitative, polyanskiy2025synchronization, karagodinnormalization}). The MLP component was included in this analysis only very recently \cite{alvarez2026perceptrons,zimin2026yuriiformer}, but these papers considered a different scaling, yielding an additional \emph{deterministic} term to the limiting equation. Stochastic versions of these models were investigated in \cite{shalova2024solutions, balasubramanian2025structure, gerber2025formation}, where the authors studied the effect of adding iid (idiosyncratic) Brownian motions to each token on the multiplicity of the steady states of the system. However, the noise structure chosen in these papers is not expected to arise intrinsically, motivated by the underlying architecture. More recently, the concurrent  works \cite{fedorov2026clustering, koubbi2026homogenized} derive a stochastic continuous-time interacting limit where the noise arises from the random initialization of the self-attention parameters, resulting in a mean-field interacting particle system with common noise. In contrast, in our setting, the randomness emerges from the inclusion of the MLP block, yielding a common noise term that is, importantly, of \emph{noninteracting} type (as opposed to \cite{fedorov2026clustering, koubbi2026homogenized}, where the noise covariance depends on the current state $\mathbf {\bar X}_t$).

\subsubsection{Large $N$ limit} 

We now consider the Eulerian description of  \eqref{eq:dynamics_tokens},  given by the Stochastic Partial Differential Equation (SPDE) in It\^o form
\begin{equ}
    \dd\mu_t + \text{div}_{\mathbb{S}^{d-1}}( P_x\att{\mu_t}{\cdot} \mu_t ) dt + \sum_{k=1}^\infty\sum_{\alpha = 1}^d \text{div}_{\mathbb{S}^{d-1}}(   (P_x\sigma_{k\alpha} )   \mu_t ) dB^{k,\alpha}_t = \frac{\kiso}{2}\Delta_{\mathbb{S}^{d-1}} \mu_t dt,
\label{eq:limiting-spde}
\end{equ}
where $\text{div}_{\mathbb{S}^{d-1}}$ and $\Delta_{\mathbb{S}^{d-1}}$ are the Riemannian divergence and Laplace-Beltrami operator on $\mathbb S^{d-1}$ and throughout  we adopt the compact notation
\begin{equ}\label{e:sigmaka}
\sigma_{k\alpha}(x) := \sigma_k(x) e_\alpha,
\end{equ}
for $\{e_\alpha\}_{\alpha = 1}^d$ the standard basis of $\mathbb R^d$ and $B^{k,\alpha}$ the $\alpha$-th entry of the $k$-th Brownian vector from \eqref{eq:dynamics_tokens}.
The solutions to the above equation are stochastic processes
with values in the space $\mathcal P_1(\mathbb S^{d-1})$ of probability measures on $\mathbb S^{d-1}$. Moreover, the McKean-Vlasov SDE associated with \eqref{eq:limiting-spde} recovers \eqref{eq:dynamics_tokens} when the initial datum is empirical. More precisely, if \eqref{eq:limiting-spde} is initialized from the empirical measure
\begin{equ}
\mu_0^N = \frac1N\sum_{i=1}^N \delta_{X_0^i},
\end{equ}
and driven by the same Brownian motions as in \eqref{eq:dynamics_tokens}, the solution $\mu_t^N$ of \eqref{eq:limiting-spde} coincides almost surely with the empirical measure of the particle system, namely
$
\mu_t^N = \mu_{\mathbf{\bar X}_t}^N
$. In particular, the space of $N$-point empirical measures is closed almost surely under \eqref{eq:limiting-spde}.

Below, we equip the space of measure-valued stochastic processes with the following distance:
\begin{equ} \label{e:distance}
\distS(\mu, \nu) := \E\left[ \sup_{t \in [0,T]}  \W_2^2(\mu_t, \nu_t) \right]\,,
\end{equ}
where the expectation is taken over the probability space where the processes $\mu, \nu$ are jointly defined\footnote{This includes integration over the initial condition, in case it is random, unless specified otherwise (in which case we will write the conditional expectation as $\mathbb E[\,\cdot\,|\mathcal F_0]$).}.  Here and throughout, $ \W_p$ denotes the Wasserstein $p$-distance on probability measures on $\mathbb R^d$ with finite $p$-moments, defined as 
\begin{equ}\label{e:wasserstein}
\W_p(\mu, \nu) := \left(\inf_{\pi\in\Gamma(\mu,\nu)} \int  |x-y|^p \dd \pi(x,y)\right)^{1/p}\,,
\end{equ}
where $\Gamma(\mu,\nu)$ is the set of all joint probability measures on $\S^{d-1}\times\S^{d-1}$ with marginals $\mu$ and $\nu$.

The well-posedness of \eqref{eq:limiting-spde} and the subsequent propagation of chaos result, which connects it to the particle system \eqref{eq:dynamics_tokens}, are established by adapting \cite[Theorem 13 and Theorem 21]{coghi2016propagation} to the spherical setting:
\begin{proposition}
\label{thm:main}
Let $T> 0$ and $\nu_0$ be $\mathcal F_0^B$-measurable, then there exists a unique measure-valued solution $(\nu_t)_{t\in [0,T]}$ to \eqref{eq:limiting-spde}.  Furthermore, let $X_0^i$ be sampled \iid from $\mu_0 \in \mathcal P_1(\mathbb S^{d-1})$, then there exist a positive constant $\mathsf c_T > 0$ such that
\begin{align}
\distS(\mu^N, \mu)  \leq \mathsf c_T  \begin{cases}
            N^{-1/2}& d<4\\
            N^{-1/2}\log(1+N)& d=4\\
            N^{-2/d} & d>4 
        \end{cases}.
\end{align}
\end{proposition}

The above result can be combined with the uniform-in-$N$ convergence in $m$ and $L$ discussed in Remark~\ref{r:uniform}, yielding the following:

\begin{corollary}
\label{cor:main}
Under the assumptions of Theorem~\ref{p:main} and Proposition~\ref{thm:main},
\begin{equ}
\lim_{N \to \infty} \, \lim_{L \to \infty} \, \lim_{m \to \infty} \, \distS(\mu^{N,L,m}, \mu) = 0\,.
\end{equ}
Moreover, if the activation function $\act \in C^\infty$, then the three limits above commute.
\end{corollary}

The above result also directly implies convergence in weaker topologies. For instance, at the level of the marginals, for any $t > 0$ it holds that 
\begin{equ}
    \lim_{N \to \infty} \, \lim_{L \to \infty} \, \lim_{m \to \infty} \mu^{N,L,m}_{\lfloor Lt/T \rfloor} = \mu_t
    \end{equ}
    in distribution. Along the same lines, our result equally holds in probability over the (random) sequence of initial condition, interpreting the expectation in \eqref{e:distance} as conditional on such sequence.

\paragraph{Related works} Propagation of chaos in deterministic and stochastic systems is a topic of active research, originating in the mathematical physics literature (see, e.g., \cite{sznitman1991topics} and references therein). These results were applied to deterministic mean-field models of transformers in \cite{bruno2024emergence, bruno2025multiscale,geshkovski2023mathematical}.  In the common noise setting, analogous results were obtained in \cite{coghi2016propagation}, but to the best of the authors' knowledge, no such results exist in the literature for stochastic models of representation dynamics in deep neural network architectures, in particular when considering exchangeable and quantitative deep, wide, and mean-field limits.

\subsection{ Synchronization by noise } 
\label{sec:intro-sync-by-noise}
The second main result presented in this paper concerns the qualitative behavior of the limiting dynamics identified in the previous section. 

 A series of recent works has noticed that tokens in the residual stream of simple transformer models tend to organize into clusters and eventually collapse to a single point \cite{geshkovski2023emergence, bruno2024emergence, chen2025quantitative, geshkovski2024dynamic}. At the theoretical level, this effect has then been investigated, at increasing levels of generality, for certain variants of deep architectures under specific choices of model parameters $Q,K,V$. More specifically, it was shown in \cite{geshkovski2023emergence, geshkovski2023mathematical} that, for softmax attention with $Q^\top K = \beta\Id$, for a $\beta > 0$, $V = \Id$, the self-interaction energy
\begin{equ}\label{e:energy}
    \mathcal E_\beta(\mu) := \frac 1 2 \iint \exp(\beta) -  \exp(\beta \langle x, y\rangle)  \, \dd\mu(x) \, d \mu(y)     
\end{equ}
is a Lyapunov function for the self-attention dynamics,
\begin{equ}\label{e:sa}
\partial_t \mu_t = - \text{div}_{\mathbb S^{d-1}}(\mu_t P_{x} \Attn(\mu_t))\,,
\end{equ}
i.e., it satisfies
\begin{equ}\label{e:dissipation}
    \frac \dd{\dd t} \mathcal E_\beta(\mu_t) \leq 0\,.
\end{equ}
In particular, it was shown that the model being considered enjoys a (twisted) Wasserstein gradient flow structure for this choice of energy \cite{geshkovski2023mathematical,castin2025unified}, whose zeroes (global minimizers) are fully synchronized states given by a Dirac delta $\delta_{x^*}$ centered at a $x^* \in \mathbb S^{d-1}$ \cite{bilyk2019geodesic}.
However, the presence of metastable saddle points of the energy \eqref{e:energy} saturating the bound \eqref{e:dissipation} \cite{bruno2024emergence}, prevents us from obtaining strict dissipation estimates for the dynamics \eqref{e:sa}  that hold uniformly over the initial conditions.

\subsubsection{Synchronization estimates}

Below, we prove that the addition of noise to the system, as described in the previous section, destroys the invariance of those metastable states, allowing to establish exponential decay estimates for the energy \eqref{e:energy} provided that the noise structure satisfies the following condition:
\begin{assumption}
\label{ass:sync_dissipation}
 The function $x \mapsto G^m_\ell(x)$ is not odd almost surely in the choice of parameters. Furthermore, it holds that 
    \begin{equ}\label{e:barl}
        \bar \lambda := (d-1)\kiso'(1) - (d-3) \kiso(1) < 0,
    \end{equ}
where $\kiso$ is defined in Remark \ref{r:isotropic}.
\end{assumption}

In fact, such dissipation estimates can be shown to hold provided that the noiseless model \eqref{e:sa} (resulting from the choice of parameter matrices $Q, K, V$) does not counter the synchronization by increasing too strongly the energy \eqref{e:energy}. This condition is encoded in the following assumption:
\begin{assumption}\label{ass:sync_forcing}
 There exists $\epsilon \in \mathbb R$ such that, under \eqref{e:sa},
$
    \frac \dd{\dd t} \mathcal E_\beta(\mu_t) \leq \epsilon \mathcal E_\beta(\mu_t)
$ on $\mathcal P_1(\mathbb S^{d-1})$.
\end{assumption}

Under the above conditions, we establish Gr\"onwall-type exponential control on the evolution of the energy \eqref{e:energy}. When the rate of the exponential is negative (e.g., when noise is sufficiently strong, and the deterministic field is, comparatively, not too counteractive), 
we obtain exponential decay of the energy towards its global minimizer, i.e., a totally synchronized state.
\begin{theorem}\label{t:sync}
Let $\beta > 0$, 
Assumption~\ref{ass:sync_dissipation} and 
Assumption~\ref{ass:sync_forcing} hold, 
then there exists $\bar \lambda' < 0$ such that the solutions $\bar \mu_t$ of \eqref{eq:limiting-spde} satisfy
\begin{equ}
\mathbb E\left[\mathcal E_\beta(\bar \mu_t)\right] \leq \exp((\epsilon + \bar \lambda')t) \mathbb E[
\mathcal E_\beta(\bar \mu_0)
] \,.
\end{equ}
\end{theorem}

\begin{figure}[t]
    \centering
\includegraphics[width=\linewidth]{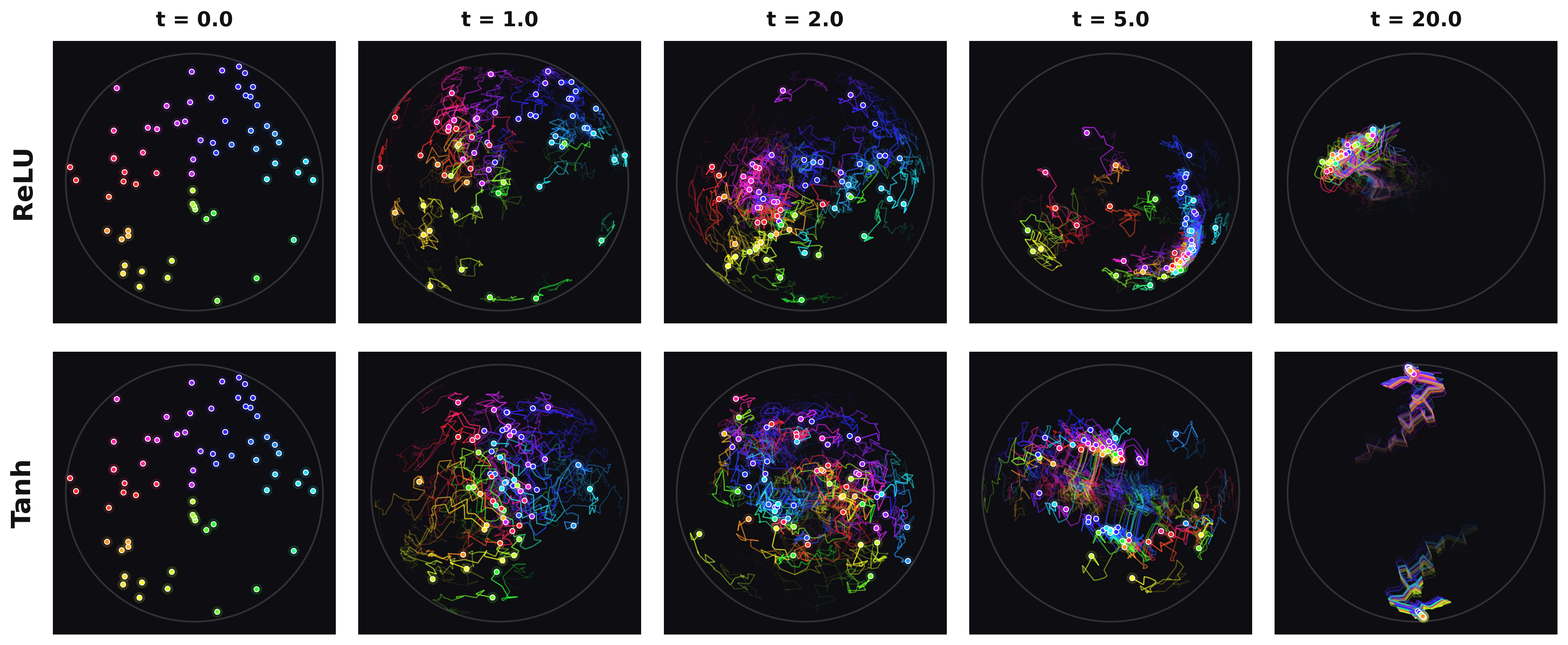}
    \caption{Noise-induced synchronization without attention. Evolution of the first two components of $N=64$ tokens on $\mathbb{S}^3$ (d = 4) over $L=2000$ layers ($T=20$), with attention disabled and MLP width $m=2048$ and variance parameters $\sigma_U^2 = \sigma_W^2 = 0$. For both ReLU and Tanh activations, an initially near-uniform token configuration evolves into localized clusters, illustrating the contractive effect of the common MLP noise discussed in Section \ref{sec:intro-sync-by-noise}. 
    We note that under the given choice of parameters, the activation $\act = \tanh$ yields an a.s. odd random field, violating Assumption~\ref{ass:sync_dissipation} and resulting in a 2-point random attractor. }
    \label{fig:sphere_clustering}
\end{figure}

The proof of this result is provided in Section~\ref{s:sync}, and numerical experiments illustrating its validity are provided in Figure~\ref{fig:sphere_clustering}   and in Appendix~\ref{sec:experiments}. 
Note that, since
\begin{equ}
|x-y|^2 \leq \frac 2{\sinh \beta} (\exp(\beta) - \exp(\beta \langle x,y\rangle) )
\end{equ}
the above result implies convergence in $L^2$, which in turn implies convergence in probability to a random point attractor. This establishes a quantitative version of the synchronization by noise result in \cite{cranston16}.

\begin{remark}[Relation to Lyapunov exponents] We note that the definition \eqref{e:barl} is closely related to the one of the top Lyapunov exponent $\lambda_{1}$ of the noisy dynamics,
\begin{equ}\label{e:purenoise}
dX^i_t = - \kiso(1)\frac{d-1}{2}X^i_t\,dt + P_{X^i_t} \sum_{k=1}^\infty \sigma_k(X_t^i) dB^k_t,\qquad i \in [N]
\end{equ}
which by \cite[Thm. 3.1]{raimond99} is defined as
\begin{equ}\label{e:lambda}
\lambda_{1} = \frac {d-3}2 \kiso'(1) - \frac{d-1}2 \kiso(1)\,.
\end{equ}
Having $\lambda_1 < 0$, in turn, is a sufficient condition for weak synchronization by isotropic noise on the sphere, i.e., for $\lim_{t \to \infty} d(X_i(t), X_j(t)) = 0$ for all $i,j \in [N]$ \cite[Thm. 3.1]{cranston16}. Note that $ \lambda_1 = \bar \lambda/2 -  \kiso(1) -  \kiso'(1)$, so that  $\bar \lambda < 0$ implies $\lambda_1 < 0$ (since $\kiso(1), \kiso'(1) \geq 0$ as shown in Section~\ref{s:sync}). On the other hand, our stronger condition $\bar \lambda < 0$ allows us to establish quantitative convergence estimates for the noisy system \eqref{e:purenoise}. Note that $\bar \lambda < 0$ is only possible for $d>3$.
\end{remark}

\begin{remark}[Assumptions of Theorem \ref{t:sync}]
The non-odd assumption of $G_\ell^m$ guarantees that the random discrete attractor of the dynamics is not composed of more than one point. 
Assumption~\ref{ass:sync_forcing} essentially controls the effect of the deterministic dynamics. If this term is not too strongly adverse to the energy dissipation estimate we obtain from the common noise term (that is, if $\epsilon + \bar \lambda'< 0$), then the energy is dissipated at an exponential rate. This is trivially the case when $Q^\top K = \beta \Id$, $V = \Id$, as in that case Assumption~\ref{ass:sync_forcing} holds with $\epsilon = 0$, but this result extends to the repulsive regime (e.g.,  $Q^\top K = \beta \Id$ and $V = - \delta \Id$ for $\delta>0$ small enough). We further note that, when considering a finite-$N$ system, by the closure of the space of $N$-particles empirical measures under \eqref{eq:limiting-spde}, it is sufficient to verify Assumption~\ref{ass:sync_forcing} on that space, as opposed to the whole $\mathcal P_1(\mathbb S^{d-1})$, to apply Theorem~\ref{t:sync}.
Finally, we note that, as the rates $\epsilon$ in Assumption~\ref{ass:sync_forcing} and $\bar \lambda'$ possibly depend on the choice of $\beta$, the sharpest dissipation estimate our bounds can provide is obtained by optimizing $\epsilon + \bar \lambda'$ over that parameter.
\end{remark}

\paragraph{Related works.} Synchronization and clustering in deterministic transformers have been active topics of research in recent years, with estimates both in the finite- \cite{markdahl2017almost, criscitiello2024synchronization, geshkovski2023mathematical} and infinite-dimensional settings \cite{chen2025quantitative}. 
Synchronization by common noise is a classical phenomenon in the theory of random dynamical systems and stochastic flows. General criteria for weak synchronization and random point attractors were established, e.g., in \cite{flandoli2017synchronization}. On the sphere, the computation of the Lyapunov spectrum for isotropic Brownian flows goes back to \cite{raimond99}, and weak synchronization in the same setting was studied in the noninteracting case by \cite{cranston16} and more recently by \cite{engel26}. Closer to the transformer setting, recent works discuss synchronization in common-noise systems on the sphere. In \cite{fedorov2026clustering}, random value matrices lead to a stochastic interacting system whose clustering behavior is precisely characterized in the two-particle setting, as opposed to the quantitative, arbitrary-$N$ estimates we establish. The homogenized analysis of \cite{koubbi2026homogenized} considers random attention parameters across layers and heads and identifies deterministic and stochastic scaling limits, together with regimes mitigating clustering, yielding an interacting common noise system that is complementary to the one treated in this work.

\subsubsection{Explicit formulas for \(\lambda_1\) and \(\bar\lambda\)}
Finally, we briefly discuss the explicit computation of synchronization rates $\bar \lambda$ and $\lambda_1$ (the top Lyapunov exponent) as a function of the model's parameters. In Section~\ref{s:sync2} we show that for generic activation $\act$ and dimension $d$:
$$
\kiso(1)=\E\left[\act\left(\frac{1}{\sqrt d}X+b_U\right)^2\right]+\varbW^2 \qquad 
\kiso'(1)=\frac{1}{d}\E\left[\act'\left(\frac{1}{\sqrt d}X+b_U\right)^2\right]\,,
$$
allowing to verify that $ \lambda_1 < 0$ and $\bar \lambda < 0$ hold for a typical choice of neural network activation for which these expressions acquire a particularly simple form:
\begin{proposition}\label{p:le}
Let $\act(y) = \max\{0,y\}$ be the ReLU activation, then 
\begin{equ}
    \lambda_1 = -\frac{1}{2d}-\frac{(\varbU^2+2\varbW^2)(d-1)}{4}<0\,,
\end{equ}
and 
\begin{equ}
    \bar \lambda = \frac 1 d - \frac {(\varbU^2+2\varbW^2)(d-3)}{2}\,.
\end{equ}
\end{proposition}
\noindent This result is proven in Section~\ref{s:sync2}.

\subsection{Structure of the paper and notation} 

The remainder of the paper is structured as follows. Theorem~\ref{p:main} and Proposition~\ref{thm:main} are proven in Sections~\ref{sec:deep_lim} and \ref{sec:MF_limit}, respectively proving estimates on the wide-and-deep limit (large $L$ and large $m$) and propagation of chaos in the long context limit (large $N$). The synchronization results Theorem~\ref{t:sync} and Proposition~\ref{p:le} are proven in Section~\ref{s:sync}.
Throughout, we use the letter $\mathsf c$ to denote some positive and finite constant that may depend on several parameters, e.g, $\alpha, \beta, \ldots$. To flag such dependence explicitly, we write $\mathsf c = \mathsf c_{\alpha, \beta, \ldots}$.
We also allow the value of $\mathsf c$ to change from line to line. Some of the more technical proofs in Section \ref{sec:deep_lim} are deferred to Appendix \ref{app:technical-proofs}.

\section{Deep and wide limit}
\label{sec:deep_lim}

In this section, we present the quantitative convergence of the two-step forward Euler (Lie-Trotter) iteration \eqref{eq:transformer-update} to the limiting particle system \eqref{eq:dynamics_tokens} for a fixed number $N$ of particles. 
While stated in the introduction for a specific choice of attention \eqref{e:attention}, our results apply to any Attention map $\attempty ~:~\mathcal P(\mathbb S^{d-1})  \times \mathbb R^{d}  \to \mathbb R^d$
satisfying the following assumptions:
\begin{assumption}[Self-Attention regularity]
\label{ass:att_reg} For any $R>0$ there exists a constant $\mathsf c>0$ such that for any $\mu,\nu\in \mathcal{P}_c(\R^d)$ with support in the Euclidean ball $B(0,R)$ of radius $R$ around $0$,
    \begin{align*}
        \|\att{\mu}{\cdot}\|_{L^\infty(\R^d,\R^d)}&\leq \mathsf c,\\
        \|\nabla_x \att{\mu}{\cdot}\|_{L^\infty(\R^d,\R^d\times\R^d)}&\leq \mathsf c,\\
        \|\att{\mu}{\cdot} -\att{\nu}{\cdot}\|_{L^\infty(B(0,R),\R^d)}&\leq \mathsf c  \W_2(\mu,\nu).
    \end{align*}
\end{assumption}
\begin{remark}[Examples of attentions satisfying Assumption~\ref{ass:att_reg}]
Assumption~\ref{ass:att_reg} is satisfied by many standard self-attention variants, as soon as the
weight matrices have bounded operator norm and we restrict to measures supported in a ball $B(0,R)$.
A convenient class is given by attention maps that can be written as a finite sum (to include multi-head attention) of terms of the form
\begin{equ}
\Gamma_\mu(x)\;:=\;\int_{\R^d} V y\,\mathcal K_\mu(x,y)\,\dd\mu(y),
\qquad x\in\R^d,
\end{equ}
for $V\in\R^{d\times d}$ and some nonnegative kernel $\mathcal K_\mu(x,\cdot)$ depending on $\mu$ (typically normalized so that
$\int \mathcal K_\mu(x,y)\,\dd\mu(y)=1$). 
For instance, in \cite[Section~2 and Appendix~B]{castin2025unified} the following choices are studied
(for fixed $Q,K\in\R^{k\times d}$):
\begin{itemize}
\item \textbf{Softmax self-attention}:
\begin{equ}
\Gamma_\mu^{(\mathrm{SM})}(x)
:=\frac{\int_{\R^d} V y\,e^{\,\langle Qx, Ky\rangle }\,\dd\mu(y)}{\int_{\R^d} e^{\,\langle Qx, Ky\rangle }\,\dd\mu(y)}.
\end{equ}
\item \textbf{$L^2$ self-attention}:
\begin{equ}
\Gamma_\mu^{(\mathrm{L2})}(x)
:=\frac{\int_{\R^d} V y\,e^{-|Qx-Ky|^2}\,\dd\mu(y)}{\int_{\R^d} e^{-|Qx-Ky|^2}\,\dd\mu(y)}.
\end{equ}
\item \textbf{Sinkhorn self-attention}:
\begin{equ}
\Gamma_{\mu,\varepsilon}^{(\mathrm{sink})}(x)
:=\frac1\varepsilon\int_{\R^d} V y\,\mathcal K_{\mu,\varepsilon}^\infty(x,y)\,\dd\mu(y),
\end{equ}
where $\mathcal K_{\mu,\varepsilon}^\infty$ is the bi-stochastic kernel obtained as the limit of the
Sinkhorn--Knopp iterations (see \cite[Eq.~(2.4)]{castin2025unified}).
\end{itemize}
If $\mathrm{supp}(\mu)\subset B(0,R)$, then the uniform boundedness and $x$--Lipschitz bounds follow from the
pointwise estimates on $\sup_x\|\Gamma_\mu(x)\|$ and $\sup_x\|D_x\Gamma_\mu(x)\|$ proved in
\cite[Appendix~B, Lemmas~B.2, B.3, B.6, B.7]{castin2025unified}.
Moreover, the same lemmas provide stability estimates of the form
$\sup_{x\in B(0,R)}|\Gamma_\mu(x)-\Gamma_\nu(x)|\le \mathsf c_R\,W_p(\mu,\nu)$ for $p\ge 1$; taking $p=2$ yields the third inequality in
Assumption~\ref{ass:att_reg}.
\end{remark}

We will further require some regularity assumptions on the MLP layer component of the architecture:
\begin{assumption}[Assumption on the MLP activation]
\label{ass1}
The activation function $\act~:~\mathbb R \to \mathbb R$ is $L_\act$-Lipschitz continuous for a $L_\act > 0$, non polynomial and non almost-everywhere constant.
\end{assumption}

We now proceed to state our first main result, an extension of Theorem~\ref{p:main}, in full generality:
\begin{proposition}
\label{prop:quantitative_deep_limit}
     Let Assumptions~\ref{ass:att_reg} and \ref{ass1} hold, let $X^{i}_{\ell}$ denote the discrete-time dynamics defined in \eqref{eq:transformer-update}, and let $\bar X_t^i$ solve \eqref{eq:dynamics_tokens}. Then for every $m, L \in \mathbb N$ there exists a filtered probability space $(\Omega,\mathcal F, \mathcal F_t, \mathbb P)$ where $G_\ell^m$ and $(B^k)_k$ are jointly defined such that
    $$
    \E\left[\sup_{0\leq t\leq T}\frac{1}{N}\sum_{i=1}^N |X^{i}_{\lfloor{Lt/T}\rfloor} - \bar X^i_t|^2\right] \leq  
    \mathsf c_N \exp({\mathsf c_N T^2} ) \left(\frac{\log L}{L} + \frac 1{ m}\right),
    $$
    where $\mathsf c_N$ is a constant depending on {$ \act, \att{\cdot}{\cdot}, d, N$, }but not on $m, L$ or $T$. Assuming further that $\act \in C^\infty(\R)$ there exists a $m, L, N$-independent constant $\mathsf c$ such that we have 
        $$
    \E\left[\sup_{0\leq t\leq T}\frac{1}{N}\sum_{i=1}^N |X^{i}_{\lfloor{Lt/T}\rfloor} - \bar X^i_t|^2\right] \leq 
    \mathsf c \exp({\mathsf c T^2} ) \left(\frac{\log L}{L} + \frac 1{ m^{1/4}}\right).
    $$
\end{proposition}

\subsection{Proof strategy}
\label{subsec:deep-wide-proof-strategy}

To compare the discrete transformer dynamics with the limiting SDE, we
introduce two intermediate objects. The first one is an auxiliary discrete-time
process driven by the limiting Gaussian field. More precisely, if
\((G_\ell)_{\ell\ge1}\) denotes the iid Gaussian field arising in the
infinite-width limit of the MLP and $\Delta t := T/L$, we define
\((\widehat {\mathbf X}_\ell)_{\ell=0,\dots,L}\) by
$
\widehat X_0^i=X_0^i,
$
for all $i\in[N]$ and
\begin{equ}\label{e:xhat}
\widehat X_{\ell+1}^i
=
\widehat X_\ell^i
+
\Delta t\,P_{\widehat X_\ell^i}
\Attn_{\widehat{\mathbf X}_\ell}(\widehat X_\ell^i)
+
\sqrt{\Delta t}\,
P_{\widehat X_\ell^i}G_{\ell+1}(\widehat X_\ell^i)
-
\Delta t\,\frac{d-1}{2}\kiso(1)\widehat X_\ell^i  \,.
\end{equ}
 Strictly speaking, this discrete scheme is defined using a suitably regularized
extension of the projection, the attention drift, and the radial correction to
a neighborhood of the sphere, as the latter is not preserved by \eqref{e:xhat} (more details on this extension in Section~\ref{sec:single-step}).

The second intermediate object is the continuous-time interpolation
\((\widehat {\mathbf X}_t)_{t\in[0,T]}\) of the auxiliary scheme. If
\(\ell_s:=\lfloor s/\Delta t\rfloor\), we define
\begin{equ}
\widehat X_t^i
=
X_0^i
+
\int_0^t
P_{\widehat X_{\ell_s}^i}
\Attn_{\widehat{\mathbf X}_{\ell_s}}(\widehat X_{\ell_s}^i)\,ds
-
\int_0^t
\frac{d-1}{2}\kiso(1)\widehat X_{\ell_s}^i\,ds
+
\sum_{k=1}^\infty
\int_0^t
P_{\widehat X_{\ell_s}^i}\sigma_k(\widehat X_{\ell_s}^i)\,dB_s^k .
\end{equ}
where the coefficients are again suitably extended to a neighborhood of the sphere and smoothly truncated outside.

The proof then proceeds via triangle inequality:
\begin{equ}
\frac 1 N \sum_{i=1}^N\vert X^{i}_{\lfloor t/\Delta t\rfloor}-\bar X_t^i\vert^2
\le
3\,\frac 1 N \sum_{i=1}^N\vert X^{i}_{\lfloor t/\Delta t\rfloor}-\widehat X^i_{\lfloor t/\Delta t\rfloor}\vert^2
+3\frac 1 N \sum_{i=1}^N\vert\widehat X^i_{\lfloor t/\Delta t\rfloor}-\widehat X^i_t\vert^2
+3\frac 1 N \sum_{i=1}^N\vert\widehat X^i_t-\bar X^i_t\vert^2.
\end{equ}
After taking the supremum over \(t\in[0,T]\) and expectation, the proof reduces
to estimating the three terms on the right-hand side separately.

\begin{enumerate}[(i)]
\item The terms  \(X^{m,L,i}_\ell-\widehat X^i_\ell\)
are estimated in  Sections~\ref{s:noise} and~\ref{sec:single-step}. More specifically, in Lemma~\ref{lem:strong_error} we show that there exists a constant
\(\mathsf c_N>0\), independent of \(m\), \(L\), and \(T\), such that
\begin{equ}\label{e:firstbound}
\mathbb E\left[\max_{0\le \ell\le L}
\frac 1 N \sum_{i=1}^N \vert X^{i}_\ell - \widehat X^i_\ell\vert^2\right]
\le
\mathsf c_N e^{\mathsf c_N T}
\left(\frac1m+\frac1L\right).
\end{equ}
The \(m^{-1}\) contribution comes from the quantitative coupling between the
finite-width MLP field \(G_\ell^m\) and its Gaussian limit \(G_\ell\), provided
by Proposition~\ref{prop:basteri_trevisan}. The \(L^{-1}\) contribution comes
from the second-order expansion of the two normalization maps. This remainder, denoted by $\mathcal R_\ell^i$, is computed in 
Lemma~\ref{lem:norm_res}  and controlled in 
Lemma~\ref{lem:squared_reminder}. These one-step bounds are
propagated over all layers by a discrete Gr\"onwall argument in
Lemma~\ref{lem:strong_error}. Bounding \eqref{e:firstbound} constitutes one of the main technical challenges of the proof.

\item The terms \(\widehat X^i_{\lfloor t/\Delta t\rfloor}-\widehat X^i_t\) are treated in Section~\ref{sec:strong-single-step}. In Lemma~\ref{lem:increment_delta} we show that there exists a
constant \(\mathsf c_T>0\), independent of \(N\), \(m\), and \(L\), such that
\begin{equ}
\mathbb E\left[\sup_{0\le t\le T}
\frac 1 N \sum_{i=1}^N\vert\widehat X^i_{\lfloor t/\Delta t\rfloor}-\widehat X^i_t\vert^2\right]
\le
\mathsf c_T\,\frac{\log L}L .
\end{equ}
This bound holds because the interpolated scheme has Euler-size increments on each time interval of length \(\Delta t\). This is a classical result extended to hold uniformly in $N$.

\item The terms  \(\widehat X_t^i-\bar X_t^i\) are also  controlled in Section~\ref{sec:strong-single-step}. Lemma~\ref{lem:strong_convergence_sde} shows that there exists a constant
\(\mathsf c_T>0\), independent of \(N\), \(m\), and \(L\), such that
\begin{equ}
\mathbb E\left[\sup_{0\le t\le T}
\frac 1 N \sum_{i=1}^N\vert\widehat X_t^i-\bar X_t^i\vert^2\right]
\le
\mathsf c_T\,\frac1L .
\end{equ}
Here, $\widehat{\mathbf{X}}_t = (\widehat X^i_t)_i$ is viewed as the Euler-Maruyama scheme
associated with the limiting SDE, and the estimate results from a classical Euler-Maruyama convergence bound in a compact subset of $\mathbb R^d$ (using regularity of the coefficients from Lemma
\ref{lem:regularity_projected_coefficients}) adapted to hold uniformly in $N$.
\end{enumerate}

Combining the three estimates above, there exists a constant \(\mathsf c_N>0\),
independent of \(m\), \(L\), and \(T\), such that
\begin{equ}
\mathbb E\left[\sup_{0\le t\le T}
\frac 1 N \sum_{i=1}^N \vert X^{i}_{\lfloor t/\Delta t\rfloor}-\bar X^i_t\vert^2\right]
\le
\mathsf c_N e^{\mathsf c_N T^2}
\left(\frac1m+\frac{\log L}L\right).
\end{equ}
This is the estimate of Proposition~\ref{prop:quantitative_deep_limit}.
Moreover, under the stronger functional-level coupling estimate discussed in
Remark~\ref{r:uniform}, the dependence on \(N\) of the constant from \eqref{e:firstbound}  can be
removed, at the price of replacing the width rate \(m^{-1}\) by the slower
rate \(m^{-1/4}\) in \eqref{e:firstbound} and, consequently, in the above formula.

\subsection{Noise structure and wide limit computations}
\label{s:noise}

The first step in our analysis is to identify the covariance structure of the
random MLP block.
Fix \(x,y\in\mathbb R^d\). Writing \(U_r^\ell x\) for the product of the
\(r\)-th row of \(U^\ell\) with \(x\), we have, for \(j,k\le d\),
\begin{equ}
G_\ell^m(x)_j
=
\frac{1}{\sqrt m}\sum_{r=1}^m
W_{jr}^\ell
\act\left(d^{-1/2}U_r^\ell x+\biasU{\ell}_r\right)
+
\biasW{\ell}_j .
\end{equ}
Using the independence and centering of the entries of \(W^\ell\), together with the
independence of \(W^\ell\), \(U^\ell\), \(\biasU{\ell}\), and
\(\biasW{\ell}\), we compute
\begin{align*}
\E\left[G_\ell^m(x)_jG_\ell^m(y)_k\right]
&=
\frac{1}{m}
\sum_{r,s=1}^m
\E\left[
W_{jr}^\ell W_{ks}^\ell
\act\left(d^{-1/2}U_r^\ell x+\biasU{\ell}_r\right)
\act\left(d^{-1/2}U_s^\ell y+\biasU{\ell}_s\right)
\right] \\
&\qquad
+
\E\left[\biasW{\ell}_j\biasW{\ell}_k\right] \\
&=
\delta_{jk}
\E\left[
\act\left(d^{-1/2}U_1^\ell x+\biasU{\ell}_1\right)
\act\left(d^{-1/2}U_1^\ell y+\biasU{\ell}_1\right)
\right]
+
\delta_{jk}\varbW^2 .
\end{align*}
The pre-activation pair
\begin{equ}
\left(
d^{-1/2}U_1^\ell x+\biasU{\ell}_1,
d^{-1/2}U_1^\ell y+\biasU{\ell}_1
\right)
\end{equ}
is a centered $2$-dimensional Gaussian with covariance matrix
\begin{equ}
    \label{eq:covariance}
\Sigma(x,y)
=
\begin{pmatrix}
d^{-1}|x|^2+\varbU^2
&
d^{-1}\langle x,y\rangle+\varbU^2
\\
d^{-1}\langle x,y\rangle+\varbU^2
&
d^{-1}|y|^2+\varbU^2
\end{pmatrix}.
\end{equ}
We therefore define the neural-network covariance kernel in the ambient space $\mathbb R^d$ as 
\begin{equ}
\K_{\mathrm{amb}}(x,y)
:=
\E\left[\act(V_1)\act(V_2)\right]+\varbW^2,
\end{equ}
where \((V_1,V_2)\) is a centered Gaussian vector with covariance matrix
\(\Sigma(x,y)\). With this notation,
\begin{equ}
\label{eq:kernel-Rd}
\E\left[G_\ell^m(x)G_\ell^m(y)^\top\right]
=
\K_{\mathrm{amb}}(x,y)\Id .
\end{equ}
In particular, the covariance of \(G_\ell^m\) is independent of \(m\).

We define \(G_\ell\) to be the centered \(\mathbb R^d\)-valued Gaussian field
with the same covariance:
\begin{equ}
\E\left[G_\ell(x)G_\ell(y)^\top\right]
=
\K_{\mathrm{amb}}(x,y)\Id,
\qquad x,y\in\mathbb R^d.
\end{equ}
Thus \(G_\ell\) is the canonical ambient infinite-width Gaussian field
associated with the given MLP initialization.

Since \(\K_{\mathrm{amb}}\) is continuous and the components of the Gaussian field $G_\ell$ are independent, that field admits
local Karhunen-Lo\`eve decompositions on compact subsets of \(\mathbb R^d\).
In particular, by \cite{Adler2007Chapter3}, on every fixed ball \(B(0,R)\subset\mathbb R^d\), there exist
scalar-valued coefficients \((\sigma_k)_{k\ge1}\) such that on $B(0,R)$
\begin{equ}
G_\ell(x)=\sum_{k=1}^{\infty}\sigma_k(x)Z_k^\ell,
\end{equ}
where the \(Z_k^\ell\) are independent, $d$-dimensional standard Gaussian vectors, and for $x,y\in B(0,R)$
\begin{equ}
\sum_{k=1}^{\infty}\sigma_k(x)\sigma_k(y)
=
\K_{\mathrm{amb}}(x,y).
\end{equ}

We now explain how this ambient construction relates to the Gaussian field
introduced in Section~\ref{sec:intro}. If \(x,y\in\mathbb S^{d-1}\), then
\(|x|=|y|=1\), so the covariance matrix \(\Sigma(x,y)\) depends only on
\(\langle x,y\rangle\). Therefore the restriction $K$ of \(\K_{\mathrm{amb}}\) to
\(\mathbb S^{d-1}\) is zonal: for $x,y\in\mathbb S^{d-1}$
\begin{equ}
\K(x,y)=\kiso(\langle x,y\rangle)
\end{equ}
yielding the structure appearing in
\eqref{eq:zonal-kernel}. Consequently,  the Gaussian field introduced in
Section~\ref{sec:intro} is simply the restriction of the ambient field
\(G_\ell\) to the sphere.

In this case, the Karhunen-Loève expansion becomes explicit: as the zonal kernel \(\kiso(\langle x,y\rangle)\) admits a
Sch\"onberg-Gegenbauer expansion
\begin{equ}
\label{eq:gegenbauer_coefficient}
\kiso(\langle x,y\rangle)
=
\sum_{n=0}^\infty
\frac{C_n N_{n,d}}{\omega_{d-1}}
P_{n,d}(\langle x,y\rangle),
\end{equ}
where \(N_{n,d}\) is the dimension of the space of spherical harmonics of
degree \(n\) in dimension \(d\),
\begin{equ}
N_{n,d}
=
\binom{n+d-1}{n}
-
\binom{n+d-3}{n-2},
\end{equ}
\(C_n\) is the \(n\)-th Sch\"onberg-Gegenbauer coefficient of $\kiso$,
\(\omega_{d-1}\) is the surface volume of \(\mathbb S^{d-1}\), and \(P_{n,d}\)
is the normalized Gegenbauer polynomial (i.e., so that $P_{n,d}(1) = 1$). By the addition theorem for spherical harmonics $Y_{n,i}$, we finally have
\begin{equ}
\kiso(\langle x,y\rangle)
=
\sum_{n=0}^\infty\sum_{i = 1}^{N_{n,d}}
C_n Y_{n,i}(x)Y_{n,i}(y),
\end{equ}
yielding an explicit expression for the coefficients $\sigma_k$ in \eqref{eq:KL_expansion}.

\begin{remark}
    \label{r:schonberg}
    We note that the zonal covariance kernel $\kiso~:~[-1,1]\to \R$ is positive definite in the sense of \cite{schoenberg}\footnote{i.e., its Gram matrix $\kiso(\langle x_i, x_j\rangle)_{i,j}$ is positive semidefinite} so that the coefficients $C_n$ in, e.g., \eqref{eq:gegenbauer_coefficient} are all nonnegative.
\end{remark}

The following lemma collects the estimates on \(G_\ell^m\), \(G_\ell\), and
the coefficients \((\sigma_k)_{k\ge1}\) that will be used later.

\begin{lemma}
\label{lem:unif_gl}
 For every \(p\ge1\) there
exists a constant
\(\mathsf c>0\) such that for all \(m\in\mathbb N\), all \(\ell\), and all
\(x,y\in B(0,3)\),
\begin{equ}
\E\left[|G_\ell(x)|^p\right]
+
\E\left[|G_\ell^m(x)|^p\right]
\le
\mathsf c,
\end{equ}
\begin{equ}
\E\left[|G_\ell^m(x)-G_\ell^m(y)|^2\right]
+
\E\left[|G_\ell(x)-G_\ell(y)|^2\right]
\le
\mathsf c |x-y|^2,
\end{equ}
and, recalling \eqref{e:sigmaka}, for every $\alpha \in [d]$
\begin{equs}
&  \sum_{k=1}^{\infty}|\sigma_k(x)|^2 = \sum_{k=1}^{\infty}|\sigma_{k\alpha}(x)|^2 
\le
\mathsf c,\\
& \sum_{k=1}^{\infty}|\sigma_k(x)-\sigma_k(y)|^2 = \sum_{k=1}^{\infty}|\sigma_{k\alpha}(x)-\sigma_{k\alpha}(y)|^2
\le
\mathsf c |x-y|^2.
\end{equs}
\end{lemma}

\begin{proof}
Since \(x\in B(0,3)\), the pre-activation \(d^{-1/2}U_1^\ell x+\biasU{\ell}_1\)
is Gaussian with variance bounded by \(9d^{-1}+\varbU^2\). Because \(\act\) is
Lipschitz, it has at most linear growth, hence all moments of the corresponding
activation are uniformly bounded on \(B(0,3)\).

For the finite-width field, conditioning on the pre-activation vector
\begin{equ}
a(x)=\act\left(d^{-1/2}U^\ell x+\biasU{\ell}\right)\in\mathbb R^m,
\end{equ}
the vector \(G_\ell^m(x)\) is Gaussian with independent coordinates and
variance \(m^{-1}|a(x)|^2+\varbW^2\). Hence, for $p\geq 2$
\begin{equ}
\E\bigl[\E\bigl[|G_\ell^m(x)|^p\,\big|\,U^\ell,\biasU{\ell}\bigr]\bigl]
 \le
 \mathsf c_{p,d} \E\bigl[\bigl(m^{-1}|a|^2+\varbW^2\bigr)^{p/2}\bigr]
\le
\mathsf c_{p,d}\left(m^{-1}\sum_{j = 1}^m\E[|a(x)_j|^p] + \varbW^p\right),
\end{equ}
while the case $p \in [1,2)$ can be reduced to $p = 2$ by H\"older inequality. 
The uniform (in $x$ and $j$) moment bounds on the entries
of \(a(x)\) therefore give
\begin{equ}
\sup_{m\ge1}\sup_{x\in B(0,3)}\E\left[|G_\ell^m(x)|^p\right]\le \mathsf c_p.
\end{equ}
The same bound for \(G_\ell(x)\) follows from the boundedness of
\(\K_{\mathrm{amb}}(x,x)\) on compact sets.

For the increment estimate,
by independence and centering of \(W^\ell\), all cross terms vanish and
\begin{equs}
\E\left[|G_\ell^m(x)-G_\ell^m(y)|^2\right]
& = 
d\,\E\left[\left(\act\left(d^{-1/2}U_1^\ell x+\biasU{\ell}_1\right)
-
\act\left(d^{-1/2}U_1^\ell y+\biasU{\ell}_1\right)\right)^2\right]\\
& \leq d L_\act^2\E\left[\left(d^{-1/2}U_1^\ell(x-y)\right)^2\right]\\
&= L_\act^2 |x-y|^2.
\end{equs}
where in the second line we used the Lipschitz continuity of \(\act\).

Since \(G_\ell\) has the same covariance as \(G_\ell^m\), the same identity
gives
\begin{equ}
\E\left[|G_\ell(x)-G_\ell(y)|^2\right]
=
\E\left[|G_\ell^m(x)-G_\ell^m(y)|^2\right]
\le
L_\act^2|x-y|^2.
\end{equ}

Finally, on \(B(0,3)\), the coefficient estimates follow from the local
Karhunen-Lo\`eve representation. For every fixed \(\alpha\in[d]\),
\begin{equ}
G_\ell(x)_\alpha=\sum_{k=1}^{\infty}\sigma_k(x)Z_{k,\alpha}^\ell,
\end{equ}
where the \(Z_{k,\alpha}^\ell\) are independent standard Gaussians. Therefore,
\begin{equ}
\sum_{k=1}^{\infty}|\sigma_k(x)|^2
=
\E\left[|G_\ell(x)_\alpha|^2\right] = \frac 1 d\E\left[|G_\ell(x)|^2\right] \leq \mathsf c ,
\end{equ}
and 
\begin{equ}
\sum_{k=1}^{\infty}|\sigma_k(x)-\sigma_k(y)|^2
=
\E\left[|G_\ell(x)_\alpha-G_\ell(y)_\alpha|^2\right] = \frac 1 d \E\left[|G_\ell(x)-G_\ell(y)|^2\right]\leq \mathsf c |x-y|^2. 
\end{equ}
Since $|\sigma_{k\alpha}(x)| = |\sigma_{k}(x)|$ and $|\sigma_{k\alpha}(x) - \sigma_{k\alpha}(y)| = |\sigma_{k}(x)-\sigma_{k}(y)|$  this proves the lemma.
\end{proof}

The following proposition builds on the main result from \cite{basteri_trevisan} and provides a bound on the difference between the MLP $G^m$ and the limiting field $G$ which is uniform on the input $x \in \S^{d-1}$.

\begin{proposition}
\label{prop:basteri_trevisan}
For any $N \in \mathbb N$ there exists a constant $\mathsf c_{N}> 0$ 
%not depending on $m,L$ 
such that for all $\mathbf x = (x_1, \dots, x_N) \in (\S^{d-1})^{N}$ and all $m,L \in \mathbb N$ there exists a coupling $\Gamma$ between $G_\ell^m$ and $G_\ell$ such that we have: 
\begin{equ}
 \sup_{i \in [N]} \mathbb{E}_\Gamma \left[|G^m_\ell(x_i)-G_\ell(x_i) |^2\right] \leq \mathsf c_{N} \frac 1 m\,.
\end{equ} 
Assuming further that $\act \in C^\infty(\R)$ we have (with a potentially different coupling $\Gamma$):
\begin{equ} \sup_{i \in [N]} \mathbb{E}_\Gamma \left[|G^m_\ell(x_i)-G_\ell(x_i) |^2\right] \leq \mathsf c \frac 1 {m^{1/4}}\,.
\end{equ}
\end{proposition}
\begin{proof}It is enough to prove the bounds with the infimum. Indeed, if
\begin{equ}
\inf_\Gamma \sup_{i \in [N]} \E_\Gamma \left[|G^m_\ell(x_i)-G_\ell(x_i)|^2\right]
\leq \mathsf c_N \frac1m,
\end{equ}
then there exists a coupling \(\Gamma\) such that
\begin{equ}
\sup_{i \in [N]} \E_\Gamma \left[|G^m_\ell(x_i)-G_\ell(x_i)|^2\right]
\leq 2\mathsf c_N \frac1m.
\end{equ}
Renaming the constant, this yields the desired statement (also for the $N$-independent constant).

We start with the first bound. For readability, we omit the dependence on \(\ell\). For \(i\in[N]\), let
\begin{equ}
g(x_i):=\frac{1}{\sqrt m}W\act\left(\frac{1}{\sqrt d}Ux_i+b_U\right),
\qquad
G'(x_i):=G(x_i)-b_W,
\end{equ}
so that
\begin{equ}
G^m(x_i)=g(x_i)+b_W,
\qquad
G(x_i)=G'(x_i)+b_W,
\end{equ}
with the same independent Gaussian bias \(b_W\sim\mathcal N(0,\varbW^2\Id)\).
Thus, it is enough to couple the vectors
\begin{equ}
\mathbf g:=\bigl(g(x_1),\dots,g(x_N)\bigr)\in(\mathbb R^d)^N,
\qquad
\mathbf G':=\bigl(G'(x_1),\dots,G'(x_N)\bigr)\in(\mathbb R^d)^N.
\end{equ}
To this end, set
\begin{equ}
a_i:=\act\left(\frac{1}{\sqrt d}Ux_i+b_U\right)\in\mathbb R^m,
\qquad i\in[N].
\end{equ}
Conditioned on \((a_1,\dots,a_N)\), the vector \(\mathbf g\) is centered Gaussian
in \((\mathbb R^d)^N\) with block covariance
\begin{equ}
\Sigma_{\mathbf g\mid a}
=
\left(\frac1m\langle a_i,a_j\rangle \Id\right)_{i,j=1}^N,
\end{equ}
while \(\mathbf G'\) is centered Gaussian with covariance
\begin{equ}
\Sigma_{\mathbf G'}
=
\left((\mathbf K_0)_{ij} \Id\right)_{i,j=1}^N \qquad \text{where}\qquad (\mathbf K_0)_{ij} := \E[\act(U_1 x_i/\sqrt d + b)\act(U_1 x_j/\sqrt d + b)].
\end{equ}
for $U_1 \sim \mathcal N(0, \Id)$ and $b \sim \mathcal N(0,\sigma_U^2)$.

Recall from \cite{villani_transport, trevisan} that for Gaussian variables \(Z_1 \sim \mathcal N(\mu_1,\Sigma_1)\) and
\(Z_2 \sim \mathcal N(\mu_2,\Sigma_2)\) in \(\mathbb R^M\),
\begin{equ}
\label{eqaux_gaussian}
    \W_2(Z_1,Z_2) \le \|\mu_1-\mu_2\| + \|\sqrt{\Sigma_1}-\sqrt{\Sigma_2}\|_F\,,
\end{equ}
where $\|\cdot\|_F$ denotes the Frobenius norm.
 Applying \eqref{eqaux_gaussian} with
\(M=dN\) and \(\mu_1=\mu_2=0\), we obtain
\begin{equ}
\label{eqaux_bt_N}
\W_2^2(\Law(\mathbf g\mid a),\Law(\mathbf G'))
\le
\bigl\|\Sigma_{\mathbf g\mid a}^{1/2}-\Sigma_{\mathbf G'}^{1/2}\bigr\|_F^2 .
\end{equ}

We now check that the right-hand side of \eqref{eqaux_bt_N} is uniformly bounded for
\(x_1,\dots,x_N\in\S^{d-1}\). 
Since each \(x_i\in\S^{d-1}\), all pre-activations have the same variance
\(d^{-1}+\varbU^2\), so all moments entering the estimate are uniform in the
choice of the points. In particular, the dependence of the constant is only on
\(N\), not on the positions of the \(x_i\) on the sphere. Applying
\cite[Lemma 3.4]{basteri_trevisan} to the empirical Gram matrix
\begin{equ}
\left(\frac1m\langle a_i,a_j\rangle\right)_{i,j=1}^N
\end{equ}
and its limit \(\mathbf K_0\), integrating on the activations we obtain
\begin{equ}
\label{eqaux_bt_3}
\E\!\left[\bigl\|\Sigma_{\mathbf g\mid a}^{1/2}-\Sigma_{\mathbf G'}^{1/2}\bigr\|_F^2\right]
\le \frac{\mathsf c_N}{m},
\end{equ}
for a constant \(\mathsf c_N>0\) depending only on \(N\), \(d\), \(\act\),
\(\varbU\), and \(\varbW\) (and not on $\{x_i\}$).

Regarding the second statement, we recall the definition of the $C^0(\mathbb{ U})$ norm from \cite[Equation 3.17]{favaro2025quantitative}: for a vector-valued function $f=(f_1,\dots,f_d)$,
$$\|f\|_{C^0( \mathbb{U})}=\max_{i\in[d]}\max_{x\in\mathbb{\bar U}}|f_i(x)|$$ 
so that, setting $\mathbb{U} = \S^{d-1}$\,,
\begin{equs}
\inf_\Gamma \sup_{i\in[N]}\E_\Gamma[|G_\ell^m(x_i)-G_\ell(x_i)|^2] & \leq \inf_\Gamma \sup_{x \in \S^{d-1}}\E_\Gamma[|G_\ell^m(x)-G_\ell(x)|^2]\\
& \leq d \inf_\Gamma \E_\Gamma[ \|G_\ell^m-G_\ell\|_{C^0(\mathbb{S}^{d-1})}^2].
\end{equs}
Since the activation function $\act\in C^\infty(\R)$ by assumption, by defining the metric
$$W_{\infty,k}(X,Y)=\left(\inf_\Gamma \E_\Gamma[\|X-Y\|_{C^k(\mathbb{U})}^2]\right)^{1/2}\,,$$
we bound  the RHS of the above as
$$\inf_\Gamma \sup_{i\in[N]}\E_\Gamma[|G_\ell^m(x_i)-G_\ell(x_i)|^2] \leq d W_{\infty,0}(G_\ell^m,G_\ell)^2 \leq \mathsf c \frac1{m^{1/4}},$$
where  the last inequality follows directly from \cite[Theorem 3.16]{favaro2025quantitative}, proving our claim.
\end{proof}

With the above result at hand, we are finally ready to couple the stochastic processes appearing in our analysis.
\begin{remark}
\label{rem:adapted_coupling_convention}
From now on all estimates involving the finite-width fields \(G_\ell^m\), the limiting Gaussian field \(G_\ell\)
and the Brownian motions \(B^k_t\) are understood under the following adapted
construction.

Fix \(m,L,N\in\mathbb N\), set \(\Delta t=T/L\), and let
\(t_\ell=\ell\Delta t\). We first choose independent \(\mathbb R^d\)-valued
Brownian motions \((B^k)_{k\ge1}\). For every layer \(\ell=0,\dots,L-1\), we
define the limiting Gaussian field associated with the Brownian increment by
\begin{equ}
G_{\ell+1}(x)
:=
\sum_{k=1}^{\infty}\sigma_k(x)
\frac{B^k_{t_{\ell+1}}-B^k_{t_\ell}}{\sqrt{\Delta t}},
\qquad x\in \mathbb S^{d-1}.
\end{equ}
Then \(G_{\ell+1}\) has the same law as the infinite-width Gaussian field
introduced in Section~\ref{s:noise}. 

We now construct the fields \(G_{\ell}^m\) and the natural filtration $(\mathcal F_\ell)_\ell$ recursively over
\(\ell \in \{1,\dots,L\}\). We start by defining
\(\mathcal F_0 := \sigma(\mathbf X_0)\). Assume that, up to layer \(\ell\),
the fields \((G_r,G_r^m)_{1\le r\le \ell}\) and the token configuration
\(\mathbf X_\ell\) have been constructed, and set
\begin{equ}
\mathcal F_\ell
:=
\sigma\!\left(
\mathbf X_0,\,(G_r,G_r^m)_{1\le r\le \ell}
\right).
\end{equ}
so that  \(\mathbf X_\ell\) is \(\mathcal F_\ell\)-measurable. 
The finite-width field \(G_{\ell+1}^m\) is then chosen conditionally on $\mathcal F_\ell$ and on \(G_{\ell+1}\) using the 
coupling of Proposition~\ref{prop:basteri_trevisan}, applied to the (conditionally)
deterministic tuple \(\mathbf x = \mathbf X_\ell\). Equivalently, given such a coupling and disintegrating with respect to the second marginal, we sample
\(G_{\ell+1}^m\) conditionally on \(G_{\ell+1}\) and \(\mathbf X_\ell\).
We then define \(\mathbf X_{\ell+1}\) by the transformer update and enlarge the
filtration by setting
\begin{equ}
\mathcal F_{\ell+1}
:=
\sigma\!\left(
\mathcal F_\ell,G_{\ell+1},G_{\ell+1}^m
\right).
\end{equ}

Thus, with respect to the natural layer filtration \((\mathcal F_\ell)_\ell\),
the fields have the prescribed marginals and satisfy, for every
\(\ell=0,\dots,L-1\),
\begin{equ}
\mathbb E\left[
\frac1N\sum_{i=1}^N
\left|
G_{\ell+1}^m(X_\ell^i)-G_{\ell+1}(X_\ell^i)
\right|^2
\,\middle|\,\mathcal F_\ell
\right]
\le
\frac{\mathsf c_N}{m}
\qquad\text{a.s.}
\end{equ}
If \(\act\in C^\infty(\mathbb R)\), the same convention is used with the
functional-level coupling estimate, giving instead
\begin{equ}
\mathbb E\left[
\frac1N\sum_{i=1}^N
\left|
G_{\ell+1}^m(X_\ell^i)-G_{\ell+1}(X_\ell^i)
\right|^2
\,\middle|\,\mathcal F_\ell
\right]
\le
\frac{\mathsf c}{m^{1/4}}
\qquad\text{a.s.},
\end{equ}
with \(\mathsf c\) independent of \(N\). All expectations below are taken with
respect to this joint adapted realization.
\end{remark}

\subsection{Reduction to single-step dynamics}
\label{sec:single-step}

In this section we reduce the dynamics in \eqref{eq:transformer-update} to a single-step update, specifically a classical Euler-Maruyama scheme as in \eqref{e:xhat}. As the resulting discrete dynamics does not preserve $\mathbb S^{d-1}$ almost surely, even if $\hat {\mathbf X}_0 \in (\S^{d-1})^N$, the first step in this process is to define suitable, sufficiently regular extensions to $\mathbb R^d$ of the operators defining our original dynamics.

We introduce the cutoff function $\rho(x) \coloneqq \varphi(|x|)$, where $\varphi:\R_{\ge 0}\to\R$ is a smooth mollifier satisfying:
\begin{equation*}
\varphi(r) \coloneqq
\begin{cases}
1 & \text{if } r \leq 3/2,\\
0 & \text{if } r \geq 2.
\end{cases}
\end{equation*}
We also define a smooth interpolation function $T:\R^d\to \R^d$ such that:
\begin{equation*}
T(x) = 
\begin{cases}
    x & \text{if } x\in B(0,2),\\
    3x/|x| & \text{if } x\in B(0,3)^c,
\end{cases}
\end{equation*}
and, by abuse of notation, we write $T(x_1,\dots,x_N) \coloneqq (T(x_1),\dots,T(x_N))$. We then replace the empirical vector field $\attN{}{x}$ with its regularized counterpart: 
\begin{equation*}
\rattN{}{x} \coloneqq \attN{}{T(x)}.
\end{equation*}

Finally,  to ensure notational consistency across the particle dynamics, we distinguish between the true drift $b$ (defined on $\mathbb{S}^{d-1}$) and the regularized drift $b^\rho$ (defined on $\mathbb{R}^d$) evaluated for a particle $x^i$ interacting with the system state $\mathbf x \in (\mathbb{R}^d)^N$:
\begin{align}
  \rproj{x^i} &\coloneqq \rho(x^i)\proj{x^i}, \\
    \driftN{}{x}&\coloneqq \proj{x^i} \attN{}{x} - \kiso \frac{d-1}{2} x^i, \\
    \rdriftN{}{x} &\coloneqq \rproj{x^i} \rattN{}{x} - \rho(x^i)^2 K_{\text{amb}}(x^i,x^i) \frac{d-1}{2} x^i.
\end{align}

We now proceed to ensure that both the drift and diffusion terms in our idealized dynamics remain well-behaved everywhere. 
\begin{lemma}
\label{lem:regularity_projected_coefficients}
For every $N \in \mathbb N$ there exists a constant\footnote{depending only on \(d\) and the
constants in Assumption~\ref{ass:att_reg}} \(\mathsf c>0\) such that for every
\(\mathbf x,\mathbf y\in(\mathbb R^d)^N\),
\begin{equ}
\frac1N\sum_{i=1}^N
|\rdriftN{}{\mathbf x}-\rdriftN{}{\mathbf y}|^2
\le
\mathsf c\,
\frac1N\sum_{i=1}^N |x^i-y^i|^2,
\end{equ}
and for every $x,y \in \mathbb R^d$
\begin{equ}
\sum_{k=1}^{\infty} |\rproj{x}\sigma_k(x)|^2 \le \mathsf c,
\qquad
\sum_{k=1}^{\infty} |\rproj{x}\sigma_k(x)-\rproj{y}\sigma_k(y)|^2
\le
\mathsf c |x-y|^2.
\end{equ}
\end{lemma}

\begin{proof}
The proof combines triangle inequality with the boundedness and Lipschitz continuity of $\rho, T$ and $P$ (in operator norm) to the regularity in $B(0,3)$ of the underlying
drift and diffusion operators derived in Lemma~\ref{lem:unif_gl} and required in Assumption~\ref{ass:att_reg}. To improve readability, this technical proof is deferred to Appendix~\ref{s:pl24}.
\end{proof}

We now compare the true discrete transformer update \eqref{eq:transformer-update} with the idealized single-step Euler-Maruyama dynamics $\hat{\mathbf{X}}$ given by:
\begin{equ}
\label{eq:dynamics_hat}
\begin{cases}
    \hat{X}^i_{\ell+1} = \hat{X}^i_{\ell} + \dt\, \rdriftN{\ell}{\hat{X}} + \sqrt{\dt}\, \rproj{\hat{X}^i_\ell} G_{\ell+1}(\hat{X}^i_\ell), \\
    \hat{X}^{i}_0 = X^{i}_0.
\end{cases}
\end{equ}

The remainder of this subsection is devoted to the proof of the following result. Recall that all expectations are taken on the coupling defined in  Remark~\ref{rem:adapted_coupling_convention}.

\begin{lemma}
\label{lem:strong_error}
    Under Assumptions \ref{ass:att_reg} and \ref{ass1}, for every $N\in \mathbb{N}$, there exists a constant $\cN > 0$ such that for every $m, L \in \mathbb N$ and $T>0$:
    \begin{equ}
        \E\left[\sup_{0\leq k\leq L} \frac{1}{N}\sum_{i=1}^N |{X}_k^i-\hat{{X}}^i_k|^2\right] \leq \cN \left(\frac{1}{L}+\frac{1}{m}\right) e^{\cN T}.
    \end{equ}
Assuming further that $\act \in C^\infty(\R)$ there exists a $m, L, N$-independent constant $\mathsf c$ such that 
        \begin{equ}
       \E\left[\sup_{0\leq k\leq L} \frac{1}{N}\sum_{i=1}^N |{X}_k^i-\hat{{X}}^i_k|^2\right] \leq \mathsf c \left(\frac{1}{L}+\frac{1}{m^{1/4}}\right) e^{\mathsf c T}.
    \end{equ}
\end{lemma}

In preparation for the proof of the above result, we summarize in Lemma~\ref{lem:norm_res} below a convenient decomposition of the 2-step update of the original transformer dynamics \eqref{eq:transformer-update} as a single-step dynamics plus a controlled remainder term. 

\begin{lemma}
\label{lem:norm_res}
Under the previous assumptions, if $\dt \leq (2\mathsf c)^{-1}$:
\begin{align*}
X^i_{\ell+1} &= X^i_{\ell} + \dt P_{X^{i}_\ell}\attN{\ell}{X} + {\sqrt{\dt}} P_{X^i_{\ell}} G_{\ell+1}^m(X^i_\ell) \\
&\quad - \frac{{\dt}}{2}\left| P_{X^i_{\ell}} G_{\ell+1}^m(X^i_\ell) \right|^2 X^i_{\ell} - {\dt} ((X^i_{\ell})^\top G_{\ell+1}^m(X^i_\ell)) P_{X^i_{\ell}} G_{\ell+1}^m(X^i_\ell) + \mathcal{R}_\ell^i,
\end{align*}
where the total remainder $\mathcal{R}_\ell^i$ is bounded by:
\begin{align}
\label{eq:bound_Ri}
|\mathcal{R}^i_\ell| \leq   12\mathsf c\dt^2 + \mathsf c|\mathbf{w}|^3  + 2|\mathbf{w}| \mathsf c \dt + \frac{9}{2}|\mathbf{w}|^2 \mathsf c \dt +  (1 + 3|\mathbf{w}|)|\mathbf{w} - \mathbf{w}'| + \frac{3}{2}|\mathbf{w} - \mathbf{w}'|^2,
\end{align}
with the noise increments defined as $\mathbf{w} := \sqrt{\dt}G_{\ell+1}^m (Y^i_{\ell})$, $\mathbf{w}' := \sqrt{\dt}G_{\ell+1}^m (X^i_{\ell})$
\end{lemma}

\begin{proof}[Proof of Lemma~\ref{lem:norm_res}]
    This proof, resulting from a careful Taylor expansion of the original layer update, is deferred to Appendix~\ref{app:taylor}.
\end{proof}

We proceed to prove the main estimate of this section:

\begin{proof}[Proof of Lemma~\ref{lem:strong_error}]
To prove the main result of this section, we analyze the total error by tracking the one-step increments $\Delta X_j^i = X_{j+1}^i - X_j^i$ and $\Delta \hat{X}_j^i = \hat{X}_{j+1}^i - \hat{X}_j^i$. By applying Lemma \ref{lem:norm_res} to expand the layer normalization, the increment of the true particles is given by:
\begin{align*}
    \Delta X_j^i &= \dt \proj{X^i_j}\attN{j}{X} + \sqrt{\dt} \proj{X^i_j} G_{j+1}^m (X^i_j) \\
    &\quad - \frac{\dt}{2}\left| \proj{X^i_j} G_{j+1}^m (X^i_j) \right|^2 X^i_j - \dt \left((X^i_j)^\top G_{j+1}^m (X^i_j)\right) \proj{X^i_j} G_{j+1}^m (X^i_j) + \mathcal{R}^i_j.
\end{align*}
To bridge the gap between this discrete update and the idealized continuous dynamics, we define an auxiliary increment $\Delta \tilde{X}_j^i$ that replaces the finite-width neural network noise $G^m$ with the exact isotropic Gaussian process $G$:
\begin{align*}
    \Delta \tilde{X}_j^i &:= \dt \proj{X^i_j}\attN{j}{X} + \sqrt{\dt} \proj{X^i_j} G_{j+1} (X^i_j) \\
    &\quad - \frac{\dt}{2}\left| \proj{X^i_j} G_{j+1} (X^i_j) \right|^2 X^i_j - \dt \left((X^i_j)^\top G_{j+1} (X^i_j)\right) \proj{X^i_j} G_{j+1} (X^i_j).
\end{align*}

We recall from Remark~\ref{rem:adapted_coupling_convention} the definition of the filtration \((\mathcal F_j)_{j=0,\dots,L}\), namely
\begin{equ}
\mathcal F_j:=\sigma\!\left(\mathbf{X}_0,\,(G_r^m,G_r)_{1\le r\le j}\right),
\end{equ}
and that \(\mathbf X_j\) is \(\mathcal F_j\)-measurable. We further recall that, conditionally on $\mathcal F_j$, with $x=X_j^i$, the variables $G_{j+1}^m(x)$ and $G_{j+1}(x)$ are centered and have the same covariance $\K(x,x)\Id$.

Using the inequality $(a+b+c)^2 \le 3a^2 + 3b^2 + 3c^2$, we decompose the cumulative expected supremum $Z_{\ell+1}$ into three terms:
\begin{align}
    Z_{\ell+1} &:= \E\left[\sup_{0\leq k\leq \ell+1} \frac{1}{N}\sum_{i=1}^N \left|\sum_{j=0}^{k-1} (\Delta X_j^i - \Delta \hat{X}_j^i)\right|^2\right] \nonumber \\
    &\leq 3\underbrace{\E\left[\sup_{0\leq k\leq \ell+1} \frac{1}{N}\sum_{i=1}^N \left|\sum_{j=0}^{k-1} (\Delta \tilde{X}_j^i - \Delta \hat{X}_j^i)\right|^2 \right]}_{\text{(A) }} \nonumber \\
    &\quad + 3\underbrace{\E\left[\sup_{0\leq k\leq \ell+1} \frac{1}{N}\sum_{i=1}^N \left|\sum_{j=0}^{k-1} (\Delta X_j^i - \Delta \tilde{X}_j^i - \mathcal{R}_j^i)\right|^2 \right]}_{\text{(B) }} \nonumber \\
    &\quad + 3\underbrace{\E\left[\sup_{0\leq k\leq \ell+1} \frac{1}{N}\sum_{i=1}^N \left|\sum_{j=0}^{k-1} \mathcal{R}_j^i\right|^2 \right]}_{\text{(C) }}.
\end{align}
We now proceed to bound the three terms on the RHS separately. \\[10pt]
\textbf{Bounding (C):} Using the Cauchy-Schwarz inequality, we can bound the cumulative Taylor remainder as follows
\begin{equation*}
    (C) \leq L \sum_{j=0}^L \frac{1}{N}\sum_{i=1}^N \E\left[|\mathcal{R}_j^i|^2\right] \leq L (L \mathsf c \dt^3) \leq \mathsf c T^2 \dt,
\end{equation*}
where the second inequality is justified in Lemma \ref{lem:squared_reminder} below to improve readability.\\[10pt]
\noindent \textbf{Bounding (B):} 
By definition of the filtration \((\mathcal F_j)_{j=0,\dots,L}\), 
we note that, since $G_{j+1}^m(x)$ and $G_{j+1}(x)$ are both centered with the same covariance, for any $x \in \S^{d-1}$ we have:
\begin{equ}
\E[\proj{x}(G_{j+1}^m(x)-G_{j+1}(x))\vert \mathcal F_j]=0,\qquad
\E\left[| \proj{x}G_{j+1}^m(x)|^2-| \proj{x}G_{j+1}(x)|^2 \,\vert \mathcal F_j\right]=0,
\end{equ}
and
\begin{equ}
\E\!\left[(x^\top G_{j+1}^m(x))\proj{x}G_{j+1}^m(x) \,\vert \mathcal F_j\right]
=\proj{x}\E[G_{j+1}^m(x)(G_{j+1}^m(x))^\top\vert \mathcal F_j]x
=\K(x,x)\proj{x}x=0,
\end{equ}
with the same identity for $G_{j+1}(x)$. Therefore, defining  $W_j^i := \Delta X_j^i - \Delta \tilde{X}_j^i - \mathcal{R}_j^i$, we get $\E[W_j^i\mid \mathcal F_j]=0$, so $(W_j^i)_j$ is a martingale difference sequence with respect to the layer filtration $\mathcal F_j$. By Doob's maximal inequality:
\begin{align*}
    (B) &\leq \frac{4}{N}\sum_{i=1}^N \E\left[ \left|\sum_{j=0}^{\ell} W_j^i\right|^2 \right] = \frac{4}{N}\sum_{i=1}^N \sum_{j=0}^{\ell} \E\left[|W_j^i|^2\right] \\
    &\leq \mathsf c \dt \sum_{j=0}^{\ell} \frac{1}{N}\sum_{i=1}^N \E\left[\E\left[|G_{j+1}^m(X_j^i) - G_{j+1}(X_j^i)|^2|\mathcal F_j\right]\right] + O(\dt^2) \\
    &\leq \mathsf c_N L \dt \left(\frac{1}{m} + \dt\right) \leq \mathsf c_N T \left(\frac{1}{m} + \dt\right),
\end{align*}
where we utilized the coupling between $G$ and $G^m$ from Proposition \ref{prop:basteri_trevisan},
yielding a constant $\mathsf c_N$ depending on $N$.

Under the further assumption that $\act$ is of class $C^\infty$ on its domain, the second statement in Proposition \ref{prop:basteri_trevisan} can be used to obtain a bound uniform in $N$:
\begin{equ}
\label{eqaux_uniform_B}
    (B) \leq\mathsf c T \left(\frac{1}{m^{1/4}} + \dt\right).
\end{equ}
\textbf{Bounding (A):} We further split $\Delta \tilde{X}_j^i - \Delta \hat{X}_j^i$ into a predictable conditional drift $D_j^i$ and a martingale difference $M_j^i$:
\begin{equation*}
    D_j^i := \E[\Delta \tilde{X}_j^i - \Delta \hat{X}_j^i \mid \mathcal{F}_j] \quad \text{and} \quad M_j^i := \Delta \tilde{X}_j^i - \Delta \hat{X}_j^i - D_j^i.
\end{equation*}
For the exact continuous dynamics, $\E[G_{j+1}(x) ] = 0$. Furthermore, utilizing the isotropic property of the kernel ($\K(x,x) = \kiso \Id$), we evaluate the conditional expectations of the correction terms:
\begin{align*}
    \E\left[| \proj{x} G_{j+1}(x) |^2 \vert \mathcal F_j\right] &= \Tr(\proj{x}\K(x,x)) = \kiso \Tr(I - x x^\top) = \kiso(d-1), \\
    \E\left[(x^\top G_{j+1}(x)) \proj{x} G_{j+1}(x)\vert \mathcal F_j \right] &= \proj{x}\K(x,x) x = \kiso(I - x x^\top) x = 0.
\end{align*}
Substituting these into the expectation of $\Delta \tilde{X}_j^i$, the correction precisely yields the drift term $-\kiso (d-1) X_j^i/2$. Since the true particles $X_j^i$ lie strictly on the sphere $\S^{d-1}$, the cutoff function evaluates to $\rho(X_j^i) = 1$, implying the true drift and regularized drift coincide ($\driftNinline{j}{X} = \rdriftNinline{j}{X}$). Thus:
\begin{equation*}
    D_j^i = \dt\left( \rdriftN{j}{X} - \rdriftN{j}{\hat{X}} \right).
\end{equation*}
Using Cauchy-Schwarz and the uniform Lipschitz property of the regularized drift (Lemma \ref{lem:regularity_projected_coefficients}):
\begin{align*}
    \E\left[\sup_{0\leq k\leq \ell+1} \frac{1}{N}\sum_{i=1}^N \left|\sum_{j=0}^{k-1} D_j^i\right|^2\right] &\leq L \sum_{j=0}^{\ell} \frac{1}{N}\sum_{i=1}^N \E\left[|D_j^i|^2\right] \\
    &\leq L \dt^2 \mathsf c \sum_{j=0}^{\ell} \E\left[\frac{1}{N}\sum_{i=1}^N |X_j^i - \hat{X}_j^i|^2\right] \leq \mathsf c T \dt \sum_{j=0}^{\ell} Z_j.
\end{align*}
Since $M_j^i$ is a martingale difference sequence by construction ($\E[M_j^i \mid \mathcal{F}_j] = 0$), Doob's maximal inequality yields:
\begin{align*}
    \E\left[\sup_{0\leq k\leq \ell+1} \frac{1}{N}\sum_{i=1}^N \left|\sum_{j=0}^{k-1} M_j^i\right|^2\right] &\leq 4 \sum_{j=0}^{\ell} \frac{1}{N}\sum_{i=1}^N \E\left[|M_j^i|^2\right] \\
    &\leq \mathsf c \dt \sum_{j=0}^{\ell} \E\left[\frac{1}{N}\sum_{i=1}^N \left|\proj{X_j^i} G_{j+1}(X_j^i) - \rproj{\hat{X}_j^i} G_{j+1}(\hat{X}_j^i)\right|^2\right] + \mathsf c T \dt \\
    &\leq \mathsf c \dt \sum_{j=0}^{\ell} Z_j + \mathsf c T \dt,
\end{align*}
where the final inequality follows from the uniform boundedness of the projections together with the Lipschitz bound $\sum_{k\geq 1}|\rproj{x}\sigma_k(x)-\rproj{y}\sigma_k(y)|^2\leq \mathsf c |x-y|^2$ for the limiting coefficient family, proved in Lemma \ref{lem:regularity_projected_coefficients}.\\[10pt]
\textbf{Conclusion:} Combining the bounds for (A), (B), and (C), we obtain a recursive bound on the cumulative error:
\begin{equation*}
    Z_{\ell+1} \leq \mathsf c_N\left( \dt + \frac{1}{m} \right) + \mathsf c \dt \sum_{j=0}^{\ell} Z_j.
\end{equation*}
Because $Z_0 = 0$ by the matched initial conditions, applying the discrete Gr\"onwall lemma yields the final bound:
\begin{equation*}
    Z_L \leq \mathsf c_N\left(\dt + \frac{1}{m}\right) e^{\mathsf c_NT} = \mathsf c_N\left(\frac{T}{L} + \frac{1}{m}\right) e^{\mathsf c_NT} \leq \mathsf c_N\left(\frac{1}{L} + \frac{1}{m}\right) e^{\mathsf c_NT}.
\end{equation*}

    Under the assumption that the activation $\act$ is in $C^\infty(\R)$, we leverage our uniform in $N$ bound on (B) \eqref{eqaux_uniform_B} to obtain the second claim in this result:
    \begin{equ}
    Z_L \leq \mathsf c\left(\dt + \frac{1}{m^{1/4}}\right) e^{\mathsf cT} = \mathsf c\left(\frac{T}{L} + \frac{1}{m^{1/4}}\right) e^{\mathsf cT} \leq \mathsf c\left(\frac{1}{L} + \frac{1}{m^{1/4}}\right) e^{\mathsf cT}.
\end{equ}
\end{proof}

A central challenge in this expansion is the singularity of the map $v \mapsto v/|v|$ at the origin. We now provide a set of general geometric results to bound the remainders of the normalization map, ensuring that the noise injected by the MLP layers does not push the particles into unstable regions.

\begin{lemma}
\label{lem:squared_reminder}
    For the remainder $\mathcal{R}_\ell^i$ defined in Lemma \ref{lem:norm_res}, there exists a constant $\mathsf c > 0$ such that for all $\ell \in \{1,  \dots, L\}$ and $i \in \{1, \dots, N\}$ and all $m \in \mathbb N$:
    \begin{equ}
        \E\left[|\mathcal{R}_\ell^i|^2\right] \leq \mathsf c\dt^3.
    \end{equ}
\end{lemma}
\begin{proof}
    The bound follows directly from the estimates in Equation \eqref{eq:bound_Ri}. By applying the elementary inequality $(a_1 + \dots + a_k)^2 \leq k(a_1^2 + \dots + a_k^2)$, we can bound the second moment of the total remainder by bounding the second moment of each constituent term.
    
    Using the boundedness of the moments of the Gaussian process $G_{\ell+1}^m$ (Lemma \ref{lem:unif_gl}) and recalling that $\mathbf{w} = \sqrt{\dt} G_{\ell+1}^m(Y_{\ell}^i)$ implies $\E[|\mathbf{w}|^p] \leq \mathsf c_p \dt^{p/2}$, we have:
    \begin{align*}
        \mathsf c \E\left[|\mathbf{w}|^6\right] \leq \mathsf c\dt^3 \,,
\qquad         \E\left[ \left( 2|\mathbf{w}| \mathsf c\dt + \frac{9}{2}|\mathbf{w}|^2 \mathsf c \dt \right)^2 \right] \leq \mathsf c\dt^3 \,.
    \end{align*}
    
    For the noise substitution error involving $\delta := \sqrt{\dt}(G_{\ell+1}^m (Y^i_{\ell}) - G_{\ell+1}^m (X^i_{\ell}))$, we rely on the uniform Lipschitz property of $G_{\ell+1}^m$ and the fact that $|Y^i_{\ell}- X^i_{\ell}| \leq \mathsf c\dt$ (from Eq. \eqref{eq:Y_X_E_ATT}). This yields:
    \begin{equ}
        \E[|\delta|^2] = \dt \E\left[|G_{\ell+1}^m (Y^i_{\ell}) - G_{\ell+1}^m (X^i_{\ell})|^2\right] \leq \dt \E\left[ L^2 |Y^i_{\ell} - X^i_{\ell}|^2 \right] \leq \mathsf c \dt^3.
    \end{equ}
    Moreover, if we set
    \begin{equ}
        b := \act\left(\frac{1}{\sqrt d}U^{\ell+1}Y^i_{\ell} + \biasU{\ell+1}\right)
        - \act\left(\frac{1}{\sqrt d}U^{\ell+1}X^i_\ell + \biasU{\ell+1}\right),
    \end{equ}
    then conditionally on $b$ the difference $G_{\ell+1}^m(Y^i_{\ell})-G_{\ell+1}^m(X^i_\ell)$ is centered Gaussian with covariance $\frac1m |b|^2 \Id$, since the output bias cancels. Hence
    \begin{equ}
        \E[|\delta|^4 \mid b]
        = \mathsf c_d \dt^2 \left(\frac1m |b|^2\right)^2.
    \end{equ}
    Using Jensen's inequality and the Lipschitz continuity of $\act$,
    \begin{equ}
        \E\left[\left(\frac1m |b|^2\right)^2\right]
        = \E\left[\left(\frac1m \sum_{j=1}^m b_j^2\right)^2\right]
        \leq \frac1m \sum_{j=1}^m \E[|b_j|^4]
        \leq \mathsf c |Y^i_{\ell}-X^i_\ell|^4
        \leq \mathsf c \dt^4.
    \end{equ}
    Therefore $\E[|\delta|^4] \leq \mathsf c\dt^6$. Combining these bounds for the final term yields:
    \begin{equ}
        \E\left[ \left( (1 + 3|\mathbf{w}|)|\delta| + \frac{3}{2}|\delta|^2 \right)^2 \right] \leq \mathsf c\dt^3.
    \end{equ}
    Summing these individual bounds concludes the proof.
\end{proof}

\subsection{Strong convergence of the single step dynamics}
\label{sec:strong-single-step}

We now recall the three processes entering the final comparison: the discrete approximation scheme $(\hat X_\ell)_{\ell=0}^L$ defined in \eqref{eq:dynamics_hat}, its continuous-time interpolation $(\hat X_t)_{t\in[0,T]}$, and the limiting diffusion $(\bar X_t)_{t\in[0,T]}$ from \eqref{eq:dynamics_tokens}. For $t \in [0,T]$, let $\ell_t$ denote the index of the time interval such that $t_{\ell_t} \le t < t_{\ell_t+1}$, where $t_k = k \dt$. The interpolated process is defined by
\begin{equ}
\label{eq:dynamics_hat_interp}
\hat{X}^{i}_{t} = X^i_0 + \int_0^t \rdriftN{\ell_s}{\hat{X}} \, ds + \sum_{k=1}^\infty \int_0^t \rproj{\hat{X}^i_{\ell_s}} \sigma_k(\hat{X}^i_{\ell_s}) \, dB^k_s.
\end{equ}
At the discrete grid points $t_\ell$, this process coincides exactly with the discrete updates $\hat X_\ell^i$.

Since $\bar X_t^i\in \S^{d-1}$ almost surely for all $t$ (see lemma below), and the regularized coefficients coincide with the original ones on the sphere, we may equivalently view $\bar X$ as solving the regularized equation
\begin{equ}
\label{eq:dynamics_true}
\begin{cases}
    d{\bar X}^i_{t} = \rdriftN{t}{\bar X} \, dt + \sum_{k=1}^\infty \rproj{\bar X^i_t} \sigma_k({\bar X}^i_t) \, dB^k_t, \\
    {\bar X}^{i}_0 = X^{i}_0.
\end{cases}
\end{equ}

Before comparing these dynamics, we record the relevant a priori bounds.
\begin{lemma}
\label{lem:uniform_bounds}
    There exists a constant $\mathsf c > 0$, depending on $T$ but independent of $L$ and $N$, such that for all $i$:
    \begin{equ}
        \E\left[\sup_{0\le t\le T} |\bar X_t^i|^2\right] = 1 \quad \text{and} \quad \E\left[\sup_{0\le t\le T} |\hat{X}_t^i|^2\right] \le \mathsf c.
    \end{equ}
\end{lemma}
\begin{proof}
The first bound holds identically. Indeed, applying It\^o's formula to \(f(x)=|x|^2\) along trajectories of \eqref{eq:dynamics_tokens}
we obtain
\begin{equ}
\begin{aligned}
\dd|\bar X_t^i|^2
&=2\bar X_t^i\cdot d\bar X_t^i+\sum_{k=1}^{\infty}\Tr\!\left((\sigma_k(\bar X_t^i)P_{\bar X_t^i})(\sigma_k(\bar X_t^i)P_{\bar X_t^i})^\top\right)dt \\
&=-\kiso(d-1)|\bar X_t^i|^2dt+\sum_{k=1}^{\infty}\sigma_k(\bar X_t^i)^2\Tr(P_{\bar X_t^i})\,dt,
\end{aligned}
\end{equ}
since \(\bar X_t^i\cdot P_{\bar X_t^i}v=0\) for every \(v\in\R^d\). As \(|\bar X_t^i|=1\) implies \(\Tr(P_{\bar X_t^i})=d-1\) and \(\sum_{k=1}^{\infty}\sigma_k(\bar X_t^i)^2=\kiso\) on \(\S^{d-1}\), the two drift terms cancel and so
$
\dd|\bar X_t^i|^2=0.
$
Hence \(|\bar X_t^i|=|X_0^i|=1\) for all \(t\in[0,T]\) almost surely.
    For the interpolated process $\hat{X}_t^i$, the regularized drift and diffusion coefficients are globally bounded by Lemma~\ref{lem:regularity_projected_coefficients}. The uniform bound on the expected supremum of the squared norm then follows by a standard argument based on Cauchy-Schwarz for the drift integral, Doob's maximal inequality for the stochastic integral, and the continuous Gr\"onwall lemma.
\end{proof}

We first control the distance between the discrete scheme and its interpolation.

\begin{lemma}
\label{lem:increment_delta}
    There exists a constant $\mathsf c_T>0$, independent of $N$, $m$, and $L$, such that
    \begin{equ}
    \E\left[\sup_{0\le t\le T}\frac1N\sum_{i=1}^N |\hat X_{\ell_t}^i-\hat X_t^i|^2\right]
    \le \mathsf c_T\,\frac{\log(2L)}{L}.
    \end{equ}
\end{lemma}

\begin{proof}
For every $\ell\in\{0,\dots,L-1\}$, every $t\in[t_\ell,t_{\ell+1})$, and every $i\in[N]$, by \eqref{eq:dynamics_hat_interp} we have
\begin{equ}
\hat X_t^i-\hat X_\ell^i
=
(t-t_\ell)\rdriftN{\ell}{\hat X}
+
\sum_{k=1}^\infty \int_{t_\ell}^t \rproj{\hat X_\ell^i}\sigma_k(\hat X_\ell^i)\,dB_s^k.
\end{equ}
Hence
\begin{equ}
\frac1N\sum_{i=1}^N |\hat X_t^i-\hat X_\ell^i|^2
\le
2\dt^2\,\frac1N\sum_{i=1}^N |\rdriftN{\ell}{\hat X}|^2
+
2\frac1N\sum_{i=1}^N
\left|
\sum_{k=1}^\infty \int_{t_\ell}^t \rproj{\hat X_\ell^i}\sigma_k(\hat X_\ell^i)\,dB_s^k
\right|^2.
\end{equ}
By Lemma~\ref{lem:regularity_projected_coefficients},
\begin{equ}
\frac1N\sum_{i=1}^N |\rdriftN{\ell}{\hat X}|^2\le \mathsf c,
\qquad
\sum_{k=1}^\infty |\rproj{\hat X_\ell^i}\sigma_k(\hat X_\ell^i)|^2\le \mathsf c
\quad\text{for every }i.
\end{equ}

Let \(p=2\max\{1,\log(2L)\}\). Since \(x\mapsto x^{p/2}\) is convex, for every \(\ell\),
\begin{align*}
&\E\Bigg[\Bigg(
\sup_{t_\ell\le t<t_{\ell+1}}
\frac1N\sum_{i=1}^N
\left|
\sum_{k=1}^\infty \int_{t_\ell}^t \rproj{\hat X_\ell^i}\sigma_k(\hat X_\ell^i)\,dB_s^k
\right|^2
\Bigg)^{p/2}\Biggm|\mathcal F_{t_\ell}\Bigg] \\
&\qquad\le
\frac1N\sum_{i=1}^N
\E\Bigg[
\sup_{t_\ell\le t<t_{\ell+1}}
\left|
\sum_{k=1}^\infty \int_{t_\ell}^t \rproj{\hat X_\ell^i}\sigma_k(\hat X_\ell^i)\,dB_s^k
\right|^p
\Biggm|\mathcal F_{t_\ell}\Bigg] \\
&\qquad\le
\frac1N\sum_{i=1}^N
(\mathsf c p)^{p/2}
\left(
\int_{t_\ell}^{t_{\ell+1}}
\sum_{k=1}^\infty |\rproj{\hat X_\ell^i}\sigma_k(\hat X_\ell^i)|^2\,ds
\right)^{p/2}
\le
(\mathsf c p\dt)^{p/2},
\end{align*}
where in the third row we applied the Burkholder--Davis--Gundy inequality to each $\R^d$-valued martingale.
Therefore
\begin{align*}
&\E\left[\sup_{0\le t\le T}\frac1N\sum_{i=1}^N |\hat X_{\ell_t}^i-\hat X_t^i|^2\right] \\
&\qquad\le
\mathsf c\,\dt^2
+
\mathsf c\,\E\Bigg[
\max_{0\le \ell\le L-1}
\sup_{t_\ell\le t<t_{\ell+1}}
\frac1N\sum_{i=1}^N
\left|
\sum_{k=1}^\infty \int_{t_\ell}^t \rproj{\hat X_\ell^i}\sigma_k(\hat X_\ell^i)\,dB_s^k
\right|^2
\Bigg] \\
&\qquad\le
\mathsf c\,\dt^2
+
\mathsf c\left(
\sum_{\ell=0}^{L-1}
\E\Bigg[\Bigg(
\sup_{t_\ell\le t<t_{\ell+1}}
\frac1N\sum_{i=1}^N
\left|
\sum_{k=1}^\infty \int_{t_\ell}^t \rproj{\hat X_\ell^i}\sigma_k(\hat X_\ell^i)\,dB_s^k
\right|^2
\Bigg)^{p/2}\Bigg]
\right)^{2/p} \\
&\qquad\le
\mathsf c\,\dt^2+\mathsf c\,p\,L^{2/p}\dt.
\end{align*}
Since \(L^{2/p}\le e\) and \(\dt=T/L\), this gives
\begin{equ}
\E\left[\sup_{0\le t\le T}\frac1N\sum_{i=1}^N |\hat X_{\ell_t}^i-\hat X_t^i|^2\right]
\le
\mathsf c_T\,\frac{\log(2L)}{L}.
\end{equ}
\end{proof}

We next control the distance between $\hat X$ and $\bar X$. For this we use a general Euler-Maruyama estimate for systems of $N$ interacting particles driven by common noise, which is stated in the appendix as Proposition \ref{prop:euler_maruyama}. Although the result is classical, we reprove it there because we need a version which is uniform in $N$, allows infinitely many driving noises, and is formulated directly under the global Lipschitz and boundedness assumptions verified in our setting. 

\begin{lemma}
\label{lem:strong_convergence_sde}
    Under Assumption \ref{ass:att_reg}, the time-marginal error between the interpolated scheme $\hat{X}$ and the true continuous dynamics $\bar X$ satisfies:
    \begin{equ}
        \E\left[\sup_{0\leq t\leq T}\frac{1}{N}\sum_{i=1}^N |\hat{X}^i_t - \bar X^i_t|^2\right] \leq \frac{\mathsf c}{L} (4T+T^2) e^{(4+T)\mathsf cT},
    \end{equ}
    where $\mathsf c > 0$ is a constant independent of $N$, $L$, and $m$.
\end{lemma}
\begin{proof}
    This bound is a direct consequence of Proposition~\ref{prop:euler_maruyama} applied to the regularized equation \eqref{eq:dynamics_true}, with
    \begin{equ}
    a(x^i,\mathbf x)=b^\rho(x^i,\mu_{\mathbf x}^N),
    \qquad
    \tilde \sigma_k(x)=\rproj{x}\sigma_k(x).
    \end{equ}
    Lemma~\ref{lem:regularity_projected_coefficients} gives the global Lipschitz and boundedness estimates required in the proposition, while Lemma~\ref{lem:uniform_bounds} ensures that the regularized dynamics coincide with the original limiting dynamics along the trajectory of $\bar X$. This yields the claim.
\end{proof}

We can now conclude the proof of Proposition~\ref{prop:quantitative_deep_limit}.
\begin{proof}[Proof of Proposition \ref{prop:quantitative_deep_limit}]
    Using $(a+b+c)^2 \le 3a^2+3b^2+3c^2$, we obtain
    \begin{equs}
        \E\left[\sup_{0 \le t \le T} \frac{1}{N}\sum_{i=1}^N |X_{\ell_t}^i - \bar X_{t}^i|^2\right]
        &\leq 3\E\left[\max_{0 \le k \le L} \frac{1}{N}\sum_{i=1}^N |X_k^i - \hat{X}_{k}^i|^2\right] \\
        &\label{e:secondterm} \quad + 3\E\left[\sup_{0 \le t \le T} \frac{1}{N}\sum_{i=1}^N |\hat{X}_{\ell_t}^i - \hat{X}_t^i|^2\right] \\
        &\quad + 3\E\left[\sup_{0 \le t \le T} \frac{1}{N}\sum_{i=1}^N |\hat{X}_t^i - \bar X_t^i|^2\right].
    \end{equs}
The three terms above are bounded by Lemmas~\ref{lem:strong_error}, \ref{lem:increment_delta}, and \ref{lem:strong_convergence_sde}. 
    Collecting the three bounds gives
    \begin{equs}
    \E\left[\sup_{0\leq t\leq T}\frac{1}{N}\sum_{i=1}^N |{X^{i}}_{\lfloor{Lt/T}\rfloor} - \bar X^i_t|^2\right]
        &\leq \mathsf c_N \exp({\mathsf c_N T^2}) \left(\frac{\log L}{L} + \frac{1}{m}\right),\\
        \E\left[\sup_{0\leq t\leq T}\frac{1}{N}\sum_{i=1}^N |{X^{i}}_{\lfloor{Lt/T}\rfloor} - \bar X^i_t|^2\right]
        &\leq \mathsf c \exp({\mathsf c T^2}) \left(\frac{\log L}{L} + \frac{1}{m^{1/4}}\right),
    \end{equs}
    which concludes the proof.
\end{proof}

\begin{remark}
    \label{r:droplog}
    If the $\sup$ in the above formulas is taken over the subset $t \in \{T\ell/L\}_{\ell \in \{0, \dots L\}}$ (where $t_\ell = t$) the second term in \eqref{e:secondterm}, carrying the $\log L/L$ rate, vanishes, improving the rate to $1/L$.
\end{remark}

\section{Mean-Field Limit}
\label{sec:MF_limit}

From Section \ref{sec:deep_lim}, we know that the limiting SDE of each token is described by Equation \eqref{eq:dynamics_tokens}:
\begin{equation*}
\dd {\bar X}^i_{t} =\left(P_{\bar X^i_t}\attN{t}{\bar X}-\kiso\frac{d-1}{2}\bar X^i_t\right)\dd t + \sum_{k=1}^\infty  \sigma_k({\bar X}^i_t) P_{\bar X^i_t}\dd B^k_t,
\end{equation*}
with the initial condition $\bar X_0^i = X_0^i\in\S^{d-1}$ for all $i \in [N]$ and $P_x:=\Id-xx^\top$.
Recalling the definition of the fields $\sigma_{k\alpha}~:~\S^{d-1}\to \mathbb R^d$ from \eqref{e:sigmaka}, we now proceed as in \cite{coghi2016propagation} to derive the corresponding Eulerian formulation, by applying It\^o's formula to $\dd\langle \mu_t^N,\varphi\rangle$ for the empirical measure 
\begin{equ}
\mu_t^N:=\frac1N\sum_{i=1}^N \delta_{\bar X_t^i}.
\end{equ}

For any test function \(\varphi\in C_c^\infty(\R^d)\) (the space of compactly supported, smooth test functions), It\^o's formula yields:
\begin{equ}
\begin{aligned}
\dd\langle \mu_t^N,\varphi\rangle
&=
\frac1N\sum_{i=1}^N
\nabla\varphi(\bar X_t^i)\cdot P_{\bar X^i_t}\attN{t}{\bar X}^i\,\dd t
-\kiso\frac{d-1}{2}\,\langle \mu_t^N, x\cdot \nabla\varphi(x)\rangle\,\dd t \\
&\quad
+\frac12
\left\langle \mu_t^N,\sum_{k=1}^\infty\sum_{\alpha=1}^d
(P_x\sigma_{k\alpha}(x))^\top D^2\varphi(x)(P_x\sigma_{k\alpha}(x))\right\rangle \dd t \\
&\quad
+\sum_{k=1}^\infty\sum_{\alpha=1}^d
\left\langle \mu_t^N,\nabla\varphi(x)\cdot P_x\sigma_{k\alpha}(x)\right\rangle \dd B_t^{k,\alpha} .
\end{aligned}
\end{equ}
From Section \ref{s:noise} we know that for \(x\in\mathbb S^{d-1}\) the noise is isotropic, i.e.
\begin{equ}
\sum_{k=1}^\infty\sum_{\alpha=1}^d \big(P_x\sigma_{k\alpha}(x)\big) \big(P_x\sigma_{k\alpha}(x)\big)^\top
=
\kiso\,P_x .
\end{equ}
Indeed, by \eqref{e:sigmaka},
\begin{equ}
\sum_{k=1}^\infty\sum_{\alpha=1}^d \big(P_x\sigma_{k\alpha}(x)\big) \big(P_x\sigma_{k\alpha}(x)\big)^\top
=
\sum_{k=1}^\infty \sigma_k(x)^2 \sum_{\alpha=1}^d (P_xe_\alpha)(P_xe_\alpha)^\top
=
\sum_{k=1}^\infty \sigma_k(x)^2 P_x
=
\kiso\,P_x,
\end{equ}
since \(\sum_{\alpha=1}^d (P_xe_\alpha)(P_xe_\alpha)^\top=P_xP_x^\top=P_x\), and \(\sum_{k=1}^\infty \sigma_k(x)^2=\kiso\) on \(\S^{d-1}\).
Then we can rewrite the second order term as
\begin{equ}
\sum_{k=1}^\infty\sum_{\alpha=1}^d
(P_x\sigma_{k\alpha}(x))^\top D^2\varphi(x)(P_x\sigma_{k\alpha}(x))
=
\kiso\,\Tr\!\big(P_xD^2\varphi(x)\big),
\end{equ}
hence
\begin{equ}
\begin{aligned}
\dd\langle \mu_t^N,\varphi\rangle
&=
\frac1N\sum_{i=1}^N
\nabla\varphi(\bar X_t^i)\cdot P_{\bar X^i_t}\attN{t}{\bar X}^i\,\dd t \\
&\quad
+\frac{\kiso}{2}
\left\langle \mu_t^N,
\Tr\!\big(P_xD^2\varphi(x)\big)-(d-1)x\cdot \nabla\varphi(x)
\right\rangle \dd t \\
&\quad
+\sum_{k=1}^\infty\sum_{\alpha=1}^d
\left\langle \mu_t^N,\nabla\varphi(x)\cdot P_x\sigma_{k\alpha}(x)\right\rangle \dd B_t^{k,\alpha}.
\end{aligned}
\end{equ}

In this formulation, we recognize the appearance of the Laplace-Beltrami operator $\Delta_{\S^{d-1}}$. On \(\S^{d-1}\),
\begin{equ}
\Delta_{\S^{d-1}}(\varphi|_{\S^{d-1}})(x)
=
\Tr\!\big(P_xD^2\varphi(x)\big)-(d-1)x\cdot \nabla\varphi(x).
\end{equ}
Therefore
\begin{equ}
\begin{aligned}
\dd\langle \mu_t^N,\varphi\rangle
&=
\frac1N\sum_{i=1}^N
\nabla\varphi(\bar X_t^i)\cdot P_{\bar X^i_t}\attN{t}{\bar X}^i\,\dd t \\
&\quad
+\frac{\kiso}{2}
\left\langle \mu_t^N,\Delta_{\S^{d-1}}(\varphi|_{\S^{d-1}})\right\rangle \dd t \\
&\quad
+\sum_{k=1}^\infty\sum_{\alpha=1}^d
\left\langle \mu_t^N,\nabla\varphi(x)\cdot P_x\sigma_{k\alpha}(x)\right\rangle \dd B_t^{k,\alpha}.
\end{aligned}
\end{equ}
Given that the dynamics is bound to the sphere as proven in Lemma~\ref{lem:uniform_bounds}, we can write the SPDE in weak form testing against functions on $\S^{d-1}$. For every \(f\in C^\infty(\S^{d-1})\),
\begin{equ}
\dd\langle \mu_t^N,f\rangle
=
\langle \mu_t^N,\nabla_{\S^{d-1}} f\cdot P_x\att{\mu_t^N}{\cdot}\rangle\,\dd t
+
\frac{\kiso}{2}\langle \mu_t^N,\Delta_{\S^{d-1}}f\rangle\,\dd t
+
\sum_{k=1}^\infty\sum_{\alpha=1}^d
\langle \mu_t^N,\nabla_{\S^{d-1}} f\cdot P_x\sigma_{k\alpha}\rangle\,\dd B_t^{k,\alpha},
\end{equ}
where we recall that $\nabla_{\S^{d-1}}$ and $\mathrm{div}_{\S^{d-1}}$ denote, respectively, the Riemannian gradient and divergence operators on the sphere. Equivalently, in divergence form,
\begin{equ}
\dd\mu_t^N
+
 \mathrm{div}_{\S^{d-1}}(P_x\att{\mu_t^N}{\cdot}\mu_t^N)\,\dd t
+
\sum_{k=1}^\infty\sum_{\alpha=1}^d \mathrm{div}_{\S^{d-1}}\!\big((P_x\sigma_{k\alpha})\mu_t^N\big)\,\dd B_t^{k,\alpha}
=
\frac{\kiso}{2}\Delta_{\S^{d-1}}\mu_t^N\,\dd t ,
\end{equ}
which matches the SPDE in Equation \eqref{eq:limiting-spde}. 

\subsection{Quantitative limits}
Here we prove quantitative convergence limits of the empirical measure $\empmu_t$ to the limiting measure $\mu_t$, defined as the unique measure-valued solution to \eqref{eq:limiting-spde} in the sense of \cite[Definition 11]{coghi2016propagation}, namely,
\begin{itemize}
  \item for every $f\in C_b(\S^{d-1})$, $t\mapsto\langle\mu_t, f\rangle$ is an adapted process with a continuous version,
  \item for every $f\in C_b^2(\S^{d-1})$,
  \begin{equ}
    \begin{aligned}
      \langle \mu_t,f\rangle
        &= \langle \mu_{0},f\rangle
        + \int_0^t\langle \mu_s,\nabla_{\S^{d-1}} f\cdot P_x\att{\mu_s}{\cdot}\rangle\,\dd s
        + \frac{\kiso}{2}\int_0^{t}\langle \mu_s,\Delta_{\S^{d-1}}f\rangle\,\dd s\\
        &\quad + \sum_{k=1}^\infty\sum_{\alpha=1}^d \int_0^t\langle \mu_s,\nabla_{\S^{d-1}} f\cdot P_x\sigma_{k\alpha}\rangle\,\dd B_s^{k,\alpha}.
    \end{aligned}
  \end{equ}
\end{itemize}
Existence and uniqueness of solutions, in the sense given above, to \eqref{eq:limiting-spde} follows immediately from \cite[Theorem 13]{coghi2016propagation} thanks to the global Lipschitz bounds of Lemmas~\ref{lem:unif_gl} and \ref{lem:regularity_projected_coefficients}. We proceed to prove the remaining claim of Proposition~\ref{thm:main}.

\begin{proposition}
\label{prop:quantitative_MF_limit}
    Let $\mu_0 \in \mathcal P_1(\S^{d-1})$ be the initial distribution, $\empmu_0$ be the empirical measure of \eqref{eq:dynamics_tokens} at initialization, sampled iid from $\mu_0$. Let $r_{d}(N)$ be the rate
    \begin{equ}
        r_{d}(N)=\begin{cases}
            N^{-1/2}& d<4\\
            N^{-1/2}\log(1+N)& d=4\\
            N^{-2/d} & d>4 
        \end{cases}.
    \end{equ}
    Then the following bound holds:
    \begin{equ}
        \E\left[\sup_{0\leq t\le T} \W_2^2(\empmu_t,\mu_t)\right]\le \mathsf c_{T}\, r_d(N)
    \end{equ}
    where $\mathsf c_{T}$ is a constant that depends only on $T, d, \act, \varbU,\varbW$ and the Lipschitz constants in Assumption~\ref{ass:att_reg}. Here, $\mathbb E$ is the unconditional expectation on the underlying probability space (including the initialization).
\end{proposition}

\begin{proof}

Taking the expectation and applying the bound from \cite[Theorem 1]{fournier2015rate} for the squared $2$-Wasserstein distance yields:

\begin{equ} \label{eq:W1_bound}
\mathbb{E}[\mathcal{W}_2^2(\mu_0^N, \mu_0)] \le \mathsf c r_d(N)
\end{equ}
where $\mathsf c$ is a constant depending only on $d$.
Note that because our measures are compactly supported on $\mathbb{S}^{d-1}$, moments of all orders are uniformly bounded, and the prefactor $M_q$ in \cite[Theorem 1]{fournier2015rate} is absorbed into the optimal rate $r_d(N)$.

Finally, we adapt \cite[Theorem 21]{coghi2016propagation} to bound the continuous-time supremum by the initial distance. Since the sphere is closed under the dynamics of the particle system \eqref{e:sa} and the limiting \eqref{eq:limiting-spde} as proven in Lemma~\ref{lem:uniform_bounds}, we extend the definition of our coefficients to $\R^d$ as detailed in Section~\ref{sec:single-step}. We note that the bounds in Lemmas~\ref{lem:unif_gl} and \ref{lem:regularity_projected_coefficients} verify that the standing assumptions of \cite[Theorem 21]{coghi2016propagation} are satisfied in our setting.  
We also note that while the result in \cite[Theorem 21]{coghi2016propagation}
does not include the supremum inside the expectation, one can use the same Gr\"onwall argument with the stochastic integral treated as a continuous martingale and the Burkholder--Davis--Gundy inequality inserted at the analogue of \cite[Equation~(43)]{coghi2016propagation}, which yields the following expected-supremum estimate:
\begin{align}
\mathbb{E}\left[\sup_{0\le t\le T}\mathcal{W}_2^2(\mu_t^N,\mu_t)\right] &\le 4e^{4T(2TL_b^2+\mathsf c_2L_\sigma^2)}\mathbb{E}[\mathcal{W}_2^2(\mu_0^N,\mu_0)] \\
&\le 4 e^{4T(2TL_b^2+\mathsf c_2L_\sigma^2)} \cdot\mathsf c  r_d(N)
\end{align}
where $L_b$ and $L_\sigma$ are the Lipschitz constants of the coefficients of the system, $\mathsf c_2$ is the constant appearing in the Burkholder--Davis--Gundy inequality with exponent 2, and $\mathsf c$ is a constant that only depends on $d$. 
Letting $\mathsf c_T = 4e^{4T(2TL_b^2+\mathsf c_2L_\sigma^2)}\cdot\mathsf c $ yields the desired bound and concludes the proof.
\end{proof}

\begin{remark}[Uniformity of the Lipschitz constants]
    Crucially, the constant $\mathsf c_T$ remains bounded as $N \to \infty$. The terms $L_b$ and $L_\sigma$ represent the global Lipschitz constants of the regularized drift (the attention mechanism) and the diffusion coefficients, respectively. As established in Section~\ref{sec:deep_lim}, because the attention map is mollified and the Karhunen-Loève eigenfunctions $\sigma_k$ are derived from the bounded, isotropic kernel $\mathcal{K}$, these Lipschitz constants depend strictly on the dimension $d$, the sequence length $T$, the activation function $\act$, and the mollification parameter. They are entirely independent of the number of tokens $N$, ensuring the exponential factor in the Gr\"onwall bound does not blow up in the mean-field limit.
\end{remark}

The following result concludes the first part of the paper, where we proved the statements presented in Section \ref{sec:intro-scaling-limit}.

\begin{proof}[Proof of Corollary \ref{cor:main}]
Fix $N \in \mathbb N$ and let
\begin{equ}
\mu_t^N := \frac{1}{N}\sum_{i=1}^N \delta_{\bar X_t^i},
\end{equ}
where $\bar X$ solves \eqref{eq:dynamics_tokens}. We realize the discrete fields $G_\ell^m$, the Brownian motions $(B^k)_{k \geq 1}$, the particle system $\bar X$, and the SPDE solution $\mu$ on a common probability space, utilizing the coupling $\Gamma_{m,L}$ from Remark \ref{rem:adapted_coupling_convention} and the same initial sample $(X_0^i)_{i=1}^N$. 
For each $t \in [0,T]$, pairing the $i$-th atom of $\mu_t^{N,L,m}$ with the $i$-th atom of $\mu_t^N$ gives the pointwise transport estimate
\begin{equ}
\W_2^2(\mu_t^{N,L,m}, \mu_t^N) \leq \frac{1}{N}\sum_{i=1}^N |X_{\lfloor Lt/T\rfloor}^i - \bar X_t^i|^2.
\end{equ}
Combining this with the triangle inequality for $\W_2$ yields
\begin{align*}
\distS(\mu^{N,L,m}, \mu)
&\leq 2\,\distS(\mu^{N,L,m}, \mu^N) + 2\,\distS(\mu^N, \mu) \\
&\leq 2\,\E\left[\sup_{0\leq t\leq T}\frac{1}{N}\sum_{i=1}^N |X_{\lfloor Lt/T\rfloor}^i - \bar X_t^i|^2\right] + 2\,\mathsf c_T\, r_d(N) \\
&\leq 2\,\mathsf c_N \exp(\mathsf c_N T^2)\left(\frac{\log L}{L} + \frac{1}{m}\right) + 2\,\mathsf c_T\, r_d(N),
\end{align*}
where the last line uses Proposition~\ref{prop:quantitative_deep_limit} for the deep-width error and Proposition~\ref{prop:quantitative_MF_limit} for the mean-field error.

For fixed $N$, the first term vanishes by sending first $m \to \infty$ and then $L \to \infty$. The remaining term tends to $0$ as $N \to \infty$ because $r_d(N) \to 0$. This proves
\begin{equ}
\lim_{N \to \infty} \, \lim_{L \to \infty} \, \lim_{m \to \infty} \, \distS(\mu^{N,L,m}, \mu) = 0.
\end{equ}

If, in addition, we combine the above bounds with the second (uniform in $N$) estimate of Proposition~\ref{prop:quantitative_deep_limit} for $\act \in C^\infty$ as discussed in Remark~\ref{r:uniform}, then the right-hand side above yields
\begin{equ}
\distS(\mu^{N,L,m}, \mu)
\leq \mathsf c_T \left(\frac{\log L}{L} + \frac{1}{m^{1/4}} + r_d(N)\right)\,.
\end{equ}
This converges to $0$ along any sequence with $m,L,N \to \infty$. Hence the three limits commute.
\end{proof}

\section{Exponential clustering for common-noise model}\label{s:sync}
In this section we address the synchronization by noise phenomenon introduced in Section \ref{sec:intro-sync-by-noise}.

We recall that the common-noise dynamics admits both an Eulerian and a Lagrangian description. At the Eulerian level, the evolution of the law is governed by the limiting SPDE \eqref{eq:limiting-spde}.
On the other hand, the corresponding McKean--Vlasov equations for particles transported by the same common noise provide a Lagrangian formulation of the same dynamics, given by \eqref{eq:dynamics_tokens}. In particular, if the initial datum is empirical, then the SPDE solution coincides almost surely with the empirical measure of the corresponding particle system.

Accordingly, in this section we work with the following Lagrangian reformulation of the SPDE. Consider two particles $X_t,Y_t\in \mathbb S^{d-1}$ driven by the same common noise, and conditionally independent given the common-noise filtration $(\mathcal F_t^B)_{t\ge 0}$. Recall from Section \ref{sec:deep_lim} that the corresponding McKean--Vlasov dynamics is given by
\begin{align*}
dX_t &=
\left(
P_{X_t}\Attn(\mu_t,X_t)
-\kiso\frac{d-1}{2}X_t
\right)dt
+
\sum_{k = 1}^\infty\sum_{\alpha=1}^d P_{X_t}\sigma_{k\alpha}(X_t) \, dB_t^{k,\alpha}, \\
dY_t &=
\left(
P_{Y_t}\Attn(\mu_t,Y_t)
-\kiso\frac{d-1}{2}Y_t
\right)dt
+
\sum_{k = 1}^\infty\sum_{\alpha=1}^d P_{Y_t}\sigma_{k\alpha}(Y_t) \, dB_t^{k,\alpha},
\end{align*}
where $P_x = \Id - xx^\top$ is the orthogonal projection onto the tangent space of $\S^{d-1}$ at $x$, and the term $-\kiso(d-1)x/2$ represents the It\^o correction ensuring the particles remain on the sphere. By construction,
\begin{equ}
\Law(X_t \mid \mathcal F_t^B)=\Law(Y_t \mid \mathcal F_t^B)=\mu_t,
\end{equ}
and $X_t,Y_t$ are conditionally independent given $\mathcal F_t^B$.

As a consequence, using the conditional independence of $X_t$ and $Y_t$ given $\mathcal F_t^B$, we can write
\begin{equation*}
\mathcal E_\beta(\mu_t)
=
\frac 1 2 \iint \left(e^\beta -  e^{\beta\langle x,y\rangle}\right)\,\dd\mu_t(x)\dd\mu_t(y)
=
\frac 1 2 e^\beta
-
\frac 1 2 \E\bigl[e^{\beta\langle X_t,Y_t\rangle}\mid \mathcal F_t^B\bigr].
\end{equation*}

\subsection{Proof of Theorem~\ref{t:sync}}

By It\^o's formula for the two-particle process $(X_t,Y_t)$, the evolution of the energy is given by
\begin{align}\notag
\frac{d}{dt}&\E[\mathcal E_\beta(\mu_t)] \\\notag
&=
\underbrace{-\E\left[
\beta
\iint
e^{\beta\langle x,y\rangle}
\,y\cdot
P_x\Attn(\mu_t,x)
\,\dd\mu_t(x)\dd\mu_t(y)
\right]}_{(\text{I})} \\\label{e:ito_sync}
& \quad
\underbrace{+\E\left[
\kiso\frac{d-1}{2}
\iint \langle x,\nabla_x e^{\beta\langle x,y\rangle}\rangle\,\dd\mu_t(x)\dd\mu_t(y)
\right]}_{(\text{II})} \\
&\notag \quad
\underbrace{-\frac{1}{2}\E\left[
\sum_{k,\alpha}
\iint
(P_x\sigma_{k\alpha})^\top D^2_{xy} e^{\beta\langle x,y\rangle} P_y\sigma_{k\alpha}
\,\dd\mu_t(x)\dd\mu_t(y)
\right]}_{(\text{III})} \\
&\notag \quad
\underbrace{-\frac{1}{2}\E\left[
\sum_{k,\alpha}
\iint
\left(
\frac12 (P_x\sigma_{k\alpha})^\top D^2_{xx} e^{\beta\langle x,y\rangle} P_x\sigma_{k\alpha} + \frac12 (P_y\sigma_{k\alpha})^\top D^2_{yy} e^{\beta\langle x,y\rangle} P_y\sigma_{k\alpha}
\right)
\,\dd\mu_t(x)\dd\mu_t(y)
\right]}_{(\text{IV})}.
\end{align}
where $\nabla_x, D^2$ are the gradient and Hessian in ambient space with respect to the variables $x,y$.
We analyze the drift terms (I), (II), (III), and (IV) separately.

The first term (I) represents the contribution from the attention drift. By Assumption~\ref{ass:sync_forcing} we can control the argument of the expectation as 
\begin{align*}
-\beta
\iint
e^{\beta\langle x,y\rangle}
\,y\cdot
P_x\Attn(\mu,x)
\,\dd\mu(x)\dd\mu(y) \leq \epsilon \mathcal E_\beta(\mu)
\end{align*}
so that 
\begin{align}\label{e:sync_forcing}
(\text{I})
&
\leq \epsilon \mathbb E[\mathcal E_\beta(\mu_t)].
\end{align}

For the second term (II), using that $\nabla_x e^{\beta\langle x,y\rangle} = \beta e^{\beta\langle x,y\rangle}y$ we obtain
\begin{equation*}
(\text{II}) = \E\left[
\kiso\frac{d-1}{2}
\iint \beta e^{\beta u}u\,\dd\mu_t(x)\dd\mu_t(y)
\right],
\end{equation*}
where we denote $u = \langle x,y \rangle$.

The third term (III) involves the mixed trace contraction:
\begin{align*}
(\text{III})
&=
-\frac{1}{2}\E\left[
\sum_{k,\alpha}
\iint
\beta e^{\beta u}
\left(
(P_x\sigma_{k\alpha})\cdot(P_y\sigma_{k\alpha})
+
\beta\,(P_x\sigma_{k\alpha}\cdot y)(x\cdot P_y\sigma_{k\alpha})
\right)
\,\dd\mu_t(x)\dd\mu_t(y)
\right] \\
&=
-\frac{1}{2}\E\left[
\iint
\beta e^{\beta u}
\left(
\kiso(u)\Tr(P_xP_y)
+
\beta\kiso(u)\Tr(P_xyx^\top P_y)
\right)
\,\dd\mu_t(x)\dd\mu_t(y)
\right].
\end{align*}
Indeed, by \eqref{e:sigmaka},
\begin{equ}
\sum_{k,\alpha}(P_x\sigma_{k\alpha})\cdot(P_y\sigma_{k\alpha})
=
\sum_k \sigma_k(x)\sigma_k(y)\sum_\alpha (P_xe_\alpha)\cdot(P_ye_\alpha)
=
\kiso(u)\Tr(P_xP_y),
\end{equ}
and similarly
\begin{equ}
\sum_{k,\alpha}(P_x\sigma_{k\alpha}\cdot y)(x\cdot P_y\sigma_{k\alpha})
=
\sum_k \sigma_k(x)\sigma_k(y)\sum_\alpha (P_xe_\alpha\cdot y)(x\cdot P_ye_\alpha)
=
\kiso(u)\Tr(P_xyx^\top P_y).
\end{equ}
Using the properties $P_xx = 0$, $\Tr(P_x)=d-1$, and standard trace identities, we have:
\begin{align*}
\Tr(P_xP_y) = d - 2 + u^2\qquad \text{and} \qquad 
  \Tr(P_xyx^\top P_y) = -u(1-u^2)\,.
\end{align*}
Substituting these back:
\begin{align*}
(\text{III})
&=
-\frac{1}{2}\E\left[
\iint
\kiso(u)\beta e^{\beta u}
\left(d-2 + u^2 -\beta u(1-u^2)\right)
\,\dd\mu_t(x)\dd\mu_t(y)
\right].
\end{align*}

The fourth term (IV) involves the diagonal second-order contribution. Using
\begin{equ}
D^2_{xx} e^{\beta\langle x,y\rangle} = \beta^2 e^{\beta\langle x,y\rangle}yy^\top,
\qquad
D^2_{yy} e^{\beta\langle x,y\rangle} = \beta^2 e^{\beta\langle x,y\rangle}xx^\top,
\end{equ}
and, again by \eqref{e:sigmaka},
\begin{equ}
\sum_{k,\alpha} (P_x\sigma_{k\alpha}\cdot y)^2
=
\sum_k \sigma_k(x)^2 \sum_\alpha (P_xe_\alpha\cdot y)^2
=
\kiso\,|P_xy|^2
=
\kiso(1)(1-u^2),
\end{equ}

(and similarly for the term with $x$ and $y$ exchanged), we get
\begin{align*}
(\text{IV})
&=
-\frac{1}{2}\E\left[
\iint
\kiso\beta^2 e^{\beta u}(1-u^2)
\,\dd\mu_t(x)\dd\mu_t(y)
\right].
\end{align*}

We now combine terms (II), (III), and (IV), yielding
\begin{align}\label{e:sync_dissipation_total}
-\frac{1}{2}\E\left[
\iint
\beta e^{\beta u}
\left(
\kiso(u)\bigl[d-2 + u^2 - \beta u(1-u^2)\bigr]
+ \kiso\bigl[\beta(1-u^2) - (d-1)u\bigr]
\right)
\,\dd\mu_t(x)\dd\mu_t(y)
\right].
\end{align}
We complete the proof of Theorem~\ref{t:sync} with the following estimate:
\begin{lemma}\label{l:sync_main}
Let Assumption~\ref{ass:sync_dissipation} hold and let $d > 2$, then there exists a constant $\bar \lambda' < 0$ such that 
\begin{equ}
\label{eq:k_c_energy}
-\frac 1 2  \beta e^{\beta u}\left(\kiso(u)(d-2+u^2-\beta u(1-u^2))+\kiso(1)(\beta(1-u^2)-(d-1)u)\right)\leq \bar \lambda' \frac 1 2 (e^\beta-e^{\beta u})
\end{equ}
for every $u\in[-1,1]$.
\end{lemma}

Indeed, by combining \eqref{e:sync_forcing} and \eqref{eq:k_c_energy} with \eqref{e:ito_sync}, we obtain:
\begin{equation*}
\frac{d}{dt}\E[\mathcal E_\beta(\mu_t)] \leq (\epsilon +\bar \lambda') \E[\mathcal E_\beta(\mu_t)].
\end{equation*}
By Gr\"onwall's lemma, this implies exponential synchronization speed:
\begin{equation*}
\E[\mathcal E_\beta(\mu_t)] \leq \exp((\epsilon + \bar \lambda') t)\,\E[\mathcal E_\beta(\mu_0)].
\end{equation*}\qed

To prove the missing Lemma~\ref{l:sync_main}, we recall that by Remark~\ref{r:schonberg} the kernel $\kiso$ admits the following decomposition:
\begin{equ}\label{e:decomposition}
\kiso(u)=\sum_{n=0}^{\infty}c_nP_{n,d}(u),
\end{equ}
where $P_{n,d}$ are the normalized Gegenbauer polynomials of order $n$ in dimension $d$ (so that $P_{n,d}(1) = 1$), and by Sch\"onberg theorem \cite{schoenberg} $\{c_n\}_{n \geq 1}$ are nonnegative constants.

\begin{proof}[Proof of Lemma~\ref{l:sync_main}]
We define
\begin{equ}\label{e:FG}
\bar \lambda' := - \inf_{u \in [-1,1]} F(u),
\qquad
F(u):=
\frac{\beta e^{\beta u}G(u)}
{e^\beta-e^{\beta u}},
\end{equ}
where
\begin{equ}
G(u)
:=
\kiso(u)\bigl(d-2+u^2-\beta u(1-u^2)\bigr)
+
\kiso(1)\bigl(\beta(1-u^2)-(d-1)u\bigr).
\end{equ}
We decompose
\begin{equ}
G(u)=G_0(u)+\beta G_1(u),
\end{equ}
where
\begin{align}\label{e:G}
G_0(u)
&:=
\kiso(u)(d-2+u^2)-\kiso(1)(d-1)u,\\
G_1(u)
&:=
(1-u^2)(\kiso(1)-u\kiso(u)).
\end{align}

We prove that \(G(u)>0\) for every \(u\in[-1,1)\), and then study the limit
as \(u\to1^-\).
First, since $|P_{n,d}(u)|\leq 1 = P_{n,d}(1)$ \cite[Section 4.7]{szeg1939orthogonal}, the decomposition \eqref{e:decomposition} yields
\begin{equ}
|\kiso(u)|\le \sum_{n=0}^{\infty} c_n |P_{n,d}(u)| \leq \sum_{n=0}^{\infty} c_n P_{n,d}(1)=  \kiso(1),
\qquad u\in[-1,1].
\end{equ}
Hence
\begin{equ}
u\kiso(u)\le |u|\,|\kiso(u)|\le \kiso(1),
\end{equ}
and therefore
\begin{equ}\label{e:G1}
G_1(u)=(1-u^2)(\kiso(1)-u\kiso(u))\ge0,
\qquad u\in[-1,1].
\end{equ}

We now prove positivity of \(G_0\). Since
\begin{equ}
G_0(u)
=
\kiso(1)(d-2+u^2)
\left(
\frac{\kiso(u)}{\kiso(1)}
-
\frac{(d-1)u}{d-2+u^2}
\right),
\end{equ}
and \(\kiso(1)(d-2+u^2)>0\), it is enough to show that
\begin{equ}\label{e:intermediate}
\frac{\kiso(u)}{\kiso(1)}
>
\frac{(d-1)u}{d-2+u^2},
\qquad u\in[-1,1).
\end{equ}
To prove this, we use the inequality from \cite[Eq. 1]{hrycak2019inequalities}
\begin{equ}
P_{n,d}(u)\ge 1-P'_{n,d}(1)(1-u),
\qquad u\in[-1,1]\,.
\end{equ}
so that we have
\begin{equ}
\kiso(u)
=
\sum_{n=0}^{\infty}c_nP_{n,d}(u)
\ge
\sum_{n=0}^{\infty}c_n
-
(1-u)\sum_{n=0}^{\infty}c_nP'_{n,d}(1)=\kiso(1)-\kiso'(1)(1-u),
\end{equ}
Then, by Assumption~\ref{ass:sync_dissipation} we have $\kiso'(1) < \kiso(1)(d-3)/(d-1)$ so that 
\begin{equ}
\kiso(u) > \kiso(1)\left(1 - \frac{d-3}{d-1}(1-u) \right) = \kiso(1) \frac {2 + (d-3)u}{d-1} 
\end{equ}
We finally obtain the desired bound \eqref{e:intermediate} by noting that for $u \in [-1,1)$ we have
\begin{align*}
\frac{2+(d-3)u}{d-1}
-
\frac{(d-1)u}{d-2+u^2}=
\frac{(1-u)^2\bigl((d-3)u+2d-4\bigr)}
{(d-1)(d-2+u^2)} \geq \frac{(1-u)^2}
{d-2+u^2} > 0.
\end{align*}
where we used that for \(u\in[-1,1]\) and \(d>3\),
\begin{equ}
(d-3)u+2d-4\ge d-1>0.
\end{equ}
Combining $G_0 > 0$ from \eqref{e:intermediate} with \eqref{e:G1}, this gives
\begin{equ}
G(u)=G_0(u)+\beta G_1(u)>0,
\qquad u\in[-1,1).
\end{equ}
and \(F(u)>0\) for every \(u\in[-1,1)\) as claimed.

It remains to study the limit as \(u\to1^-\). Since \(G_1(u)=O((1-u)^2)\), we
only need the first-order expansion of \(G_0\):
\begin{equ}
G_0(u)
=
\bigl(\kiso(1)(d-3)-\kiso'(1)(d-1)\bigr)(1-u)
+
o(1-u).
\end{equ}
Moreover,
\begin{equ}
e^\beta-e^{\beta u}
=
\beta e^\beta(1-u)+o(1-u).
\end{equ}
Therefore
\begin{equ}
\lim_{u\to1^-}F(u)
=
\kiso(1)(d-3)-\kiso'(1)(d-1)
=
-\bar\lambda
>0.
\end{equ}
Hence \(F\) extends continuously and positively to \([-1,1]\). Since the
extended \(F\) is continuous on the compact interval \([-1,1]\), we obtain
\begin{equ}
\inf_{u\in[-1,1]}F(u)>0.
\end{equ}
Thus
\begin{equ}
\bar\lambda':=-\inf_{u\in[-1,1]}F(u)<0.
\end{equ}
\end{proof}

\subsection{Contraction rates and Lyapunov exponents for general activations}\label{s:sync2}

In this section we aim to precisely state and prove Proposition \ref{p:le}. 
Recall from Section \ref{s:noise} that the limiting Gaussian process $(G_\ell(x))_{x\in\R^d}$ has covariance $K(x,y) \Id$, where:
\begin{equ}\label{eq:starK}
K(x,y)=\E\left[\act(Z(x))\act(Z(y))\right]+\varbW^2,
\end{equ}
and $\big(Z(x)\big)_{x\in\R^d}$ is a scalar field with covariance
\begin{equ}
\E[Z(x) Z(y)]=\frac{1}{d}\,\ip{x}{y}+\varbU^2.
\end{equ}
Fix $x,y\in\R^d$ and set $t=\ip{x}{y}$ and, for brevity, 
\begin{equ}
u(t):=\sqrt{\frac{1+t}{2d}},\qquad v(t):=\sqrt{\frac{1-t}{2d}}.
\end{equ}
If $X,Y$ and $b$ are independent real standard Gaussian variables, then
\begin{equ}\label{eq:zxzy}
\begin{aligned}
Z(x) &\stackrel{\text{law}}{=}  u(t)\,X + v(t)\,Y+\varbU b,\\
Z(y) &\stackrel{\text{law}}{=}  u(t)\,X - v(t)\,Y+\varbU b.
\end{aligned}
\end{equ}
We can then write
\begin{equ}\label{eq:alpha}
\kiso(t)=\E\left[\act\big(u(t)X+v(t)Y+\varbU b\big)\,\act\big(u(t)X-v(t)Y+\varbU b\big)\right]+\varbW^2.
\end{equ}
In particular, for $t=1$ we have $u(1)=\frac{1}{\sqrt d}$ and $v(1)=0$, hence
\begin{equ}
    \kiso(1)=\E\left[\act\left(\frac{1}{\sqrt d}X+\varbU b\right)^2\right]+\varbW^2. 
\end{equ}
\begin{lemma}
\label{lem:alpha-prime}
    Assume that $\act$ satisfies Assumption \ref{ass1} and is almost everywhere twice differentiable. Then the following equality holds:
    \begin{equ}
        \lim_{t\to1^-}\kiso'(t)=\frac{1}{d}\E\left[\act'\left(\frac{1}{\sqrt d}X+\varbU b\right)^2\right].
    \end{equ}
\end{lemma}
\begin{proof}
    This is a special case of the Gaussian covariance-derivative formula. Consider a Gaussian variable $Z_t$ with covariance matrix $\Sigma_t$, if $f$ is twice differentiable almost everywhere and has at most polynomial growth of its derivatives, we have 
    \begin{equ}
        \frac{\dd}{\dd t}\E[f(Z_t)]=\frac{1}{2}\E\left[\mathrm{Tr}\left(\Sigma_t'\nabla^2f(Z_t)\right)\right].
        \label{eq:covariance-derivative}
    \end{equ}
    In our case, we have $Z_t=(U_t,V_t)\sim\mathcal{N}(0,\Sigma_t)$ with 
    \begin{equ}
        \Sigma_t=\begin{pmatrix}
        1/d+\varbU^2&t/d+\varbU^2\\
        t/d+\varbU^2&1/d+\varbU^2
    \end{pmatrix}\,.
    \end{equ}
    We then define $f(U_t,V_t)=\act(U_t)\act(V_t)$.
    We can then compute the Hessian
    \begin{equ}
        \nabla^2f(x,y)=\begin{pmatrix}
            \act''(x)\act(y)&\act'(x)\act'(y)\\\act'(x)\act'(y)&\act(x)\act''(y)
        \end{pmatrix}
    \end{equ}
    and
    \begin{equ}
        \Sigma_t'=\begin{pmatrix}
            0&d^{-1}\\d^{-1}&0
        \end{pmatrix}.
    \end{equ}
    Applying \eqref{eq:covariance-derivative} we get exactly
    \begin{equ}
        \kiso'(t)=\frac{1}{d}\E\left[\act'\big(u(t)X+v(t)Y+\varbU b\big)\,\act'\big(u(t)X-v(t)Y+\varbU b\big)\right].
    \end{equ}
    Taking the limit for $t\to1$ from the left, that due to the hypothesis on $\act$ we can pass under the integral by dominated convergence, we get 
\begin{equ}
        \kiso'(1)=\frac{1}{d}\E\left[ \act'\left(\frac{1}{\sqrt d}X+\varbU b\right)^2 \right],
    \end{equ}
    hence concluding the proof.
\end{proof}

We recall the last result we prove, Proposition~\ref{p:le}:

\begin{proposition}
Let $\act(y) = \max\{0,y\}$ be the ReLU activation, then 
\begin{equ}
    \lambda_1 = -\frac{1}{2d}-\frac{(\varbU^2+2\varbW^2)(d-1)}{4}<0,
\end{equ}
and 
\begin{equ}
    \bar \lambda = \frac 1 d - \frac {(\varbU^2+2\varbW^2)(d-3)}{2}.
\end{equ}
\end{proposition}

\begin{proof}
    Using the formula given in \cite[Thm. 3.1]{raimond99} and the results in the previous lemma we can readily compute 
\begin{align*}
    \lambda_{1} &= \frac {d-3}2 \kiso'(1) - \frac{d-1}2 \kiso(1)\\
    & = \frac {d-3}2\frac{1}{d}\E\left[ \act'\left(\frac{1}{\sqrt d}X+\varbU b\right)^2 \right]- \frac{d-1}2 \left( \E[\act\left(\frac{1}{\sqrt d}X+\varbU b\right)^2]+\varbW^2 \right)\\
    &=\frac{d-3}{4d}-\frac{d-1}2\left(\frac{1}{2}\left(\frac{1}{d}+\varbU^2\right)+\varbW^2\right)\\
    &=\frac{d-3}{4d}-\frac{d-1}{4d}(1+d(\varbU^2+2\varbW^2)) = \frac{-2-d(\varbU^2+2\varbW^2)(d-1)}{4d}<0.
\end{align*}
    Next, we derive the explicit formula for $\bar\lambda$. Using the results from the previous lemma, we compute
    \begin{align*}
        \bar \lambda := & (d-1)\kiso'(1) - (d-3) \kiso(1)\\
        = & (d-1)\frac{1}{d}\E\left[ \act'\left(\frac{1}{\sqrt d}X+\varbU b\right)^2 \right]- (d-3) \left( \E[\act\left(\frac{1}{\sqrt d}X+\varbU b\right)^2]+\varbW^2 \right)\\
        = & \frac{d-1}{2d}-(d-3)\left(\frac{1}{2}\left(\frac{1}{d}+\varbU^2\right)+\varbW^2\right)\\
        =&\frac{d-1}{2d}-\frac{d-3}{2d}(1+d(\varbU^2+2\varbW^2)) = \frac{2-d(\varbU^2+2\varbW^2)(d-3)}{2d}.
    \end{align*}
    This explicitly yields the desired formula and concludes the proof.
\end{proof}

\section*{Acknowledgements}

AA and GB thank N.~Karagodin,  J.~C.~Mattingly, M.~Scardala, and D.~Trevisan for insightful discussions during the preparation of this work. AA and GB acknowledge the partial support of the Swiss National Science Foundation project grant 2000-1-243054. EMG and SS thank the Institute for Mathematical Statistics and Actuarial Sciences at the University of Bern for hosting them at the beginning of this project. 
This research was supported by a Ph.D.\ fellowship funded by the Italian Ministry of University and Research (MUR) under the PNRR program “Transizioni digitali e ambientali (TDA)” (D.M.\ n.~118/2023, CUP~I51J23000400007), within the project “Limiti di scala di dinamiche stocastiche.”
MR acknowledges the partial support of the project PNRR - M4C2 - Investimento 1.3, Partenariato Esteso PE00000013 - \emph{FAIR - Future Artificial Intelligence Research} - Spoke 1 \emph{Human-centered AI}, funded by the European Commission under the NextGeneration EU programme, and of the MUR Excellence Department Project awarded to the Department of Mathematics, University of Pisa, CUP I57G22000700001.

\bibliographystyle{plain}
\bibliography{ref}

\appendix

\section{Technical Lemmas for Section~\ref{sec:deep_lim}}
\label{app:technical-proofs}

\subsection{Proof of Lemma~\ref{lem:regularity_projected_coefficients}}
\label{s:pl24}

We repeatedly use that the truncation map \(T\), the cutoff \(\rho\), and the
truncated projection \(x\mapsto \rproj{x}\) are globally Lipschitz and bounded,
and that \(T(\mathbb R^d)\subset B(0,3)\).

By Assumption~\ref{ass:att_reg}, the attention map is Lipschitz on bounded
sets. Since \(T(x^i),T(y^i)\in B(0,3)\), the regularized attention
\begin{equ}
\rattN{}{x}:=\att{\emp{}{T(x)}}{T(x^i)}
\end{equ}
satisfies
\begin{equ}
\frac1N\sum_{i=1}^N
|\rattN{}{x}-\rattN{}{y}|^2
\le
\mathsf c\,
\frac1N\sum_{i=1}^N |x^i-y^i|^2.
\end{equ}
Moreover, \(\rattN{}{x}\) is uniformly bounded, since it is evaluated on the
fixed compact set \(B(0,3)\).

Recalling that
\begin{equ}
\rdriftN{}{x}
=
\rproj{x^i}\rattN{}{x}
-
\rho(x^i)^2\kerRd(x^i,x^i)\frac{d-1}{2}x^i,
\end{equ}
we note that the map
\begin{equ}
x\longmapsto \rho(x)^2\kerRd(x,x)\frac{d-1}{2}x
\end{equ}
is globally Lipschitz, because \(\rho\) has compact support and \(\kerRd\) is
Lipschitz on \(B(0,3)\). Using the decomposition
\begin{align*}
\rproj{x^i}\rattN{}{x}-\rproj{y^i}\rattN{}{y}
&=
\rproj{x^i}\bigl(\rattN{}{x}-\rattN{}{y}\bigr)\\
&\quad
+
\bigl(\rproj{x^i}-\rproj{y^i}\bigr)\rattN{}{y},
\end{align*}
together with the boundedness and Lipschitz continuity of \(\rproj{\cdot}\),
we obtain
\begin{equ}
|\rdriftN{}{x}-\rdriftN{}{y}|^2
\le
\mathsf c\,|\rattN{}{x}-\rattN{}{y}|^2
+
\mathsf c\,|x^i-y^i|^2.
\end{equ}
Summing over \(i\) and using the previous bound for
\(\rattN{}{x}-\rattN{}{y}\) gives
\begin{equ}
\frac1N\sum_{i=1}^N
|\rdriftN{}{x}-\rdriftN{}{y}|^2
\le
\mathsf c\,
\frac1N\sum_{i=1}^N |x^i-y^i|^2.
\end{equ}

We next use the identities
\begin{equ}
\sum_{k=1}^{\infty} |\rproj{x}\sigma_k(x)|^2
=
\E\bigl[|\rproj{x}G_\ell(x)|^2\bigr],
\qquad
\sum_{k=1}^{\infty} |\rproj{x}\sigma_k(x)-\rproj{y}\sigma_k(y)|^2
=
\E\bigl[|\rproj{x}G_\ell(x)-\rproj{y}G_\ell(y)|^2\bigr].
\end{equ}
If \(x\notin B(0,3)\), then \(\rproj{x}=0\), so the first estimate is
immediate. If \(x\in B(0,3)\), then by boundedness of \(\rproj{x}\) and
Lemma~\ref{lem:unif_gl},
\begin{equ}
\E\bigl[|\rproj{x}G_\ell(x)|^2\bigr]
\le
\|\rproj{x}\|_{\mathrm{op}}^2\,\E[|G_\ell(x)|^2]
\le
\mathsf c,
\end{equ}
which proves
\begin{equ}
\sum_{k=1}^{\infty} |\rproj{x}\sigma_k(x)|^2 \le \mathsf c.
\end{equ}

If both \(x\) and \(y\) lie outside \(B(0,3)\), then
\(\rproj{x}=\rproj{y}=0\), so the Lipschitz estimate is trivial. If
\(x,y\in B(0,3)\), we write
\begin{equ}
\rproj{x}G_\ell(x)-\rproj{y}G_\ell(y)
=
\rproj{x}\bigl(G_\ell(x)-G_\ell(y)\bigr)
+
(\rproj{x}-\rproj{y})G_\ell(y).
\end{equ}
Hence,
\begin{equ}
\E\bigl[|\rproj{x}G_\ell(x)-\rproj{y}G_\ell(y)|^2\bigr]
\le
2\|\rproj{x}\|_{\mathrm{op}}^2\,\E[|G_\ell(x)-G_\ell(y)|^2]
+
2\|\rproj{x}-\rproj{y}\|_{\mathrm{op}}^2\,\E[|G_\ell(y)|^2].
\end{equ}
The first term is bounded by \(\mathsf c|x-y|^2\) using
Lemma~\ref{lem:unif_gl}; the second is bounded by \(\mathsf c|x-y|^2\) using
the Lipschitz continuity of \(x\mapsto \rproj{x}\) and the moment bound for
\(G_\ell\) on \(B(0,3)\). This proves the estimate when \(x,y\in B(0,3)\).

Finally, if one point lies in \(B(0,3)\) and the other outside, say
\(x\in B(0,3)\) and \(y\notin B(0,3)\), then either \(|x-y|\ge1\), in which
case the bound follows from the uniform growth estimate, or one inserts an
intermediate point on \(\partial B(0,3)\) and applies the previous case
together with the triangle inequality. This concludes the proof.
\qed

\subsection{Proof of Lemma~\ref{lem:norm_res}}\label{app:taylor}

We start by computing the Taylor expansion of the normalization operator to a sufficiently high order:
\begin{lemma}
\label{lem:taylor_norm}
Let $\mathbf{v} \in \mathbb{S}^{d-1}$ and $\mathbf{w} \in \mathbb{R}^{d}$. The layer normalization map $\LN(\mathbf{u}) = \frac{\mathbf{u}}{|\mathbf{u}|}$ satisfies the following expansions:
\begin{align*}
    \LN(\mathbf{v}+\mathbf{w}) &= \mathbf{v} + P_{\mathbf{v}}\mathbf{w} + \mathbf{R}_1, \\
    \LN(\mathbf{v}+\mathbf{w}) &= \mathbf{v} + P_{\mathbf{v}}\mathbf{w} - \langle \mathbf{v}, \mathbf{w} \rangle P_{\mathbf{v}}\mathbf{w} - \frac{1}{2}|P_{\mathbf{v}}\mathbf{w}|^2 \mathbf{v} + \mathbf{R}_2,
\end{align*}
where the remainders are bounded by:
\begin{align*}
     |\mathbf{R}_1| &\le \frac{3 |\mathbf{w}|^2}{\left( \min_{t \in [0,1]} |\mathbf{v} + t\mathbf{w}| \right)^2}, \\
     |\mathbf{R}_2| &\le \frac{6 |\mathbf{w}|^3}{\left( \min_{t \in [0,1]} |\mathbf{v} + t\mathbf{w}| \right)^3}.
\end{align*}
\end{lemma}

\begin{proof}
We expand $\LN(\mathbf{u})$ around the unit vector $\mathbf{v}$ ($|\mathbf{v}|=1$).

\paragraph{First Order Expansion.} 
The differential of the reciprocal norm is $D(|\mathbf{u}|^{-1})[\mathbf{w}] = -|\mathbf{u}|^{-3}\langle \mathbf{u}, \mathbf{w} \rangle$. By the product rule:
\begin{equ}
    D\LN(\mathbf{u})[\mathbf{w}] = D(|\mathbf{u}|^{-1})[\mathbf{w}] \mathbf{u} + |\mathbf{u}|^{-1} D(\mathbf{u})[\mathbf{w}] = -\frac{\langle \mathbf{u}, \mathbf{w} \rangle}{|\mathbf{u}|^3} \mathbf{u} + \frac{1}{|\mathbf{u}|} \mathbf{w}.
\end{equ}
Evaluating at $\mathbf{u}=\mathbf{v}$ yields the first-order term: $D\LN(\mathbf{v})[\mathbf{w}] = \mathbf{w} - \langle \mathbf{v}, \mathbf{w} \rangle \mathbf{v} = P_{\mathbf{v}}\mathbf{w}$.

To bound $\mathbf{R}_1$, we compute the second derivative $D^2\LN(\mathbf{u})[\mathbf{w}, \mathbf{w}]$:
\begin{align*}
    D^2\LN(\mathbf{u})[\mathbf{w}, \mathbf{w}] &= D\left( -|\mathbf{u}|^{-3}\langle \mathbf{u}, \mathbf{w} \rangle \mathbf{u} + |\mathbf{u}|^{-1}\mathbf{w} \right)[\mathbf{w}] \\
    &= \frac{3\langle \mathbf{u}, \mathbf{w} \rangle^2 \mathbf{u}}{|\mathbf{u}|^5} - \frac{|\mathbf{w}|^2 \mathbf{u} + 2\langle \mathbf{u}, \mathbf{w} \rangle \mathbf{w}}{|\mathbf{u}|^3}.
\end{align*}
By Cauchy-Schwarz, $|D^2\LN(\mathbf{u})[\mathbf{w}, \mathbf{w}]| \le 6|\mathbf{w}|^2/|\mathbf{u}|^2$. Applying Taylor's theorem with the Lagrange remainder form $|\mathbf{R}_1| \le \frac{1}{2} \sup_{t\in[0,1]} |D^2\LN(\mathbf{v}+t\mathbf{w})|$ gives the first bound.

\paragraph{Second Order Expansion.} 
Evaluating the second derivative at $\mathbf{u}=\mathbf{v}$:
\begin{equ} D^2\LN(\mathbf{v})[\mathbf{w}, \mathbf{w}] = 3\langle \mathbf{v}, \mathbf{w} \rangle^2 \mathbf{v} - |\mathbf{w}|^2 \mathbf{v} - 2\langle \mathbf{v}, \mathbf{w} \rangle \mathbf{w}. \end{equ}
Using the decomposition $\mathbf{w} = P_{\mathbf{v}}\mathbf{w} + \langle \mathbf{v}, \mathbf{w} \rangle \mathbf{v}$ and $|\mathbf{w}|^2 = |P_{\mathbf{v}}\mathbf{w}|^2 + \langle \mathbf{v}, \mathbf{w} \rangle^2$, the second-order term $\mathbf{T}_2 = \frac{1}{2}D^2\LN(\mathbf{v})$ simplifies to:
\begin{equ} \mathbf{T}_2 = - \langle \mathbf{v}, \mathbf{w} \rangle P_{\mathbf{v}}\mathbf{w} - \frac{1}{2}|P_{\mathbf{v}}\mathbf{w}|^2 \mathbf{v}. \end{equ}

For $\mathbf{R}_2$, we differentiate $D^2\LN(\mathbf{u})$ to find the third derivative:
\[
D^3\LN(\mathbf u)[\mathbf w,\mathbf w,\mathbf w]
=
-15
\frac{\langle \mathbf u,\mathbf w\rangle^3}{|\mathbf u|^7}\mathbf u
+
9
\frac{\langle \mathbf u,\mathbf w\rangle|\mathbf w|^2}{|\mathbf u|^5}\mathbf u
+
9
\frac{\langle \mathbf u,\mathbf w\rangle^2}{|\mathbf u|^5}\mathbf w
-
3
\frac{|\mathbf w|^2}{|\mathbf u|^3}\mathbf w .
\]
Grouping terms by powers of $|\mathbf{u}|$ yields:
\begin{enumerate}
    \item Terms of $O(|\mathbf{u}|^{-3})$ contribute $3|\mathbf{w}|^3$.
    \item Terms of $O(|\mathbf{u}|^{-5})$ contribute $18|\mathbf{w}|^3$.
    \item Terms of $O(|\mathbf{u}|^{-7})$ contribute $15|\mathbf{w}|^3$.
\end{enumerate}
Summing these gives $|D^3 \LN(\mathbf{u})| \le 36|\mathbf{w}|^3/|\mathbf{u}|^3$. The Lagrange remainder $|\mathbf{R}_2| \le \frac{1}{3!} \sup |D^3 \LN|$ yields the constant $36/6 = 6$, completing the proof.
\end{proof}

We now proceed to provide the proof of Lemma~\ref{lem:norm_res}:

\begin{proof}[Proof of Lemma~\ref{lem:norm_res}]
First, we expand the intermediate step $Y^i_{\ell}$. Using the first-order Taylor expansion of the layer normalization map $\LN(v) = v/|v|$, we have:
\begin{equ}
\label{eq:Y_X_E_ATT}
Y^i_{\ell} = X^i_\ell + \dt \, P_{X^i_\ell}\attN{\ell}{X} + \mathcal{E}_{\text{att}},
\end{equ}
where the error term $\mathcal{E}_{\text{att}}$ corresponds to the generic remainder $\mathbf{R}_1$ from Lemma \ref{lem:taylor_norm} and can therefore be controlled, assuming $\dt \leq (2\mathsf c)^{-1}$, as:
\begin{equ}
|\mathcal{E}_{\mathrm{att}}| \leq \frac{3\dt^2 |\attN{\ell}{X}|^2}{\left(\min_{t\in[0,1]}|X^i_\ell +t \dt \attN{\ell}{X}|\right)^2 } \leq 12\mathsf c\dt^2.
\end{equ}
Next, using the defined noise increment $\mathbf{w} = \sqrt{\dt}G_{\ell+1}^m (Y^i_{\ell})$, the second transformer update is $X^i_{\ell+1} = \LN(Y^i_{\ell} + \mathbf{w})$. Expanding this around $Y^i_{\ell}$ via Lemma \ref{lem:taylor_norm} yields:
\begin{equ}
\label{eq:X_step_expansion}
X^i_{\ell+1} = Y^i_{\ell} + P_{Y^i_{\ell}} \mathbf{w} - \frac{1}{2}| P_{Y^i_{\ell}} \mathbf{w} |^2 Y^i_{\ell} - ((Y^i_{\ell})^\top \mathbf{w}) P_{Y^i_{\ell}} \mathbf{w} + \mathcal{E}_{\text{mlp}},
\end{equ}
where the MLP error is identically bounded by Lemma \ref{lem:split_mlp_remainder} as:
\begin{equ}
|\mathcal{E}_{\mathrm{mlp}}| \leq \mathsf c |\mathbf{w}|^3 
\end{equ}

We now substitute the expression for $Y^i_{\ell}$ into \eqref{eq:X_step_expansion} to express all projections and vector fields in terms of the base point $X^i_\ell$. From \eqref{eq:Y_X_E_ATT}, we know $|Y^i_{\ell} - X^i_{\ell}|\leq \mathsf c \dt$. Furthermore, because $Y^i_{\ell}$ is the output of the first layer normalization, it lies exactly on the unit sphere ($|Y^i_{\ell}| = 1$). Combining the expansions results in:
\begin{equ}
X^i_{\ell+1} = X^i_{\ell} + \dt P_{X^{i}_\ell}\attN{\ell}{X} + P_{X^i_{\ell}} \mathbf{w} - \frac{1}{2}| P_{X^i_{\ell}} \mathbf{w} |^2 X^i_{\ell} - ((X^i_{\ell})^\top \mathbf{w}) P_{X^i_{\ell}} \mathbf{w} + \mathcal{E}_{\text{total}},
\end{equ}
where the total error satisfies $|\mathcal{E}_{\text{total}}| \leq |\mathcal{E}_{\text{att}}| + |\mathcal{E}_{\text{mlp}}| + |\mathcal{E}_{\text{base}}|$. The base-change error $\mathcal{E}_{\text{base}}$ is defined as:
\begin{align*}
    \mathcal{E}_{\text{base}} &:= \left(P_{Y^i_{\ell}}\mathbf{w} -P_{X^i_{\ell}} \mathbf{w}\right) 
    - \frac{1}{2} \left( | P_{Y^i_{\ell}} \mathbf{w} |^2 Y^i_{\ell} - | P_{X^i_{\ell}} \mathbf{w} |^2 X^i_{\ell} \right) \\
    &\quad - \left( ((Y^i_{\ell})^\top \mathbf{w}) P_{Y^i_{\ell}}\mathbf{w} - ((X^i_{\ell})^\top \mathbf{w}) P_{X^i_{\ell}}\mathbf{w} \right).
\end{align*}
We bound each term in $\mathcal{E}_{\text{base}}$ using standard Lipschitz properties:
\begin{enumerate}
    \item Using the projection bound $|P_u - P_v| \leq 2|u-v|$,
    \begin{equ}
    |P_{Y^i_{\ell}}\mathbf{w} - P_{X^i_{\ell}} \mathbf{w}| \leq 2|Y^i_{\ell} - X^i_{\ell}| |\mathbf{w}| \leq 2 |\mathbf{w}| \mathsf c \dt.
    \end{equ}
    
    \item Let $A(u) = \frac{1}{2}| P_u \mathbf{w} |^2$. Using the decomposition $|A(Y)Y - A(X)X| \leq |A(Y)-A(X)||Y| + |A(X)||Y-X|$ and recalling $|Y^i_{\ell}|=1$:
    \begin{equ}
    \left| \frac{1}{2}| P_{Y^i_{\ell}} \mathbf{w} |^2 - \frac{1}{2}| P_{X^i_{\ell}} \mathbf{w} |^2 \right| 
    = \frac{1}{2} \left|((X^i_{\ell})^\top \mathbf{w})^2 - ((Y^i_{\ell})^\top \mathbf{w})^2\right| 
    \leq |\mathbf{w}|^2 \mathsf c \dt.
    \end{equ}
    Since $|A(X)| \leq \frac{1}{2}|\mathbf{w}|^2$, the total bound for this term is:
    \begin{equ}
    |\mathbf{w}|^2 \mathsf c \dt + \frac{1}{2}|\mathbf{w}|^2 \mathsf c \dt = \frac{3}{2}|\mathbf{w}|^2 \mathsf c \dt.
    \end{equ}
    
    \item Using the product rule on $((u)^\top \mathbf{w}) P_{u}\mathbf{w}$:
    \begin{align*}
    &\left|((Y^i_{\ell})^\top \mathbf{w}) P_{Y^i_{\ell}}\mathbf{w} - ((X^i_{\ell})^\top \mathbf{w}) P_{X^i_{\ell}}\mathbf{w}\right| \\
    &\quad \leq |(Y^i_{\ell} - X^i_{\ell})^\top \mathbf{w}| |P_{Y^i_{\ell}}\mathbf{w}| + |(X^i_{\ell})^\top \mathbf{w}| |P_{Y^i_{\ell}}\mathbf{w} - P_{X^i_{\ell}}\mathbf{w}| \\
    &\quad \leq (\mathsf c \dt |\mathbf{w}|)(|\mathbf{w}|) + (|\mathbf{w}|)(2|\mathbf{w}| \mathsf c \dt) = 3|\mathbf{w}|^2 \mathsf c \dt.
    \end{align*}
\end{enumerate}
Summing these contributions yields $|\mathcal{E}_{\text{base}}| \leq 2|\mathbf{w}| \mathsf c \dt + \frac{9}{2}|\mathbf{w}|^2 \mathsf c \dt$.

Finally, the current expression still depends on $\mathbf{w} = \sqrt{\dt}G_{\ell+1}^m (Y^i_{\ell})$. We substitute this with the point-matched noise $\mathbf{w}' := \sqrt{\dt}G_{\ell+1}^m (X^i_{\ell})$. Let $\delta := \mathbf{w} - \mathbf{w}'$. The error $\mathcal{E}_{\text{noise}}$ resulting from this substitution is bounded by:
\begin{enumerate}
    \item Linear term: 
    \begin{equ}
        |P_{X^i_{\ell}}\mathbf{w} - P_{X^i_{\ell}}\mathbf{w}'| \leq |\delta|.
    \end{equ}
    \item First quadratic term: 
    \begin{equ}
    \left| \frac{1}{2}| P_{X^i_{\ell}} \mathbf{w} |^2 X^i_{\ell} - \frac{1}{2}| P_{X^i_{\ell}} \mathbf{w}' |^2 X^i_{\ell} \right| \leq |\mathbf{w}||\delta| + \frac{1}{2}|\delta|^2.
    \end{equ}
    \item Second quadratic term:
    \begin{equ}
    \left| ((X^i_{\ell})^\top \mathbf{w}) P_{X^i_{\ell}} \mathbf{w} - ((X^i_{\ell})^\top \mathbf{w}') P_{X^i_{\ell}} \mathbf{w}' \right| \leq 2|\mathbf{w}||\delta| + |\delta|^2.
    \end{equ}
\end{enumerate}
Summing these yields $|\mathcal{E}_{\text{noise}}| \leq (1 + 3|\mathbf{w}|)|\delta| + \frac{3}{2}|\delta|^2$. Collecting all terms ($\mathcal{R}_\ell^i = \mathcal{E}_{\text{att}} + \mathcal{E}_{\text{mlp}} + \mathcal{E}_{\text{base}} + \mathcal{E}_{\text{noise}}$) concludes the proof.
\end{proof}
\begin{lemma}
\label{lem:split_mlp_remainder}
    Let $\mathbf{w} := \sqrt{\dt} G_{\ell+1}^m(Y_{\ell}^i)$ and let $\mathcal{E}_{\mathrm{mlp}}$ be the corresponding MLP remainder from \eqref{eq:X_step_expansion}. Then
    \begin{equ}
        |\mathcal{E}_{\mathrm{mlp}}| \leq \mathsf c |\mathbf{w}|^3
    \end{equ}
    almost surely. 
\end{lemma}
\begin{proof}
    Let $A_\ell^i := \{|\mathbf{w}|\leq 1/4\}$. Since $|Y_{\ell}^i|=1$, on $A_\ell^i$ we have
    \begin{equ}
        \min_{t\in[0,1]}|Y_{\ell}^i+t\mathbf{w}|\geq 1-|\mathbf{w}|\geq \frac34,
    \end{equ}
    and therefore Lemma \ref{lem:taylor_norm} gives
    \begin{equ}
        |\mathcal{E}_{\mathrm{mlp}}|\mathbf 1_{A_\ell^i}\leq \mathsf c |\mathbf{w}|^3 \mathbf 1_{A_\ell^i}.
    \end{equ}
    On the complementary event $(A_\ell^i)^c$, we use the exact identity
    \begin{equ}
        \mathcal{E}_{\mathrm{mlp}}=X_{\ell+1}^i-Y_{\ell}^i-P_{Y_{\ell}^i}\mathbf{w}
        +\frac12|P_{Y_{\ell}^i}\mathbf{w}|^2Y_{\ell}^i
        +((Y_{\ell}^i)^\top \mathbf{w})P_{Y_{\ell}^i}\mathbf{w}.
    \end{equ}
    Since $|X_{\ell+1}^i|=|Y_{\ell}^i|=1$, $|P_{Y_{\ell}^i}\mathbf{w}|\leq |\mathbf{w}|$, and $|(Y_{\ell}^i)^\top \mathbf{w}|\leq |\mathbf{w}|$, this yields
    \begin{equ}
        |\mathcal{E}_{\mathrm{mlp}}| \leq 2+|\mathbf{w}|+\frac32|\mathbf{w}|^2.
    \end{equ}
    Since $|\mathbf{w}| \geq 1/4$ on $(A_\ell^i)^c$, the right-hand side is bounded by $\mathsf c |\mathbf{w}|^3$. Hence
    \begin{equ}
        |\mathcal{E}_{\mathrm{mlp}}|\mathbf 1_{(A_\ell^i)^c}\leq \mathsf c |\mathbf{w}|^3 \mathbf 1_{(A_\ell^i)^c}.
    \end{equ}
    Combining the two regions yields
    \begin{equ}
        |\mathcal{E}_{\mathrm{mlp}}| \leq \mathsf c |\mathbf{w}|^3
    \end{equ}
    almost surely. 
\end{proof}

\subsection{Lemmas for Section~\ref{sec:strong-single-step}}

We prove here the version of the Euler-Maruyama estimate used in Section~\ref{sec:strong-single-step}.

\begin{lemma}
\label{lem:increment_delta_appendix}
    Under the assumptions of Proposition~\ref{prop:euler_maruyama}, for any $0 \leq u \leq t \leq T$ and any particle $i \in \{1, \dots, N\}$,
    \begin{equ}
        \E\left[|X^i_t - X^i_u|^2\right] \leq 2 K_3 (t-u)(1 + t-u).
    \end{equ}
\end{lemma}
\begin{proof}
    We use the integral formulation of the true process $X^i$. By the algebraic inequality $(a+b)^2 \leq 2a^2 + 2b^2$, the squared increment is bounded by:
    \begin{equation*}
        |X^i_t - X^i_u|^2 \leq 2 \left| \int_u^t a(X^i_s, \mathbf{X}_s) \, \dd s \right|^2 + 2 \left| \sum_{k=1}^\infty \int_u^t \tilde \sigma_k(X^i_s) \, dB^k_s \right|^2.
    \end{equation*}
    Taking expectations, we apply the Cauchy-Schwarz inequality to the drift term and the It\^o isometry to the diffusion term:
    \begin{equation*}
        \E\left[|X^i_t - X^i_u|^2\right] \leq 2(t-u) \int_u^t \E\left[|a(X^i_s, \mathbf{X}_s)|^2\right] \dd s + 2 \int_u^t \E\left[\sum_{k=1}^\infty |\tilde \sigma_k(X^i_s)|^2\right] \dd s.
    \end{equation*}
    By the uniform growth assumption \eqref{eq:em_growth} from Proposition~\ref{prop:euler_maruyama}, we have $|a|^2 + \sum_{k=1}^\infty |\tilde \sigma_k|^2 \leq K_3$, which implies both the drift magnitude squared and the diffusion variance are individually bounded by $K_3$. Thus:
    \begin{align*}
        \E\left[|X^i_t - X^i_u|^2\right] &\leq 2(t-u) \int_u^t K_3 \, \dd s + 2 \int_u^t K_3 \, \dd s \\
        &\leq 2 K_3 (t-u)^2 + 2 K_3 (t-u) \\
        &= 2 K_3 (t-u)(1 + t-u).
    \end{align*}
\end{proof}

To avoid confusion with the Karhunen-Lo\`eve modes appearing in the main text, in the result below we denote the diffusion coefficients in the appendix proposition by $(\tilde \sigma_k)_k$.

\begin{proposition}
\label{prop:euler_maruyama}
    Let $T > 0$ and step size $\dt = T/L \leq 1$. Let the discrete time points be $t_\ell = \ell \dt$ for $\ell=0, \dots, L$.
    Consider a system of $N$ particles governed by the SDE
    \begin{equ}
        dX^i_t = a(X^i_t, \mathbf{X}_t)\,dt + \sum_{k=1}^\infty \tilde \sigma_k(X^i_t)\, dB^k_t, \qquad X^i_0 = x^i_0,
    \end{equ}
    and its continuous-time Euler-Maruyama interpolation
    \begin{equ}
        \hat X^i_t = x^i_0 + \int_0^t a(\hat X^i_{t_{\ell_s}},\hat{\mathbf X}_{t_{\ell_s}})\,\dd s + \sum_{k=1}^\infty \int_0^t \tilde \sigma_k(\hat X^i_{t_{\ell_s}})\,\dd B_s^k.
    \end{equ}
    Assume that there exist constants $K_2,K_3>0$ such that for all $N\in\mathbb N$ and $x,y\in(\R^d)^N$,
    \begin{align}
        \E[|X_0^i|^2] &< \infty, \\
        \frac1N\sum_{i=1}^N \left(|a(x^i,\mathbf x)-a(y^i,\mathbf y)|^2 + \sum_{k=1}^\infty |\tilde \sigma_k(x^i)-\tilde \sigma_k(y^i)|^2\right)
        &\le K_2 \frac1N\sum_{i=1}^N |x^i-y^i|^2, \label{eq:em_lipschitz} \\
        |a(x^i,\mathbf x)|^2 + \sum_{k=1}^\infty |\tilde \sigma_k(x^i)|^2 &\le K_3. \label{eq:em_growth}
    \end{align}
    Then
    \begin{equ}
    \E\left[\sup_{0\leq t\leq T}\frac{1}{N}\sum_{i=1}^N |X^i_t - \hat{X}^i_t|^2\right]
    \leq 16 K_2 K_3 (4T+T^2) e^{4 K_2 (4+T) T} \dt.
    \end{equ}
\end{proposition}

\begin{proof}[Proof of Proposition \ref{prop:euler_maruyama}]
    Let $Z(t) := \E\left[\sup_{0\leq s \leq t} \frac{1}{N}\sum_{i=1}^N |X^i_s - \hat{X}^i_s|^2\right]$.
    By subtracting the integral formulations of $X^i_t$ and $\hat{X}^i_t$, we have:
    \begin{align*}
    X^i_s - \hat{X}^i_s &= \int_0^s \left(a(X^i_u, \mathbf{X}_u) - a(\hat{X}^i_{t_{\ell_u}}, \hat{\mathbf{X}}_{t_{\ell_u}})\right) \dd u \\
    &\quad + \sum_{k=1}^\infty \int_0^s \left(\tilde \sigma_k(X^i_u) - \tilde \sigma_k(\hat{X}^i_{t_{\ell_u}})\right) \dd B^k_u.
    \end{align*}
    We decompose the error by introducing the coefficients evaluated at the \textit{true} process frozen at the grid point $t_{\ell_u}$. Using the algebraic inequality $(a+b+c+d)^2 \leq 4(a^2+b^2+c^2+d^2)$, we bound the mean squared error by four distinct terms:
    \begin{equation*}
        Z(t) \leq 4 \left( A^1_t + A^2_t + B^1_t + B^2_t \right). 
    \end{equation*}

    \textbf{1. Bounding $A^1_t$ (Spatial discretization error):}
    \begin{equation*}
        A^1_t := \E\left[\sup_{0\leq s \leq t} \frac{1}{N}\sum_{i=1}^N \left| \int_0^s \left( a(X^i_{t_{\ell_u}}, \mathbf{X}_{t_{\ell_u}}) - a(\hat{X}^i_{t_{\ell_u}}, \hat{\mathbf{X}}_{t_{\ell_u}}) \right) \dd u \right|^2 \right].
    \end{equation*}
    Using the Cauchy-Schwarz inequality ($\left|\int_0^s f \, \dd u\right|^2 \leq s \int_0^s |f|^2 \, \dd u \leq T \int_0^t |f|^2 \, \dd u$) and the global Lipschitz assumption \eqref{eq:em_lipschitz}:
    \begin{align*}
        A^1_t &\leq T \int_0^t \frac{1}{N}\sum_{i=1}^N \E\left[ \left| a(X^i_{t_{\ell_u}}, \mathbf{X}_{t_{\ell_u}}) - a(\hat{X}^i_{t_{\ell_u}}, \hat{\mathbf{X}}_{t_{\ell_u}}) \right|^2 \right] \dd u \\
        &\leq T K_2 \int_0^t \E\left[ \frac{1}{N}\sum_{i=1}^N |X^i_{t_{\ell_u}} - \hat{X}^i_{t_{\ell_u}}|^2 \right] \dd u.
    \end{align*}
    Since $|X^i_{t_{\ell_u}} - \hat{X}^i_{t_{\ell_u}}|^2 \leq \sup_{r \leq u} |X^i_r - \hat{X}^i_r|^2$, we have $A^1_t \leq T K_2 \int_0^t Z(u) \, \dd u$.

    \textbf{2. Bounding $A^2_t$ (Time regularity of the drift):}
    \begin{equation*}
        A^2_t := \E\left[\sup_{0\leq s \leq t} \frac{1}{N}\sum_{i=1}^N \left| \int_0^s \left( a(X^i_u, \mathbf{X}_u) - a(X^i_{t_{\ell_u}}, \mathbf{X}_{t_{\ell_u}}) \right) \dd u \right|^2 \right].
    \end{equation*}
    Using Cauchy-Schwarz and the Lipschitz assumption \eqref{eq:em_lipschitz}:
    \begin{equation*}
        A^2_t \leq T K_2 \int_0^t \frac{1}{N}\sum_{i=1}^N \E\left[ |X^i_u - X^i_{t_{\ell_u}}|^2 \right] \dd u.
    \end{equation*}
    By Lemma~\ref{lem:increment_delta_appendix} and the assumption $\dt \le 1$, $\E[ |X^i_u - X^i_{t_{\ell_u}}|^2 ] \leq 2K_3 \dt(1+\dt) \leq 4 K_3 \dt$. Thus:
    \begin{equation*}
        A^2_t \leq T K_2 \int_0^t 4 K_3 \dt \, \dd u = 4 T^2 K_2 K_3 \dt.
    \end{equation*}

    \textbf{3. Bounding $B^1_t$ (Martingale spatial error):}
    \begin{equation*}
        B^1_t := \E\left[\sup_{0\leq s \leq t} \frac{1}{N}\sum_{i=1}^N \left| \sum_{k=1}^\infty \int_0^s \left( \tilde \sigma_k(X^i_{t_{\ell_u}}) - \tilde \sigma_k(\hat{X}^i_{t_{\ell_u}}) \right) \dd B^k_u \right|^2 \right].
    \end{equation*}
    Using Doob's maximal inequality and the It\^o isometry:
    \begin{align*}
        B^1_t &\leq 4 \frac{1}{N}\sum_{i=1}^N \E\left[ \int_0^t \sum_{k=1}^\infty \left| \tilde \sigma_k(X^i_{t_{\ell_u}}) - \tilde \sigma_k(\hat{X}^i_{t_{\ell_u}}) \right|^2 \dd u \right] \\
        &\leq 4 K_2 \int_0^t \E\left[ \frac{1}{N}\sum_{i=1}^N |X^i_{t_{\ell_u}} - \hat{X}^i_{t_{\ell_u}}|^2 \right] \dd u \leq 4 K_2 \int_0^t Z(u) \, \dd u.
    \end{align*}

    \textbf{4. Bounding $B^2_t$ (Time regularity of the diffusion):}
    \begin{equation*}
        B^2_t := \E\left[\sup_{0\leq s \leq t} \frac{1}{N}\sum_{i=1}^N \left| \sum_{k=1}^\infty \int_0^s \left( \tilde \sigma_k(X^i_u) - \tilde \sigma_k(X^i_{t_{\ell_u}}) \right) \dd B^k_u \right|^2 \right].
    \end{equation*}
    Again using Doob's inequality, the It\^o isometry, and the Lipschitz condition \eqref{eq:em_lipschitz}:
    \begin{equation*}
        B^2_t \leq 4 K_2 \int_0^t \frac{1}{N}\sum_{i=1}^N \E\left[ |X^i_u - X^i_{t_{\ell_u}}|^2 \right] \dd u.
    \end{equation*}
    Applying the bound from Lemma~\ref{lem:increment_delta_appendix} yields:
    \begin{equation*}
        B^2_t \leq 4 K_2 \int_0^t 4 K_3 \dt \, \dd u = 16 T K_2 K_3 \dt.
    \end{equation*}
    
    \textbf{Conclusion:}
    Substituting the four bounds back into the inequality for $Z(t)$:
    \begin{align*}
        Z(t) &\leq 4 \left( T K_2 \int_0^t Z(u) \, \dd u + 4 T^2 K_2 K_3 \dt + 4 K_2 \int_0^t Z(u) \, \dd u + 16 T K_2 K_3 \dt \right) \\
        &= 16 K_2 K_3 (4T + T^2) \dt + 4 K_2(T + 4) \int_0^t Z(u) \, \dd u.
    \end{align*}
    Applying the continuous Gr\"onwall inequality $u(t) \le \alpha + \beta \int_0^t u(s) \dd s \implies u(t) \le \alpha e^{\beta t}$, we obtain:
    \begin{equation*}
        Z(T) \leq 16 K_2 K_3 (4T + T^2) \dt \, e^{4 K_2(4+T) T},
    \end{equation*}
    which concludes the proof.
\end{proof}

\section{Numerical Experiments}
\label{sec:experiments}

In this section, we numerically simulate the discrete transformer dynamics \eqref{eq:transformer-update} to verify the synchronization by noise phenomena established in Section \ref{sec:intro-sync-by-noise}. 

We track $N=64$ tokens through $L=150$ layers up to a terminal time $T=5.0$, with an inverse temperature $\beta=1.0$. The tokens $(X_0^i)_{i=1}^N$ are initialized \iid\ from the uniform distribution on $\mathbb{S}^{d-1}$. Figures \ref{fig:relu}, \ref{fig:silu}, \ref{fig:sigmoid}, \ref{fig:tanh}, \ref{fig:no_mlp} show the results of the experiments with different choices of feature dimensions $d \in \{5, 64, 128\}$, MLP bias standard deviations $\sigma \in \{0.0, 0.1\}$ (where we assume $\varbU=\varbW=\sigma$), and common activation functions (ReLU, SiLU, Sigmoid, Tanh). We scale the self-attention update by a factor of $1$ (attractive), $0$ (pure MLP noise), and $-1$ (repulsive). To monitor clustering, we track the pairwise cosine similarity $\langle X_\ell^i, X_\ell^j\rangle$ between tokens and the interaction energy $\mathcal{E}_\beta(\mathbf{X}_\ell)$ defined in \eqref{e:energy}.

As discussed in Section \ref{sec:intro-sync-by-noise}, synchronization by noise is broadly governed by the top Lyapunov exponent of the stochastic flow, $\lambda_1 = ((d-3)/2)\kiso'(1) - ((d-1)/2)\kiso(1)$. A strictly negative $\lambda_1$ ensures asymptotic synchronization. Furthermore, the exponential $L^1$-dissipation bound of Theorem \ref{t:sync} requires the stricter condition $\bar{\lambda} = (d-1)\kiso'(1) - (d-3)\kiso(1) < 0$ (Assumption \ref{ass:sync_dissipation}). Table \ref{tab:lyapunov} summarizes the numerically computed values for these quantities across our experimental grid. In particular it confirms that $\lambda_1 \le 0$ across all tested configurations, indicating that noise-induced synchronization is a generic property of these standard initializations
\begin{table}[ht]
    \centering
    \label{tab:lyapunov}
    \small
    \begin{tabular}{lrrrrrr}
        \toprule
        \textbf{Activation} & $d$ & \textbf{Bias ($\sigma$)} & $\kiso(1)$ & $\kiso'(1)$ & $\lambda_1$ & $\bar{\lambda}$ \\
        \midrule
        ReLU & 5 & 0.0 & 1.00e-01 & 1.00e-01 & -1.00e-01 & 2.00e-01 \\
        ReLU & 5 & 0.1 & 1.15e-01 & 1.00e-01 & -1.30e-01 & 1.70e-01 \\
        ReLU & 64 & 0.0 & 7.81e-03 & 7.81e-03 & -7.81e-03 & 1.56e-02 \\
        ReLU & 64 & 0.1 & 2.28e-02 & 7.81e-03 & -4.80e-01 & -8.99e-01 \\
        ReLU & 128 & 0.0 & 3.91e-03 & 3.91e-03 & -3.91e-03 & 7.81e-03 \\
        ReLU & 128 & 0.1 & 1.89e-02 & 3.91e-03 & -9.56e-01 & -1.87e+00 \\
        \midrule
        Sigmoid & 5 & 0.0 & 2.61e-01 & 1.14e-02 & -5.11e-01 & -4.77e-01 \\
        Sigmoid & 5 & 0.1 & 2.72e-01 & 1.14e-02 & -5.32e-01 & -4.98e-01 \\
        Sigmoid & 64 & 0.0 & 2.51e-01 & 9.69e-04 & -7.88e+00 & -1.52e+01 \\
        Sigmoid & 64 & 0.1 & 2.62e-01 & 9.64e-04 & -8.21e+00 & -1.59e+01 \\
        Sigmoid & 128 & 0.0 & 2.50e-01 & 4.86e-04 & -1.59e+01 & -3.12e+01 \\
        Sigmoid & 128 & 0.1 & 2.61e-01 & 4.84e-04 & -1.65e+01 & -3.26e+01 \\
        \midrule
        Tanh & 5 & 0.0 & 1.47e-01 & 1.51e-01 & -1.44e-01 & 3.09e-01 \\
        Tanh & 5 & 0.1 & 1.63e-01 & 1.49e-01 & -1.76e-01 & 2.72e-01 \\
        Tanh & 64 & 0.0 & 1.52e-02 & 1.52e-02 & -1.50e-02 & 3.06e-02 \\
        Tanh & 64 & 0.1 & 3.44e-02 & 1.49e-02 & -6.29e-01 & -1.16e+00 \\
        Tanh & 128 & 0.0 & 7.69e-03 & 7.69e-03 & -7.66e-03 & 1.55e-02 \\
        Tanh & 128 & 0.1 & 2.72e-02 & 7.55e-03 & -1.26e+00 & -2.44e+00 \\
        \midrule
        SiLU & 5 & 0.0 & 5.65e-02 & 5.84e-02 & -5.45e-02 & 1.21e-01 \\
        SiLU & 5 & 0.1 & 6.96e-02 & 5.87e-02 & -8.04e-02 & 9.57e-02 \\
        SiLU & 64 & 0.0 & 3.95e-03 & 3.97e-03 & -3.50e-03 & 8.84e-03 \\
        SiLU & 64 & 0.1 & 1.65e-02 & 4.00e-03 & -3.98e-01 & -7.56e-01 \\
        SiLU & 128 & 0.0 & 1.96e-03 & 1.97e-03 & -1.73e-03 & 4.41e-03 \\
        SiLU & 128 & 0.1 & 1.45e-02 & 1.99e-03 & -7.97e-01 & -1.56e+00 \\
        \bottomrule
    \end{tabular}
    \caption{Computed kernel values, top Lyapunov exponent ($\lambda_1$), and stability parameter ($\bar{\lambda}$) for various activation functions, dimensions $d$, and bias standard deviations $\sigma$.}
\end{table}

We observe several key phenomena:
\begin{enumerate}
    \item \textbf{Bias accelerates collapse:} Introducing a non-zero bias ($\sigma=0.1$) strictly increases $\kiso(1)$, which strongly penalizes both $\lambda_1$ and $\bar{\lambda}$. For activations like ReLU, Tanh, and SiLU, the bias is necessary to push $\bar{\lambda}$ into the strictly negative regime required by Assumption \ref{ass:sync_dissipation}, triggering the exponential energy dissipation guaranteed by Theorem \ref{t:sync}.
    \item \textbf{Sigmoid is highly unstable:} Because the Sigmoid activation does not pass through the origin ($\act(0) \neq 0$), it acts as an intrinsic, permanent bias. This yields heavily negative $\lambda_1$ and $\bar{\lambda}$ values regardless of explicit parameter bias, leading to rapid token collapse and poor preservation of token diversity.
    \item \textbf{Proximity to the edge of synchronization:} The most widely used activation functions in modern Large Language Models (ReLU, SiLU) exhibit $\lambda_1$ values that are negative but remarkably close to zero when initialized without bias (e.g., $\lambda_1 \approx 0.00$ for SiLU at $d=128$). This structural proximity to the edge of the synchronization regime suggests that avoiding immediate token collapse is closely tied to the practical trainability of deep transformer models.
    \item \textbf{Role of dimension:} Increasing the feature dimension $d$ exhibits a dual effect depending on the bias. Without bias, larger dimensions push $\lambda_1$ closer to zero, effectively slowing down synchronization. Conversely, in the presence of bias, the variance terms scale differently with $d$, causing high-dimensional models to collapse significantly faster.
     \item \textbf{Role of attention}: Attractive attention cooperates with the MLP noise to drive rapid synchronization. Repulsive attention actively resists this collapse. However, consistent with Theorem \ref{t:sync}, if the stochastic noise is sufficiently coercive (e.g., via biased MLPs), the common noise eventually overpowers the repulsive drift, dictating the long-time behavior and forcing clustering.
\end{enumerate}

Finally, to visualize the contractive effect of the MLP noise in isolation, we simulate $N=64$ tokens on $\mathbb{S}^3$ ($d=4$) over $L=2000$ layers ($T=20.0$) with the attention component disabled. Using wide MLPs ($m=2048$) with ReLU and Tanh activations, Figure \ref{fig:sphere_clustering} shows the initially uniform tokens aggregating into clusters, providing a visual verification of the noise-induced synchronization predicted by Theorem \ref{t:sync}.

\begin{figure}[hbt!]
    \centering
\includegraphics[width=0.95\textwidth]{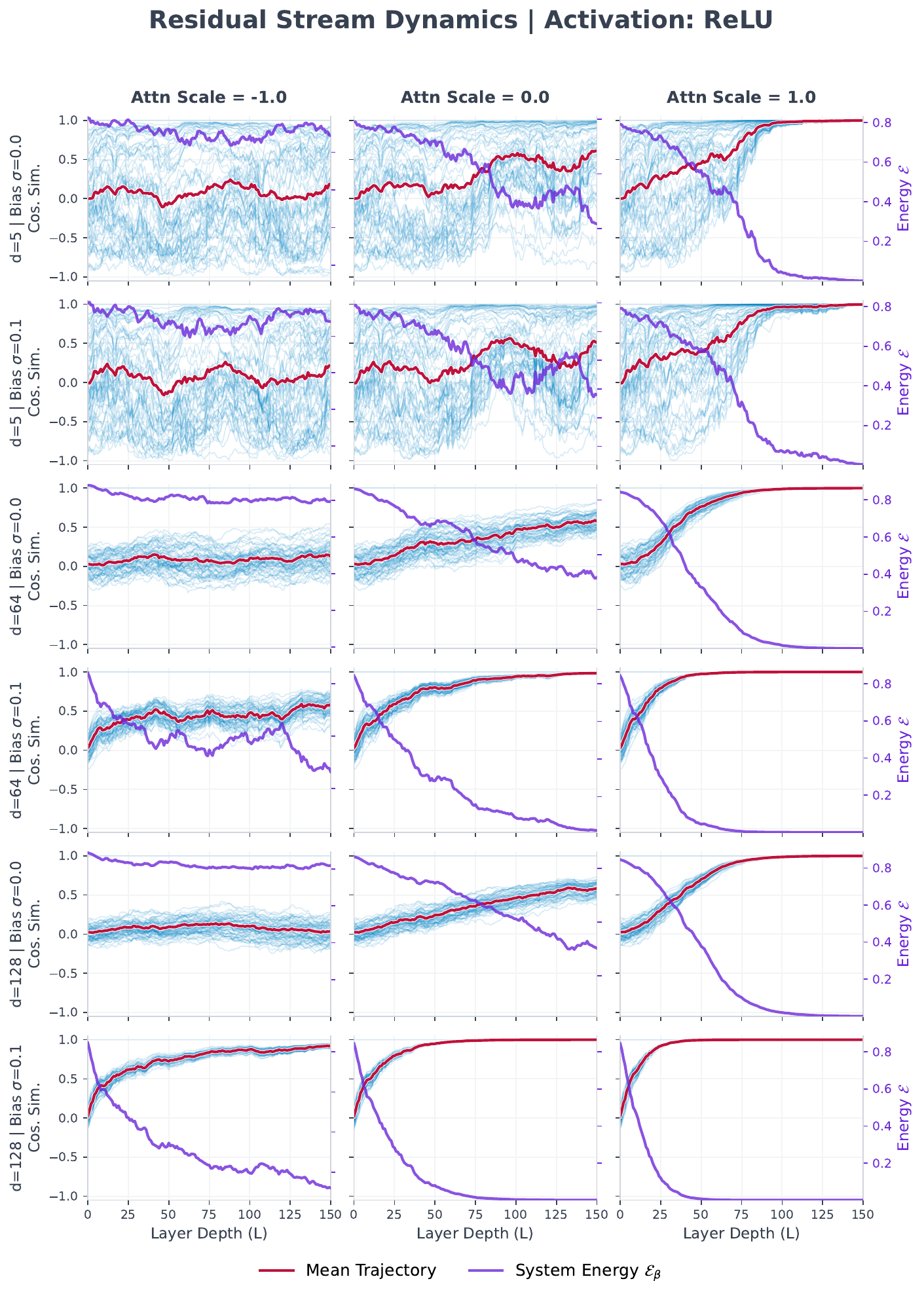}
    \caption{Numerical simulations of the discrete transformer dynamics with ReLU activation, for $N=64$ tokens over $L=150$ layers ($T=5.0$), dimensions $d\in\{5,64,128\}$, bias levels $\sigma\in\{0.0,0.1\}$, and attention scaling $1$, $0$, and $-1$. The panels show the pairwise cosine similarities and the interaction energy $\mathcal{E}_\beta(\mathbf X_\ell)$.}
    \label{fig:relu}
\end{figure}

\newpage
\begin{figure}[hbt!]
    \centering
\includegraphics[width=0.95\textwidth]{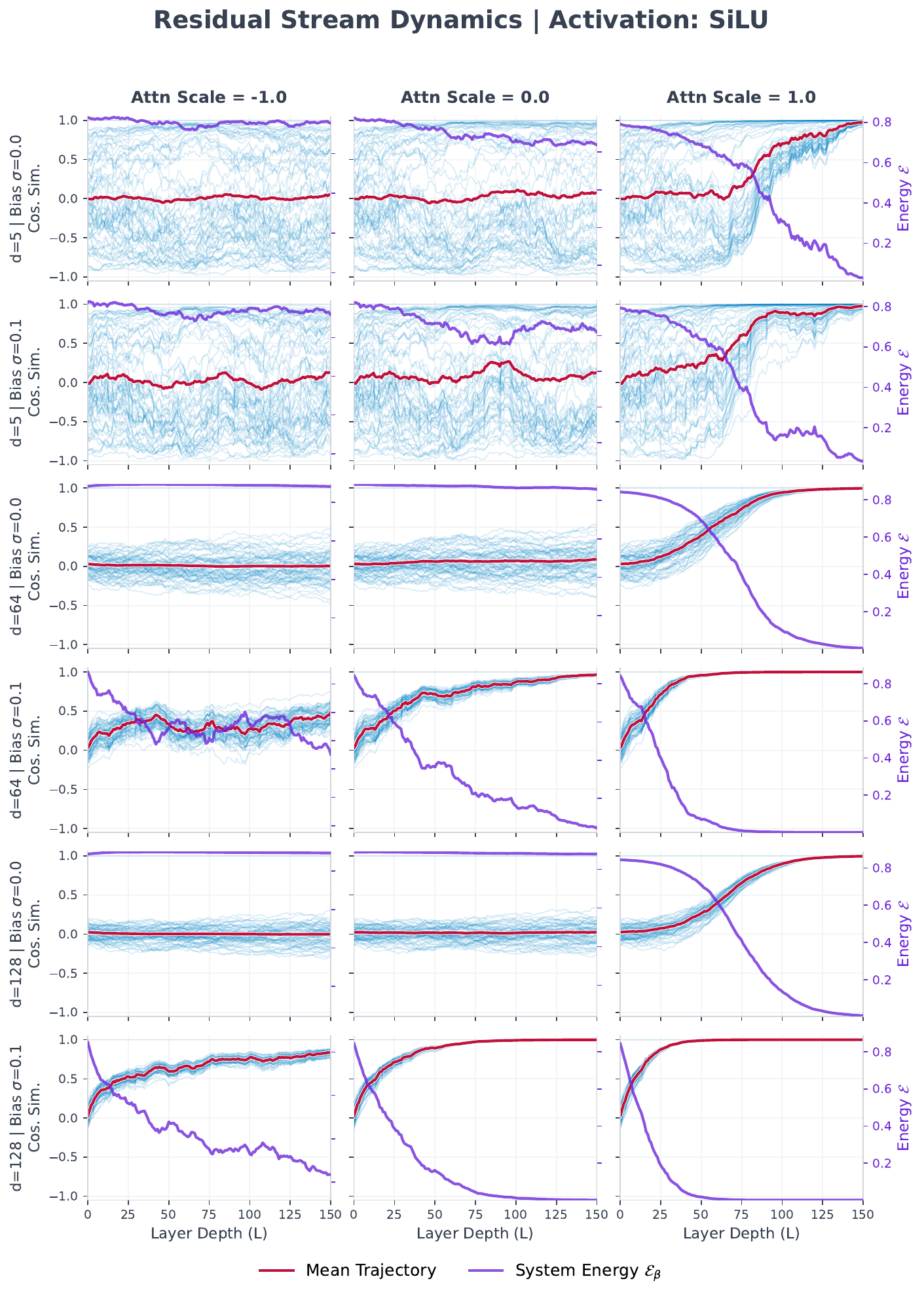}
    \caption{Numerical simulations of the discrete transformer dynamics with SiLU activation, for $N=64$ tokens over $L=150$ layers ($T=5.0$), dimensions $d\in\{5,64,128\}$, bias levels $\sigma\in\{0.0,0.1\}$, and attention scaling $1$, $0$, and $-1$. The panels show the pairwise cosine similarities and the interaction energy $\mathcal{E}_\beta(\mathbf X_\ell)$.}
    \label{fig:silu}
\end{figure}

\newpage
\begin{figure}[hbt!]
    \centering
\includegraphics[width=0.9\textwidth]{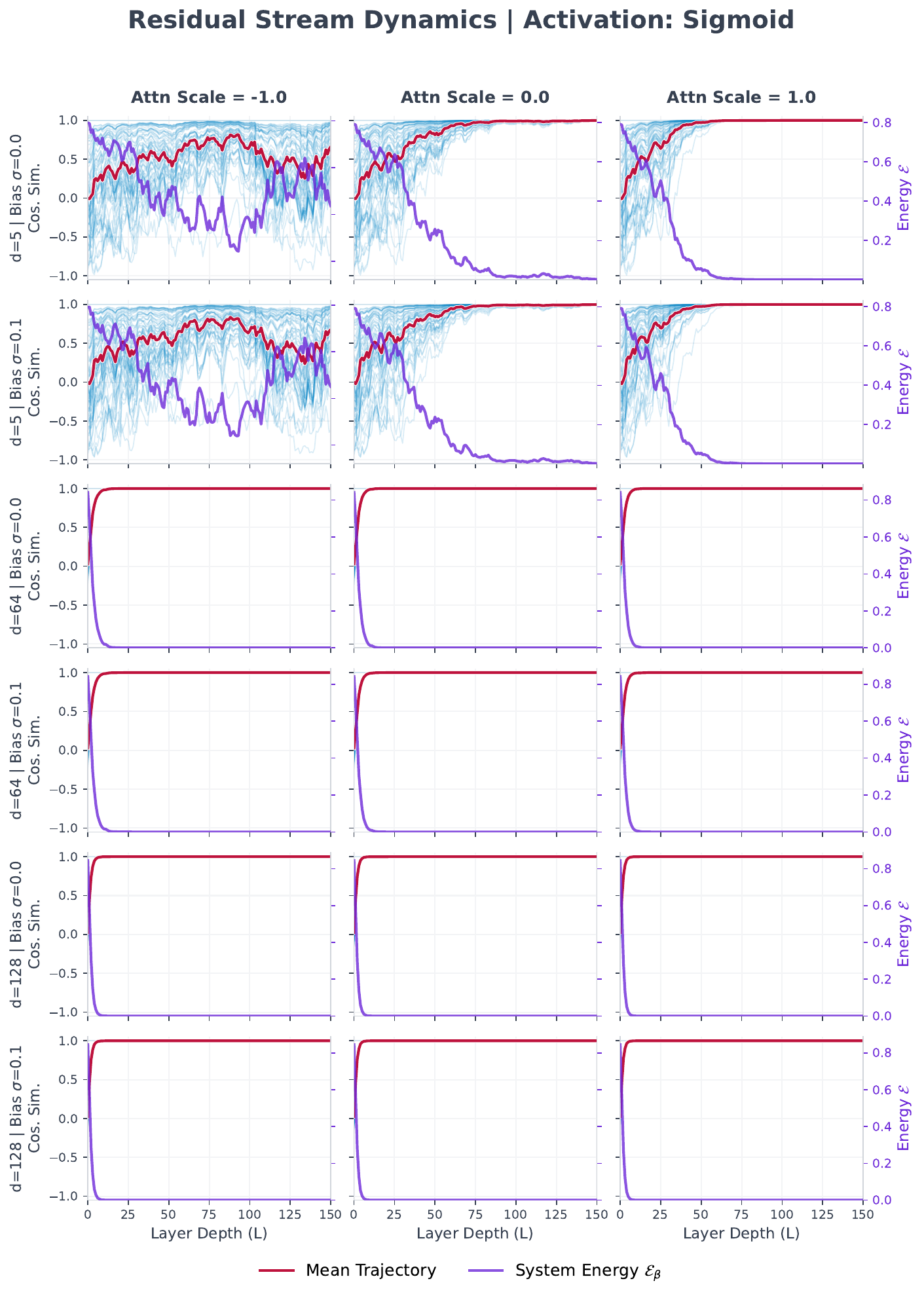}
    \caption{Numerical simulations of the discrete transformer dynamics with Sigmoid activation, for $N=64$ tokens over $L=150$ layers ($T=5.0$), dimensions $d\in\{5,64,128\}$, bias levels $\sigma\in\{0.0,0.1\}$, and attention scaling $1$, $0$, and $-1$. The panels show the pairwise cosine similarities and the interaction energy $\mathcal{E}_\beta(\mathbf X_\ell)$.}
    \label{fig:sigmoid}
\end{figure}

\newpage
\begin{figure}[hbt!]
    \centering
\includegraphics[width=0.95\textwidth]{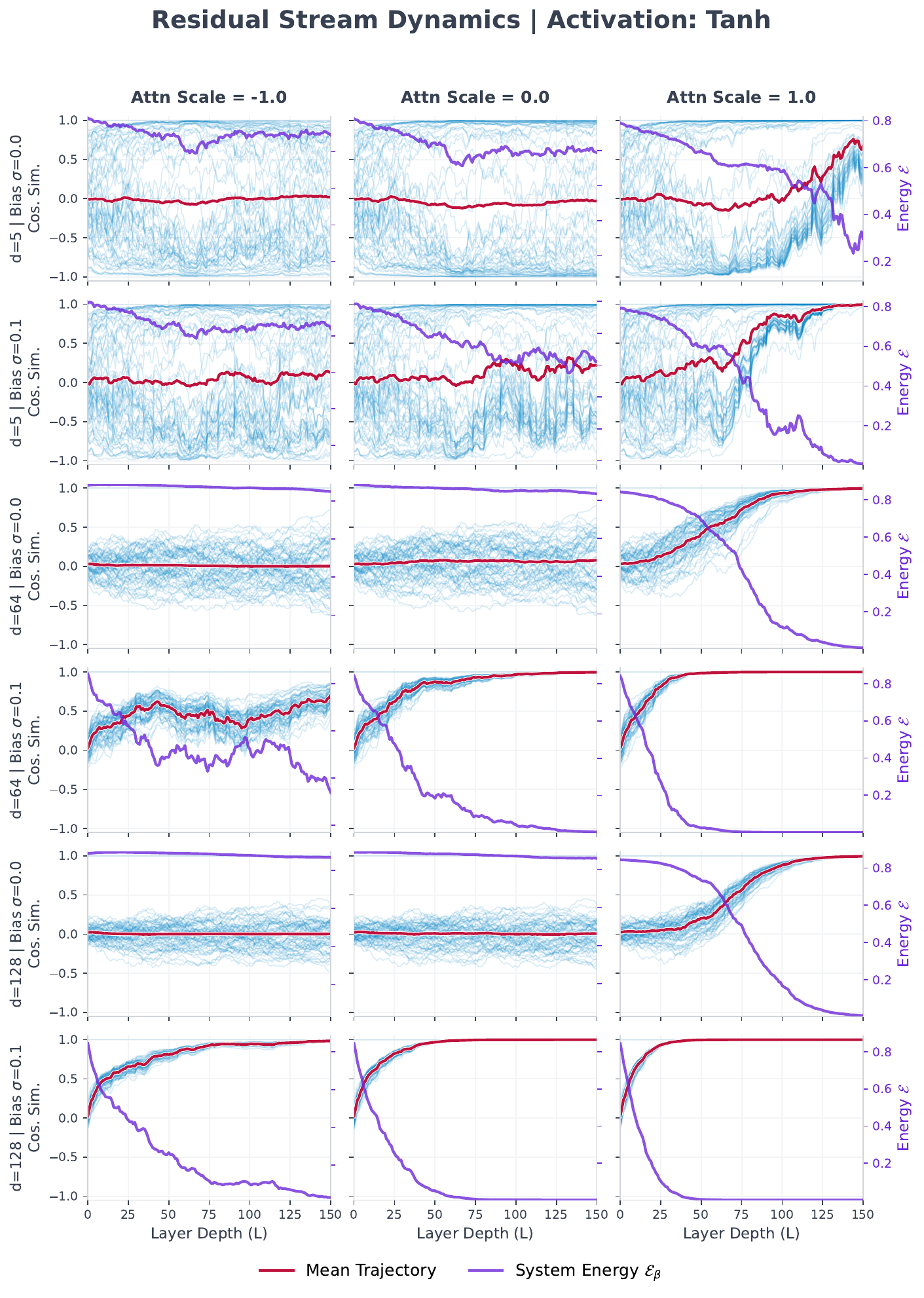}
    \caption{Numerical simulations of the discrete transformer dynamics with Tanh activation, for $N=64$ tokens over $L=150$ layers ($T=5.0$), dimensions $d\in\{5,64,128\}$, bias levels $\sigma\in\{0.0,0.1\}$, and attention scaling $1$, $0$, and $-1$. The panels show the pairwise cosine similarities and the interaction energy $\mathcal{E}_\beta(\mathbf X_\ell)$.}
    \label{fig:tanh}
\end{figure}

\begin{figure}[hbt!]
    \centering
\includegraphics[width=0.95\textwidth]{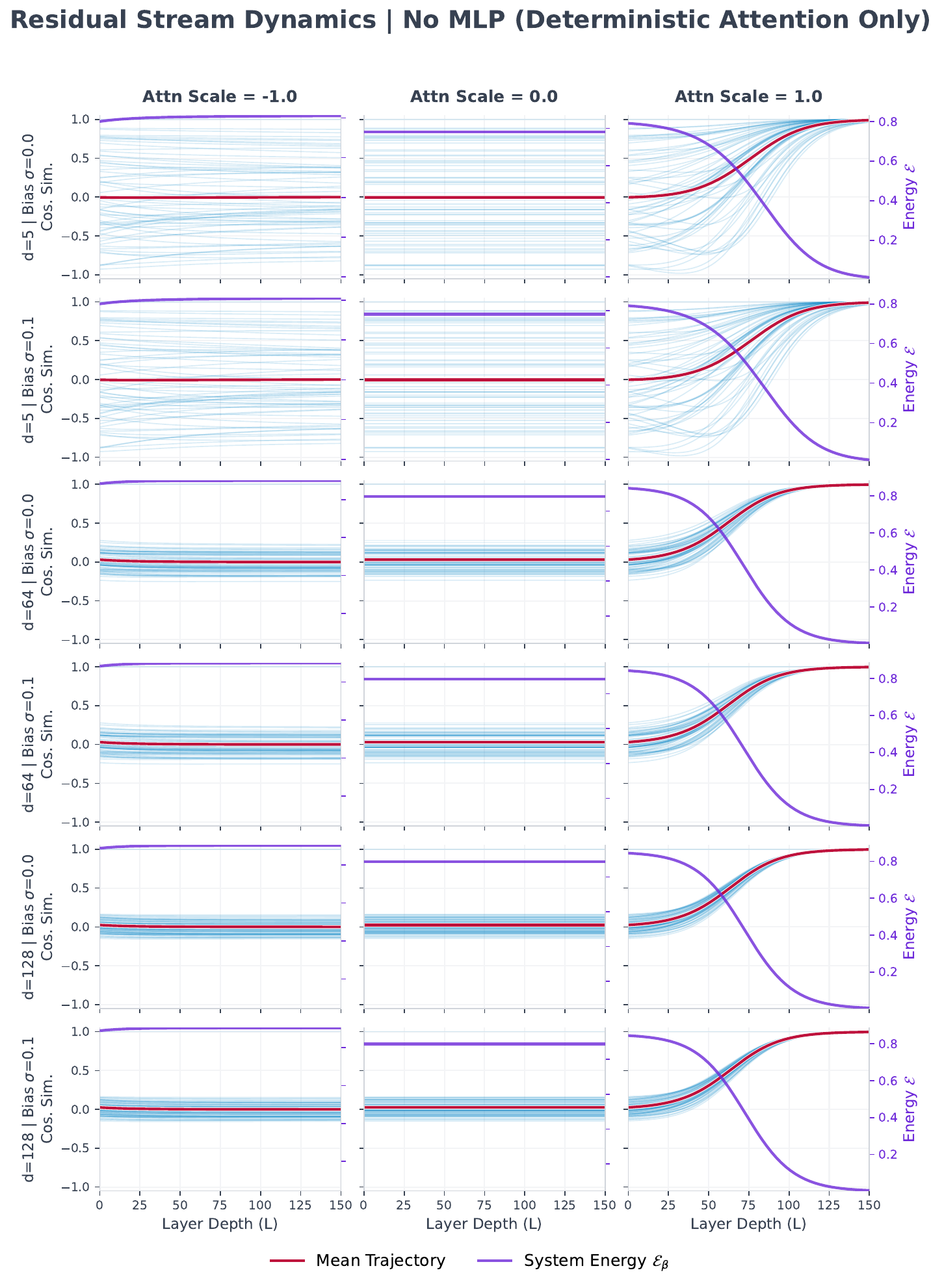}
    \caption{Numerical simulations of the discrete transformer dynamics without MLPs, for $N=64$ tokens over $L=150$ layers ($T=5.0$), dimensions $d\in\{5,64,128\}$, bias levels $\sigma\in\{0.0,0.1\}$, and attention scaling $1$, $0$, and $-1$. The panels show the pairwise cosine similarities and the interaction energy $\mathcal{E}_\beta(\mathbf X_\ell)$.}
    \label{fig:no_mlp}
\end{figure}

\clearpage

\end{document}